\documentclass[11pt]{article}
\usepackage{verbatim}
\usepackage{hyperref}
\usepackage{enumerate, amsthm}
\usepackage{bm} 
\usepackage{mathrsfs}
\usepackage{rotating}
\usepackage{cancel}
\usepackage{xcolor}
\usepackage[margin=1.1 in]{geometry}
\usepackage[margin={1.5cm, 1.5cm},font=small, labelsep=endash]{caption}
\usepackage{mathtools}
\usepackage{tikz}
\usetikzlibrary{patterns}
\usepackage{tikz-dimline}
\usepackage{ifthen}

\DeclareMathOperator*{\argmin}{arg\,min}

\newcommand{\snorm}[1]{\left\lvert#1\right\rvert}

\usepackage{mathptmx}

      \usepackage{amssymb}
      \usepackage{amsmath}
      
      \numberwithin{equation}{section}
      \theoremstyle{plain}
      \newtheorem{theorem}{Theorem}[section]
      \newtheorem{lemma}[theorem]{Lemma}
      
      \newtheorem{proposition}[theorem]{Proposition}
         \newtheorem{question}[theorem]{Question}
      \newtheorem{observation}[theorem]{Observation}
      \theoremstyle{definition}

      \theoremstyle{remark}
      \newtheorem{remark}[theorem]{Remark}

      \newcommand{\R}{{\mathbb R}}
      \newcommand{\E}{\mathbb E}
      \newcommand{\e}{\mathrm{e}}
      \newcommand{\ep}{\epsilon}
      \renewcommand{\P}{\mathbb P}
      
      \newcommand{\var}{\mathrm{Var}}
      \newcommand{\cov}{\mathrm{Cov}}
      
      \newcommand{\cross}{\textrm{cross}}
      
      \newcommand{\up}{\textrm{up}}
      \newcommand{\down}{\textrm{down}}
      \newcommand{\Z}{\mathbb{Z}}
       
      \newcommand{\N}{\mathbb{N}}
      \newcommand{\A}{\mathcal{A}}
      \newcommand{\B}{\mathcal{B}}
      \newcommand{\ball}{\mathscr{B}}

      \newcommand{\block}{\mathrm{Block}}

      \newcommand{\main}{\mathrm{main}}
      \newcommand{\gadg}{\mathrm{gadget}}

      \newcommand{\db}{\partial_{\mathrm{down}}}
      \newcommand{\ub}{\partial_{\mathrm{up}}}      
      \newcommand{\lb}{\partial_{\mathrm{left}}}      
      \newcommand{\rb}{\partial_{\mathrm{right}}}
      \newcommand{\loss}{\mathrm{Loss}}      
      \newcommand{\switch}{\mathrm{switch}} 
      
      \newcommand{\lt}{\mathrm{left}}
      \newcommand{\rt}{\mathrm{right}}

      \newcommand{\Approx}{\mathrm{approx}}
      \newcommand{\APPROX}{\mathrm{Approx}}
      \newcommand{\D}{\mathbb{D}}
      \renewcommand{\S}{\mathcal{S}}
      \newcommand{\Max}{\mathrm{Max}}
      \newcommand{\diam}{\mathrm{diam}}

      \makeatletter
      \def\@setcopyright{}
      \def\serieslogo@{}
      \makeatother
      
\begin{document}
\title{Upper bounds on Liouville first passage percolation \\ and Watabiki's
prediction}

\author{Jian Ding\thanks{Partially supported by NSF grant DMS-1455049, DMS-1757479 and an Alfred Sloan fellowship.} \\ University of Pennsylvania \and Subhajit Goswami\footnotemark[1]  \\
Institut des Hautes \'{E}tudes Scientifiques
}
\date{}

\maketitle
\begin{abstract}
Given a planar continuum Gaussian free field $h^{\mathcal U}$ in a domain $\mathcal U$ with Dirichlet boundary condition and any $\delta>0$, we let $\{h_\delta^{\mathcal U}(v): v\in \mathcal U\}$ be a real-valued smooth Gaussian process where $h_\delta^{\mathcal U}(v)$ is the average of $h^{\mathcal U}$ along a circle of radius $\delta$ with center $v$. For $\gamma > 0$, we study the Liouville first passage percolation (in scale $\delta$), i.e., the shortest path distance in $\mathcal U$ where the weight of each path $P$ is given by $\int_P \mathrm{e}^{\gamma h_\delta^{\mathcal U}(z)} |dz|$. We show that the distance between two typical points is $O(\delta^{c^* \gamma^{4/3}/\log \gamma^{-1}})$ for all sufficiently small but fixed $\gamma>0$ and some constant $c^* > 0$. In addition, we obtain similar upper bounds on the Liouville first passage percolation for discrete Gaussian free fields, as well as the Liouville graph distance which roughly speaking is the minimal number of Euclidean balls with comparable Liouville quantum gravity measure whose union contains a continuous path between two endpoints.  
Our results contradict with some reasonable interpretations of Watabiki's prediction (1993) on the random distance of Liouville quantum gravity at high temperatures.   

\end{abstract}

\section{Introduction}

In the seminal paper \cite{Kahane85}, the Gaussian multiplicative chaos was initiated and constructed as a random measure obtained from exponentiating log-correlated Gaussian fields. In the last decade, there has been extensive study on Gaussian multiplicative chaos as well as Liouville quantum gravity\footnote{We note that our convention on the terminology Liouville quantum gravity follows that in \cite{DS11}. This convention is a bit different from that adopted in Liouville field theory, and one shall be cautious about the underlying mathematical meaning of LQG when discussing in the context of Liouville field theory.}  which is an important special case of Gaussian multiplicative chaos where the underlying log-correlated field is a two-dimensional Gaussian free field. See, e.g.,  \cite{Kahane85, DS11, RV11, RV14, Shamov16, Berestycki17, RV17, APS18}. Despite much understanding on these random measures, the understanding on the random distances associated with Gaussian multiplicative chaos and Liouville quantum gravity remains elusive. In a well-known paper \cite{Watabiki93}, Watabiki made a physics prediction on the Hausdorff dimension for random distances associated with Liouville quantum gravity. The main goal of the current article is to present some bounds on the distance exponents, which as we will explain further in Section~\ref{sec:Watabiki}, seems to contradict all reasonable interpretations of Watabiki's prediction.

To present our results formally, we first introduce a number of definitions. Let $\mathcal U \subseteq \R^2$ be a bounded domain with smooth 
boundary. Let $d_{\ell_2}(S, S') = \inf_{v \in S, v' \in S'}|v - v'|$ be  the Euclidean distance 
between any two subsets $S$ and $S'$ of $\R^2$, and define $\mathcal U_\ep = \{v \in \mathcal U: d_{\ell_2}(v, \partial \mathcal U) > \ep\}$ for $\ep > 0$. In this paper we only consider domains $\mathcal U$ such that $V \coloneqq [0, 1]^2 \subseteq \mathcal U_\epsilon$ for some fixed 
$\epsilon$. Let $h^{\mathcal U}$ be a (continuum) Gaussian free field (GFF) on $\mathcal U$ with Dirichlet boundary condition. For an introduction to GFF including various formal constructions, see, e.g., \cite{S07, berestycki16}. 
Although it is not possible to make sense of $h^{\mathcal U}$ as a function on $\mathcal U$, it is regular enough so that we can interpret its Lebesgue integrals over 
sufficiently nice Borel sets in a rigorous way. In particular, for any $\delta>0$ and $v\in \mathcal U$ with $d_{\ell_2}(v, \partial \mathcal U) > \delta$,  we can take its average along the circle of radius $\delta$ around $v$, and thus obtain the \emph{circle average process} $\{h_\delta^{\mathcal U} (v): \delta>0, v \in \mathcal U, d_{\ell_2}(v, \partial \mathcal U) > \delta\}$ which is a centered Gaussian process with covariance  (below $v, v'\in \mathcal U$ and $d_{\ell_2}(v, \partial \mathcal U) > \delta>0, d_{\ell_2}(v', \partial \mathcal U) > \delta'>0$)
\begin{equation}
\label{eq:GFF_cov}
\cov(h_\delta^{\mathcal U}(v), h_{\delta'}^{\mathcal U}(v')) = \pi \int_{\partial B_\delta(v) \times \partial B_{\delta'}(v')}G_{\mathcal U}(z, z')\mu_\delta^v(dz)\mu_{\delta'}^{v'}(dz')\,,
\end{equation}
where the normalization factor $\pi$ is chosen so that the field is log-correlated, consistent with the convention in the majority of the literature.
Here $B_\delta(v)$ is the closed ball with radius $\delta$ centered at $v$, $\mu_\delta^v$ is the uniform probability measure on $\partial B_\delta(v)$ (the boundary of $B_\delta(v)$) --- an analogous interpretation for $B_{\delta'}(v')$ and $\mu_{\delta'}^{v'}$ applies. In addition, $G_{\mathcal U}(z, z')$ is the Green's function for domain 
$\mathcal U$, which is defined as
\begin{equation}
\label{eq:Green_fxn}
G_{\mathcal U}(z, z') = \int_{(0, \infty)}p_{\mathcal U}(s; z, z')ds \mbox{ for } z, z'\in \mathcal U,
\end{equation}
where $p_{\mathcal U}(s; z, z')$ is the transition density of two-dimensional Brownian motion 
killed upon exiting $\mathcal U$. More precisely, $p_{\mathcal U}(s; z, \cdot)$ is the unique (up to sets of Lebesgue measure 0) nonnegative measurable function satisfying
\begin{equation}
\label{eq:heat_kernel}
\int_{B} p_{\mathcal U}(s; z, z')dz' = P^z(W_s \in B, \tau_{\mathcal U} > s)\,
\end{equation}
for all Borel measurable subsets $B$ of $\R^2$. Here $P^z(\cdot)$ is the law of the two-dimensional standard Brownian motion $\{W_t\}_{t \geq 0}$ (see \eqref{eq:field_definition} for the formula of its transition density)  starting from $z$ and $\tau_{\mathcal U}$ is the exit time of $\{W_t\}_{t \geq 
0}$ from $\mathcal U$. It was shown in \cite{DS11} that there exists a version of the circle average process which is jointly H\"{o}lder continuous in $v$ and $\delta$ of order $\vartheta < 1/2$ on all compact subsets of $\{(v, \delta): v \in \mathcal U, 0 < \delta < d_{\ell_2}(v, 
\partial \mathcal U)\}$. Given such an instance of $h_\delta^{\mathcal U}$ and a fixed inverse-temperature parameter $\gamma > 0$, we define the \emph{Liouville first-passage percolation (Liouville FPP or LFPP) distance} $D_{\gamma, \delta}^{\mathcal U}(\cdot, \cdot)$ on $V$ by
\begin{equation}
\label{eq:Liouville_metric}
D_{\gamma, \delta}^{\mathcal U}(v, w) = \inf_{P} \int_{P}\e^{\gamma h_\delta^{\mathcal U}(z)}|dz| \coloneqq \inf_{P} \int_{[0, 1]}\e^{\gamma h_\delta^{\mathcal U}(P(t))}|P'(t)|dt\,,
\end{equation}
where $P: [0, 1] \to V$ ranges over all piecewise $C^1$ 
paths in $V$ connecting $v$ and $w$. The infimum is well-defined and measurable since we are dealing with a continuous process 
on a compact space. In fact $D_{\gamma, \delta}^{\mathcal U}(\cdot, \cdot)$ does not change if we restrict only to $C^1$ paths. 
\begin{theorem}
\label{thm:main}
Let $V \subseteq \mathcal U_\ep$ for some $\ep > 0$. Then there exist $\delta_{\gamma, \mathcal U, \epsilon} > 0$ (depending on $(\gamma, \mathcal U, \epsilon)$) and positive (small) absolute constants $c^*, \gamma_0$ such that for all $0<\gamma \leq \gamma_0$ and $0<\delta \leq \delta_{\gamma, \mathcal U, \epsilon}$, we have
$$\max_{v, w \in V}\E D_{\gamma, \delta}^{\mathcal U}(v, w) \leq  \delta^{c^*\frac{\gamma^{4/3}}{\log \gamma^{-1}}}\,.$$
\end{theorem}
Another related notion of random distance comes from the \emph{Liouville quantum gravity} (LQG) measure
$M_\gamma^{\mathcal U}$ on $\mathcal U$ (which, as described at the beginning of the introduction, is an important special case of Gaussian multiplicative chaos). For any $0<\gamma < 2$, $M_\gamma^{\mathcal U}$ is defined {in \cite{DS11}} as the almost sure weak limit of the sequence of measures $M_{\gamma, n}^{\mathcal U}$ given by
\begin{equation}\label{eq-limit-LQG}
M_{\gamma, n}^{\mathcal U} = \e^{\gamma h_{2^{-n}}^{\mathcal U}(z)}2^{-n\gamma^2/2}\sigma(dz),
\end{equation}
where $\sigma$ is the Lebesgue measure. Much on the LQG measure (and in general Gaussian multiplicative chaos) has been understood (see e.g., \cite{Kahane85, DS11, RV11, RV14, Shamov16, Berestycki17}) including the existence of the limit in \eqref{eq-limit-LQG}, the uniqueness in law for the limiting measure via different approximation schemes, as well as a KPZ correspondence
through a uniformization of the random lattice seen as a Riemann surface. Our focus in the present article is on the random distance
associated with the LQG. Given $\delta \in (0, 1)$, we say a closed Euclidean ball $B \subseteq \mathcal U$ is a $(M_\gamma^{\mathcal U}, \delta)$-ball if $M_\gamma^{\mathcal U}(B) \leq \delta^2$ and the center of $B$ is rational (to avoid unnecessary measurability 
consideration). We then define the \emph{Liouville graph distance} $\tilde D_{\gamma, \delta}^{\mathcal U}(v, w)$ between any two \emph{distinct} points $v, w \in V$ as the minimum number of $(M_\gamma^{\mathcal U}, \delta)$ 
balls whose union contains a path between $v$ and $w$. In addition, we set  $\tilde D_{\gamma, \delta}^{\mathcal 
U}(v, v) = 0$ for all $v \in V$. We name this distance as Liouville graph distance since it corresponds to the shortest path distance on a graph indexed by $\mathbb Q^2$ where adjacency relation corresponds to the intersection 
of $(M_\gamma^{\mathcal U}, \delta)$ balls. A very related graph distance was mentioned in \cite{MS15} which proposed to keep dividing each square until the LQG measure is 
below $\delta$. We chose our notion of Liouville graph distance for the reason that it seems to have more desirable invariant properties, although we expect our bound (as well as our proof) to extend to the other notion as well. 
\begin{theorem}
\label{thm:LQG}
Under the same conditions as in Theorem~\ref{thm:main}, there exist $\delta_{\gamma, \mathcal U, \epsilon} > 0$ (depending on $(\gamma, \mathcal U, \epsilon)$) and positive (small) absolute constants $c^*, \gamma_0$ such that for all $0<\gamma \leq \gamma_0$ and $0<\delta \leq \delta_{\gamma, \mathcal U, \epsilon}$, 
$$\max_{v, w \in V}\E \tilde D_{\gamma, \delta}^{\mathcal U}(v, w) \leq \delta^{-1 + c^*\frac{\gamma^{4/3}}{\log \gamma^{-1}}}\,.$$
\end{theorem}

Finally, it is as natural to consider Liouville FPP for discrete GFF (which was explicitly mentioned in \cite{Benjamini}).
Given a two-dimensional box $V_N \subseteq \Z^2$ of side length $N$, the discrete GFF $\{\eta_{N, v}: v\in V_N\}$ with Dirichlet
boundary condition is a mean-zero
Gaussian process such that 
$$\eta_{N, v}=0
\mbox{ for all } v\in \partial V_N\,, \mbox{ and }\E \eta_{N, v} \eta_{N, w}=G_{V_N}(v,w)
\mbox{ for all } v, w\in V_N\,.$$ Here $\partial V_N = \{v  \in V_N: v\mbox{ has a neighbor in } \Z^2 \setminus V_N\}$ and $G_{V_N}(v, w)$ is the Green's function of simple random walk on $V_N$, i.e., the expected number of visits to $w$ for the simple random walk on $\Z^2$ started 
at $v$ and killed upon reaching $\partial V_N$. As before, for a fixed inverse-temperature parameter $\gamma > 0$, we define the Liouville FPP distance $D_{\gamma, N}(\cdot, \cdot)$ on $V_N$ by
\begin{equation}\label{eq-LFPP-def}
D_{\gamma, N}(v, w)=\mathbf 1_{\{v \neq w\}}\min_{P}\sum_{u \in P}e^{\gamma \eta_{N, u}} \mbox{ for  } v, w\in V_N\,,
\end{equation}
where $P$ ranges over all paths in $V_N$ connecting 
$v$ and $w$ and $\mathbf 1_{\{.\}}$ is the indicator 
function. As we explain in Section~\ref{sec-discrete}, the proof of Theorem~\ref{thm:main} can be adapted to derive the following result.
\begin{theorem}\label{thm-discrete}
Given any fixed $0 < \epsilon < 1 / 2$, there exists $N_{\gamma, \epsilon} \in \N$ 
(depending on $(\gamma, \epsilon)$) 
and positive (small) absolute constant $c^*, \gamma_0$ such that for all $0<\gamma \leq \gamma_0$ and $N \geq N_{\gamma, \epsilon}$, we have 
$$ 
\max_{v, w\in V_{N, \epsilon}}\E D_{\gamma, N}(v, w) \leq N^{1 - c^* \frac{\gamma^{4/3}}{\log \gamma^{-1}}},$$
where $V_{N, \epsilon}$ is the square
$\{v \in V_N: d_\infty(v, \partial V_N) \geq \epsilon N \}$ and $d_\infty((u_1, u_2), (v_1, v_2)) = \max(|u_1 - v_1|, |u_2 - v_2|)$.
\end{theorem}
\begin{remark}
\label{remark:main_theorem}
Theorem~\ref{thm-discrete} still holds if we 
restrict $P$ to be a path within $V_{N, \epsilon}$ in \eqref{eq-LFPP-def}. 
\end{remark}
\subsection{Discussion on Watabiki's prediction}\label{sec:Watabiki}

The Liouville FPP and the Liouville graph distance are two (related) natural discrete approximations for the random distance associated with the Liouville quantum
gravity  \cite{P81,DS11,RV14}. Precise predictions on various exponents regarding to LQG distance  have been made by Watabiki \cite{Watabiki93} (see also, \cite{ANRBW98, AB14}). In particular, the Hausdorff dimension for the LQG distance is predicted to be
\begin{equation}\label{eq-Watabiki}
d_H(\gamma) = 1+\frac{\gamma^2}{4} + \sqrt{(1+ \frac{\gamma^2}{4})^2+ \gamma^2}\,.
\end{equation}
The prediction in \eqref{eq-Watabiki} was widely believed. In a recent work \cite{MS15}, Miller and Sheffield introduced and studied a process called \emph{quantum Loewner evolution}.  As a \emph{byproduct} of their work they gave a non-rigorous analysis on exponents of the LQG distance which matched Watabiki's prediction --- we also note that in \cite{MS15} the authors did express some reservations on their non-rigorous analysis. For other discussions on Watabiki's prediction in mathematical literature, see e.g., \cite{MRVZ14,GHS16}.

The precise mathematical interpretation of Watabiki's prediction is not completely clear to us. However, there are a number of reasonable ``folklore'' interpretations that seem to be widely accepted. For the Liouville graph distance, the scaling exponent $\chi = - \lim_{\delta\to 0}\frac{\E \log \tilde D_{\gamma, \delta}(v, w)}{\log \delta}$ is expected to exist and is expected to be given by (here we take $v, w$ as two fixed generic points in the domain)
\begin{equation}\label{eq-Watabiki-LGD}
\chi  = \frac{2}{d_H(\gamma)} \geq  1 - C \gamma^2\,, \mbox{ as } \gamma\to 0 \mbox{ for an absolute constant } C>0\,,
\end{equation}
where in the last step we plugged in \eqref{eq-Watabiki}. A similar interpretation to \eqref{eq-Watabiki-LGD} appeared in \cite[Conjecture 1.14]{GHS16} though the graph structure considered in \cite{GHS16} is based on the peanosphere construction of LQG and so far we see no mathematical connection to Liouville graph distance considered in the present article. Note that there is a difference of factor of 2, which is due to the fact that for the graph defined in \cite{GHS16} on average each ball contains LQG measure about $\epsilon$ (in their notation) while in our construction each ball contains LQG measure $\delta^2$. We see that Theorem~\ref{thm:LQG} contradicts \eqref{eq-Watabiki-LGD}.

There are also reasonable interpretations of Watabiki's prediction for Liouville FPP.\footnote{For instance, we learned from R{\'e}mi Rhodes and Vincent Vargas that, according to \cite{Watabiki93}, the physically appropriate approximation for the $\gamma$-LQG distance should involve $\inf_{P} \int_{P}\e^{\frac{\gamma}{d_H(\gamma)} h_\delta(z)}|dz|$ , i.e., the parameter in the exponential of GFF is $\gamma/d_H(\gamma)$ instead of $\gamma$.} For instance, see \cite[Equation (17), (18)]{AB14}. We would like to point out that in \cite[Equation (17)]{AB14} the term $\rho_\delta$ was not defined --- some reasonable interpretations include $\rho_\delta = \e^{\frac{\gamma}{2} h_\delta(z)}$ and $\rho_\delta = \e^{\gamma h_\delta(z)} \delta^{\frac{\gamma^2}{2}}$ as well as possibly replacing $\gamma$ by $\frac{\gamma}{d_H(\gamma)}$ as suggested in the footnote. 
For all these interpretations, \cite[Equation (18)]{AB14} would then imply that there exist constants $c, C>0$ such that for sufficiently small but fixed $\gamma>0$ the Liouville FPP distance between two generic points is between $\delta^{C\gamma^2}$ and $\delta^{c\gamma^2}$ as $\delta \to 0$. However, Theorem~\ref{thm:main} contradicts with all aforementioned interpretations of \eqref{eq-Watabiki} for Liouville FPP at high temperatures.

An extremely meticulous reader may be doubtful whether Liouville FPP and Liouville graph distance are reasonable to approximate the ``true'' LQG distance. To aim for a clarification, we would like to remark that in \cite{DZZ17} the authors considered a certain type of log-correlated Gaussian field which they refer to as coarse modified branching random walk. In \cite{DZZ17}, 
the exponent on the heat kernel for Liouville Brownian motion (with respect to the coarse modified branching random walk) was computed via considering Liouville FPP (naturally with coarse branching random walk as the underlying field). This may serve as a piece of evidence that the Liouville FPP is indeed a reasonable approximation, as the heat kernel for Liouville Brownian motion is generally considered to encode the distance properties in the limit.

As a final remark, we note that currently we do not have any reasonable conjecture on the precise value of the exponent for either Liouville FPP or Liouville graph distance --- we regard a precise computation of the exponent for either of the two distances as a major challenge.

\subsection{Discussion on non-universality}

Combined with \cite{DZ15}, Theorem~\ref{thm-discrete} shows that the distance exponent for first passage percolation on the exponential of log-correlated Gaussian fields is non-universal, i.e., the exponents may differ for different families of log-correlated Gaussian fields. In contrast, we note that the behavior for the  maximum is universal among log-correlated Gaussian fields (see e.g., \cite{BZ10, BDZ14,Madaule15, DRZ15}) in a sense that their expectations are the same up to additive $O(1)$ term and that the laws of the centered maxima for all these fields are in the same universal family known as Gumbel distribution with random shifts (but the random shifts may not have the same law for different fields). 

While non-universality suggests subtlety for the distance exponent of Liouville FPP, the proof in the present article does not see complication due to such subtlety. In fact, our proof should be adaptable to general log-correlated Gaussian fields with $\star$-scale invariant kernels as in \cite{DRSV14}.  The following question remains an interesting challenge, especially (in light of the non-universality) for log-correlated Gaussian fields for which a kernel representation is not known to exist. 
\begin{question}
Let $\{\varphi_{N, v}: v\in V_N\}$ be an arbitrary mean-zero Gaussian process satisfying $|\E \varphi_{N, v} \varphi_{N, u} -  \log \frac{N}{1 +  |u-v|}| \leq K$. Does an analogue of Theorem~\ref{thm-discrete} hold for $C_\gamma, c^*$ depending on $K$?
\end{question}
\subsection{Further related works}

Much effort has been devoted to understanding classical first-passage
percolation (FPP), with independent and identically distributed edge/vertex
weights. We refer the reader to \cite{ADH15,GK12} and their references
for reviews of the literature on this subject. We argue that FPP with
strongly-correlated weights is also a rich and interesting subject,
involving questions both analogous to and divergent from those asked
in the classical case. Since the Gaussian free field is in some sense
the canonical strongly-correlated random medium, we see strong motivation
to study Liouville FPP.

More specifically, as mentioned earlier Liouville FPP and Liouville graph distance play  key roles in
understanding the random distance associated with the Liouville quantum
gravity (LQG). We remark that the random distance of LQG is a major open problem, even just to make rigorous sense of it (we refer to \cite{RV16} for a rather up-to-date review). In a recent series of works of Miller and Sheffield, much understanding has been obtained for the LQG distance (in the special case when $\gamma = \sqrt{8/3}$), and we note that an essentially equivalent distance to Liouville graph distance was mentioned in \cite{MS15} as a natural approximation. While no mathematical result was obtained (perhaps not attempted either) on such approximations, the main achievement of this series of works by Miller and Sheffield (see \cite{MS15, MS15b,  MS16, MS16b} and references therein) is to produce candidate scaling limits and to establish a deep connection to the Brownian map.   Our approach is different, in the sense that we aim to understand the random distance of LQG via approximations by natural discrete distances. We also note that in a recent work \cite{GHS16} some upper and lower bounds have been obtained for a type of distance related to LQG and that their bounds are consistent with Watabiki's prediction. We further remark that currently we see no connection between our work and \cite{MS15, MS15b, MS16, MS16b, GHS16}.

Furthermore, we expect that Liouville FPP distance and Liouville graph distance are related to the heat kernel estimate for Liouville Brownian motion (LBM) --- in fact, we expect a direct and strong connection between  Liouville graph distance and the LBM heat kernel.
 The mathematical construction (of the diffusion) for LBM was provided in \cite{GRV13, B14} and the heat kernel was constructed in \cite{GRV14}. The LBM is closely related to the geometry of LQG; in \cite{FM09, BGRV14} the KPZ formula was derived from Liouville heat kernel. In \cite{MRVZ14} some nontrivial bounds for LBM heat kernel were established. A very interesting direction is to compute the heat kernel of LBM with high precision.

There has been a number of other recent works on Liouville FPP (while they focus on the case for the discrete GFF, these results are expected to extend to the case of continuum GFF).
In a recent work \cite{DD16}, it was shown that at high temperatures the appropriately normalized Liouville FPP converges subsequentially in the Gromov-Hausdorff sense to a random distance on the unit square, where all the (conjecturally unique) limiting distances are homeomorphic to the Euclidean distance. We remark that the proof method in the current paper bears little similarity to that in \cite{DD16}. In a very recent work \cite{DZ16}, it was shown that the dimension of the geodesic for Liouville FPP is strictly larger than 1, again for small but fixed $\gamma>0$. In fact, in \cite{DZ16} it proved that the distance exponent is bounded from below by a number arbitrarily close to 1 if we restrict to paths with dimensions sufficiently close to 1. This, combined with Theorem~\ref{thm-discrete} yields the lower bound on the dimension of the geodesic. While both the proofs in \cite{DZ16} and the present article use multi-scale analysis method, the details are drastically different. 

Finally, we would like to mention that there has been a few other recent works on the distance properties of two-dimensional discrete GFF, including \cite{LW16} on the random pseudo-metric on a graph defined via the zero-set of the Gaussian free field on its metric graph,  \cite{DL16} on the chemical distance on the level sets of GFF,
and \cite{BDG16} on the effective resistance distance on the random network where each edge $(u, v)$ is assigned a resistance $e^{\gamma(\eta_{N, u} + \eta_{N, v})}$. In particular, the work of \cite{BDG16} implies that the typical effective resistance between two vertices of Euclidean distance $N$ behaves like $N^{o(1)}$. Combined with Theorem~\ref{thm-discrete}, this implies that (somewhat mysteriously)  perturbing $\mathbb Z^2$ by assigning weights which are exponentials of GFF drastically distorts the shortest path distance but more or less preserves the effective resistance of $\mathbb Z^2$.

\subsection{Organization}
We now give a brief guide on the organization. In Section~\ref{sec-history}, we present our proof strategy via a toy problem which captures the core computation 
leading to the exponent $4/3$ in our main results. In Section~\ref{sec-preliminary}, we introduce a new Gaussian process (referred to simply as the ``new process'' hereafter in the current subsection) which has a simpler hierarchical structure than the circle average process. We show that the two processes can be coupled in a way such that their maximal pointwise difference is small, allowing us to carry out our main proof in terms of 
the new process. In Section~\ref{sec-construction} we describe the geometric framework for our 
inductive construction of ``light'' paths as scales increase. ``Light'' means that these paths have small weight which are computed as integrals of the exponential of the new process along the paths.  We then complete the construction in Section~\ref{sec-analysis} by choosing the optimization strategy (with respect to the new process) and thus obtain a bound on the corresponding 
distance exponent. This bound along with the coupling between the two processes from Section~\ref{sec-preliminary} leads to a proof of Theorem~\ref{thm:main} at the end of Section~\ref{sec-analysis}. In Section~\ref{sec:LQG}, we prove Theorem~\ref{thm:LQG} by relating the Liouville graph distance to Liouville FPP and applying 
Theorem~\ref{thm:main}. Finally, in Section~\ref{sec-discrete} we explain how to adapt the proof to deduce Theorem~\ref{thm-discrete}.

\medskip

\noindent {\bf Acknowledgement.} We thank Marek Biskup, Hugo Duminil-Copin, Alexander Dunlap, Steve Lalley,  Elchanan Mossel, R{\'e}mi Rhodes, Scott Sheffield, Vincent Vargas and Ofer Zeitouni for helpful discussions. We especially thank Alexander Dunlap for a careful reading of a preliminary draft and numerous useful comments. An earlier version of the manuscript was completed when both authors were in the University of Chicago.

\section{An overview of our proof strategy}\label{sec-history}
The key technical contribution of the present article is Theorem~\ref{thm:main}, provided with which Theorems~\ref{thm:LQG} and \ref{thm-discrete} follow more or less in an expected way.
Our proof strategy of Theorem~\ref{thm:main} naturally inherits that of \cite{DG15} which proved a weak version of Theorem~\ref{thm-discrete} in the context of branching random walk (BRW), and we encourage the reader to flip through \cite{DG15}  (in particular Section 1.2) which contains a prototype of the multi-scale analysis carried out in the current paper. In fact, prior to the present article, we posted an article \cite{DG16} on arXiv which proved that the distance exponent is less than $1 - \gamma^2/10^3$. Our present article proves a stronger result than \cite{DG16}. In addition, our proof simplifies that of \cite{DG16} and is self-contained. As a result,  \cite{DG16} will be superseded by the present article and will not be published anywhere. 

However, some historical remarks might be interesting and helpful. During the work of \cite{DG16}, we had in mind that the second leading term for the distance exponent  is 
of order $\gamma^2$ in light of \eqref{eq-Watabiki}. As a result, we followed \cite{DG15} and designed an inductive strategy for constructing light crossings (i.e., paths connecting two shorter boundaries of a rectangle) to prove an upper bound of $1 
- \gamma^2/10^3$. In the multi-scale construction, the order of $\gamma^2$ is exactly the order of both the gain and the loss for our strategy, and thus a much delicate analysis was carried out in \cite{DG16} since we fought 
between the two constants for the loss and the gain. A curious reader may quickly flip through \cite{DG16} for an impression on the level of technicality.

A key component in both \cite{DG15,DG16} is an inductive construction where we construct light crossings in a 
bigger scale from crossings in smaller scales. We switch between two layers of candidate crossings in a smaller scale based on the value of Gaussian variables in the bigger scale (note that there is a hierarchical 
structure for both BRW and GFF). In those papers, we used vertical crossings as our switching gadgets to connect 
horizontal crossings in top and bottom layers. A crucial improvement in the current article arises from the simple observation that a \emph{sloped} switching gadget is much 
more efficient (see Figure~\ref{fig:sloped_gadget}). In order to give a flavor of how it works we discuss a 
toy problem in this section.

We will use some order notations to avoid unnecessary named constants and these will carry the same 
meaning throughout the paper. For (nonnegative) functions $F(.)$ and $G(.)$ we write $F = O(G)$ (or $\Omega(G)$) if there exists an absolute constant $C > 0$ such that $F \leq C G$ (respectively $\geq C G$) everywhere in their domain. We write $F = \Theta(G)$ if $F$ is both $O(G)$ and 
$\Omega(G)$. If the constant depends on variables $x_1, x_2, \ldots, x_n$, we change these notations to $O_{x_1,x_2, \ldots, x_n}(G)$ and $\Omega_{x_1,x_2, \ldots, x_n}(G)$ respectively.

Let $\Gamma = \Gamma(\gamma)$ be a large positive number (we assume throughout that $0 < \gamma \ll 1$) and $\{\zeta(v): v \in \R^2\}$ be a continuous, centered 
Gaussian process on the plane. Suppose that $\zeta$ satisfies the following three properties:\\
(a) $\var (\zeta(v)) = 1$ for all $v \in \R^2$.\\
(b) For any straight line segment $L$, $\var(\int_{L} \zeta(z)|dz|) = O(\|L\|)$, where $\|L\|$ is the (Euclidean) length 
of $L$. \\
(c) If $v \in \R^2$ is orthogonal to $L$ such that $\snorm{v} \geq 0.1$ (say) where $\snorm{.}$ denotes the $\ell_2$ norm, then
$$\var\Big(\int_{L} \zeta(z)|dz| - \int_{L + v} \zeta(z)|dz|\Big) = \Omega(\|L\|)\,.$$
We first want to construct a piecewise smooth path $P$ connecting the shorter boundaries of the rectangle $V^\Gamma = [0, \Gamma] \times [0, 1]$ such that the ``weight'' of $P$, given by $\int_{P}\e^{\gamma\zeta(z)}|dz|$, has a small 
expectation. Since $\e^x = 1 + x + O(x^2)$ as $x \to 0$ and $\var(\zeta(z)) = 1$ by Property~(a), we can approximate $\e^{\gamma \zeta(z)}$ with $1 + \gamma \zeta(z) +  O(\gamma^2)$ when $\gamma$ is sufficiently small (a justification in the proof is given in Section~\ref{sec-analysis}). Thus the weight of $P$ is approximately
\begin{equation}
\label{eq:approx_randomln}
(1 + O(\gamma^2))\|P\| + \gamma\int_P \zeta(z) |dz|\,.
\end{equation}
Henceforth we will treat the above expression as the 
``true'' weight of $P$. Now consider the segments $L_1 = [0, \Gamma] \times \{0.75\}$ and $L_2 = [0, \Gamma] \times \{0.25\}$. Choose $\beta$ such that $\Gamma \gg \beta \gg 1$ (later we will choose $\beta = \gamma^{-2/3}$, and we will see that this is an optimal choice) and that $\Gamma/\beta$ is an integer.  Divide $L_i$ (here $i \in \{1, 2\}$) into segments $L_{i, 1}, L_{i, 2}, \ldots, L_{i, \Gamma / \beta}$ 
of length $\beta$ from left to right. Given $i_j \in \{1, 2\}$ for each $j \in \{1, \ldots, \Gamma / \beta\}$ (called a \emph{strategy}), we can construct a crossing, i.e., a path connecting the shorter boundaries of $V^\Gamma$ as follows. If $i_j = i_{j + 1}$ for some $j \leq \Gamma / \beta$ (we use the convention $i_{\Gamma/\beta + 1} = i_{\Gamma/\beta}$), let $L_{j, \mathrm{chosen}}$ be the segment $L_{i_j, j}$. Otherwise set $L_{j, \mathrm{chosen}}$ to be the segment joining the left endpoints of $L_{i_j, j}$ and 
$L_{i_{j+1}, j+1}$ (the sloped gadget). Notice that there can be only two possible sloped gadgets at a ``location'' $j$ which we will refer to as $L_{j, \mathrm{slope}}^1$ and $L_{j, \mathrm{slope}}^2$ in a certain arbitrary order. It is clear that the union of the segments $L_{1, \mathrm{chosen}}, L_{2, \mathrm{chosen}}, \ldots, L_{\Gamma / \beta, \mathrm{chosen}}$ forms a crossing. The weight (see \eqref{eq:approx_randomln}) of this crossing is given by
\begin{equation}
\label{eq:toy_prob}
\big(1 + O(\gamma^2)\big)\Gamma + \big(1 + O(\gamma^2)\big)\sum_{j \in [\Gamma / \beta]}\mathbf 1_{\{i_j \neq i_{j + 1}\}}(\|L_{j, \mathrm{chosen}}\| - \beta) + \gamma \sum_{j \in [\Gamma / \beta]}\int_{L_{j, \mathrm{chosen}}}\zeta(z)|dz|\,
\end{equation}
{where $[n] = \{1, \ldots, n\}$ for any positive 
integer $n$.} {We need to bound its expectation from above. However, it would be more convenient to analyze an upper bound that involves $\E(\int_{L_{i_j,  j}}\zeta(z)|dz|)$'s instead of $\E(\int_{L_{j, \mathrm{chosen}}}\zeta(z)|dz|)$'s (the advantage will become clear 
once we describe our strategy). To this end we bound $\sup_{i_j \in [2], i_{j + 1} \neq i_j}\big \vert\int_{L_{j, \mathrm{chosen}}}\zeta(z)|dz| - \int_{L_{i_j, j}}\zeta(z)|dz|\big\vert$ simply by
\begin{equation*}
\label{eq:toy_prob1}
\sum_{i \in [2]}\big \vert \int_{L_{i, j}}\zeta(z)|dz| \big \vert + \sum_{i \in [2]}\big \vert \int_{L_{j, \mathrm{slope}}^i}\zeta(z)|dz| \big \vert\,.
\end{equation*}
By Property~(b), each of the integrals above is a centered Gaussian variable with variance $O(\|L_{j, \mathrm{slope}}^i\|) = O(\sqrt{(0.75 - 0.25)^2 + \beta^2}) = O(\beta)$ 
(since $\beta \gg 1$). Hence $\E \big \vert \int_{L_{i, j}}\zeta(z)|dz| \big \vert$ and $\E \big \vert \int_{L_{j, \mathrm{slope}}^i}\zeta(z)|dz| \big \vert$ are both 
$O(\sqrt{\beta})$.} Now let $J = \{j_1, j_2, \ldots, j_{|J|}\} \subseteq [\Gamma / \beta]$ with $j_1 < j_2 < \ldots < j_{|J|} =  \Gamma / \beta$ where $|J|$ is the cardinality of $J$. Consider an arbitrary strategy {$\{i_j\}_{j \in [\Gamma / \beta]}$ (which can be random) allowing} $i_j \neq i_{j + 1}$ only on $J\setminus \{j_{|J|}\}$. {We then see that the expected value of \eqref{eq:toy_prob} is bounded above by
\begin{equation*}
\label{eq:toy_prob1*}
\big(1 + O(\gamma^2)\big)\Gamma + \big(1 + O(\gamma^2)\big)|J|(\|L_{1, \mathrm{slope}}^1\| - \beta) + \gamma |J|C\sqrt{\beta} + \gamma \sum_{j \in [\Gamma / \beta]}\E \int_{L_{i_j, j}}\zeta(z)|dz|\,,
\end{equation*}
where $C$ is a positive constant. Notice that we used the fact that the $L_{j, \mathrm{slope}}^i$'s have the same length for all $j$.} Also since $\beta \gg 1$, we have $\|L_{1, \mathrm{slope}}^1\| - \beta = {\sqrt{(0.75 - 0.25)^2 + \beta^2} - \beta} = 
O(\beta^{-1})$ {and thus the last display can be further bounded by}
\begin{equation}
\label{eq:toy_prob2}
\big(1 + O(\gamma^2)\big)\Gamma + |J|C'\beta^{-1} + \gamma |J|C\sqrt{\beta} + \gamma \sum_{k \in [|J|]}\E \big( \int_{{\mathfrak L_{{i_{j_k}}, k, J}}}\zeta(z)|dz|\big)\,.
\end{equation}
Here $C'$ is a positive constant {(recall that $\gamma$ is small and hence $1 + O(\gamma^2) = O(1)$)} and {$\mathfrak L_{i, k, J}$ is the union of segments $L_{i, j_{k -1} + 1}, L_{i, j_{k -1} + 2}, \ldots, L_{i, j_k}$ with $j_0 = 0$ (i.e., the horizontal segment joining the left endpoint of $L_{i, j_{k-1} + 1}$ and the right endpoint of 
$L_{i, j_k}$) for $i\in \{1, 2\}$. Thus, $\mathfrak L_{{i_{j_k}}, k, J}$ is interpreted as $\mathfrak L_{1, k, J}$ or $\mathfrak L_{2, k, J}$ depending on whether $i_{j_k} = 1$ or $2$.
 For any given $\beta$ and $J$, it is clear that the minimizing strategy for \eqref{eq:toy_prob2} is the following:
\begin{equation*}
\label{eq:toy_prob2*}
i_{j_k} = \argmin_i \int_{{\mathfrak L_{i, k, J}}}\zeta(z)|dz| \mbox{ for all }k \in [|J|]\,.
\end{equation*}
Since $\min(X, Y) = \tfrac{X + Y}{2} - \tfrac{|X -Y|}{2}$ and $\int_{{\mathfrak L_{i, k, J}}}\zeta(z)|dz|$'s are centered variables, it follows that
\begin{equation*}
\label{eq:toy_prob3}
\E \big( \int_{\mathfrak L_{i_{{j_k}}, k, J}}\zeta(z)|dz|\big) = -\frac{1}{2}\E\bigg \vert \int_{\mathfrak L_{1, k, J}}\zeta(z)|dz| - \int_{\mathfrak L_{2, k, J}}\zeta(z)|dz|\bigg\vert\,.
\end{equation*}
In addition, we see from Property~(c) that $ \int_{\mathfrak L_{1, k, J}}\zeta(z)|dz| - \int_{\mathfrak L_{2, k, J}}\zeta(z)|dz|$ is a centered Gaussian variable 
with variance $\Omega((j_k - j_{k - 1})\beta)$. Consequently
\begin{equation*}
\label{eq:toy_prob3*}
\E \big( \int_{\mathfrak L_{i_{{j_k}}, k, J}}\zeta(z)|dz|\big) = -\frac{1}{2}\E\bigg \vert \int_{\mathfrak L_{1, k, J}}\zeta(z)|dz| - \int_{\mathfrak L_{2, k, J}}\zeta(z)|dz|\bigg\vert \leq -c\sqrt{(j_k - j_{k-1})\beta}\,,
\end{equation*}
for some positive constant $c$. Thus if $\min_{k \in [|J|]}(j_k - j_{k-1}) \geq G$, \eqref{eq:toy_prob2} can be bounded from above by
\begin{equation}
\label{eq:toy_prob4*}
\big(1 + O(\gamma^2)\big)\Gamma - |J|\big(c\gamma \sqrt{G}\sqrt{\beta} - C'\beta^{-1} - \gamma C\sqrt{\beta}\big)\,.
\end{equation}
Let us choose $G$ to be only large enough so that 
\begin{equation}
\label{eq:toy_prob4}
c\gamma\sqrt{G}\sqrt{\beta} \geq 2(C'\beta^{-1} + \gamma C\sqrt{\beta})\,.
\end{equation}
In order for there to be a valid choice for $G$, it must be that $\Gamma \geq G\beta$. A straightforward algebra shows us that this is true if
\begin{equation}
\label{eq:toy_prob5}
\Gamma \geq \frac{8C'^2}{c^2\gamma^2 \beta^2} + \frac{8C^2\beta}{c^2}\,.
\end{equation}
We will choose $\Gamma$ according to this condition at the end. For the moment being let us just assume \eqref{eq:toy_prob5} holds. For such a choice of $G$, there exists $J$ such that $\min_{k \in [|J|]}(j_k - j_{k-1}) \geq G$ and 
$$|J| = \Omega\Big(\frac{\Gamma}{G\beta}\Big) = \frac{\Omega(\Gamma)}{\frac{4(C'\beta^{-1} + \gamma C\sqrt{\beta})^2}{c^2\gamma^2}} = \frac{\Omega(\Gamma \gamma^2)}{(\beta^{-1} + \gamma \sqrt{\beta})^2}\,.$$
Plugging this into \eqref{eq:toy_prob4*} and using \eqref{eq:toy_prob4}, we find that \eqref{eq:toy_prob4*} can be at most
\begin{eqnarray*}
&& \big(1 + O(\gamma^2)\big)\Gamma - |J|(C'\beta^{-1} + \gamma C\sqrt{\beta})\\
&=& \big(1 + O(\gamma^2)\big)\Gamma - \frac{\Omega(\Gamma \gamma^2)}{(\beta^{-1} + \gamma \sqrt{\beta})^2}(C'\beta^{-1} + \gamma C\sqrt{\beta})\\
&=& \big(1 + O(\gamma^2)\big)\Gamma - \frac{\Omega(\Gamma \gamma^2)}{\beta^{-1} + \gamma\sqrt{\beta}}\,.
\end{eqnarray*}
}
The above expression is minimized for $\beta = \gamma^{-2/3}$ and the optimal value is $\Gamma(1 - 
\Omega(\gamma^{4/3}))$ {since} $\gamma$ is small --- by \eqref{eq:toy_prob5} we see that $\Gamma$ only needs to be a large (but fixed) multiple of $\gamma^{-2/3}$ for this bound to be valid although our choice of $\Gamma$ will be 
more restrictive in the real proof. These computations, which lead to the factor of $(1 - 
\Omega(\gamma^{4/3}))$,  give us the first hint on the appearance of $\gamma^{4/3}$ in the exponent as in Theorem~\ref{thm:main}.

However the circle average process $h_{2^{-n}}^{\mathcal U}$ is not a smooth process like $\zeta$. Rather it is very close to a process like $X_n = \zeta_1(\cdot) + \zeta_2(2\cdot) + \ldots + \zeta_n(2^n \cdot)$ where $\zeta_1, \ldots, \zeta_n$ are independent copies of 
$\zeta$ (see Section~\ref{sec-history}). We can take advantage of this decomposition to construct a light crossing through $V^\Gamma$ with respect to $X_n$ by 
applying our construction at many scales $k \in [n]$. To elucidate it further let us ``thicken'' the segments $L_{i,j}$ and $L_{j, \mathrm{slope}}^i$ to rectangular 
strips of width $\Gamma^{-2}$. We split these strips into even smaller rectangles of length $\Gamma^{-1}$ so 
that their aspect ratio is the same as that of $V^\Gamma$ and then use the field $\zeta_{2\log_2\Gamma + 1}(\Gamma^{2}\cdot) + \ldots + \zeta_n(2^n\cdot)$ to construct ``efficient'' crossings through each one of them. Since their width is very small, the collection of crossings through adjacent rectangles in a strip roughly ``look like'' a line segment 
joining its shorter boundaries. Consequently we can perform our switching steps with respect to $\zeta_1$ using these collections of crossings instead of 
$L_{i,j}$'s and $L_{j, \mathrm{slope}}^i$'s. Due to independence of $\zeta_k$'s, the only way this will affect the computation of the expected weight is by changing the unit of length to $\Gamma d_{n - 2\log_2 \Gamma; n}$ where $d_{n - 2\log_2 \Gamma; n}$ is the maximum expected weight of crossings through smaller rectangles for the field $\zeta_{2\log_2\Gamma + 1}(\Gamma^{2}\cdot) + \ldots + 
\zeta_n(2^n\cdot)$ (recall that the length of these rectangles is $\Gamma^{-1}$). Notice that the latter is identically distributed as $X_{n - 2\log_2 
\Gamma}(\Gamma^2 \cdot)$ and thus $\Gamma^2 d_{n - 2\log_2 \Gamma; n}$ is the maximum expected length of an ``efficient'' crossing through a rectangle obtained by 
translating and / or rotating $V^\Gamma$ for the field $X_{n - 2\log_2 \Gamma}$. Denoting this quantity as $d_{n - 2\log_2 \Gamma}$ we then get the following recursive relation for $d_n$'s:
$$d_{n} \leq (\E \e^{\gamma Z})^{2\log_2\Gamma - 1}\Gamma^2 d_{n - 2\log_2 \Gamma; n}(1 - \Omega(\gamma^{4/3})) = (1 + \gamma^2/2)^{2\log_2\Gamma - 1}d_{n - 2\log_2 \Gamma}(1 - \Omega(\gamma^{4/3}))\,,$$
where $Z$ is a standard Gaussian variable and the first factor accounts for the effects of the fields $\zeta_2, 
\ldots,$ $\zeta_{2\log_2\Gamma}$. When $\Gamma$ is at most a power of $\gamma^{-1}$ (as is the case in the real proof) and $\gamma$ is small, the above inequality can be rewritten as $d_{n} \leq d_{n - 2\log_2 \Gamma}(1 
-\Omega(\gamma^{4/3}))$. Iterating this we finally get 
$$d_n \leq \Gamma (1 -\Omega(\gamma^{4/3}))^{n/2\log_2\Gamma} = \Gamma (2^{-n})^{\frac{\Omega(\gamma^{4/3})}{\log \gamma^{-1}}}\,,$$
which shows, in a high level, why we get the factor of $\gamma^{4/3}$ in the exponent as in Theorem~\ref{thm:main}.

We remark that the simple observation on the sloped switching strategy is more natural when considering continuous path in the plane --- this is why our main proof focuses on the case of continuous GFF.  In the case for discrete GFF, we first bound the distance minimizing the lengths over all continuous paths and then argue that for each continuous path there is a lattice path whose weight grows by a factor that is negligible. 

\subsection{Conventions, notations and some useful definitions}
In the remainder of the paper we assume that $\gamma$ is small enough (less than some small, positive absolute constant) for our bounds or inequalities to hold although 
we keep this implicit in our discussions. {Define $\Gamma = \min_{k \in \N}\{2^k: 2^k \geq \gamma^{-2}\}$ to be} the smallest integer power of 2 that is $\geq 
\gamma^{-2}$. Thus $1 \leq \Gamma \gamma^{2} < 2$. (It will be clear later that our analysis works if we alternatively define {$\Gamma = \min_{k\in \N}\{2^k: 
2^k \geq \gamma^{-\chi}\}$} for any $\chi >4/3$). {We further denote by $m_\Gamma = \log_2 \Gamma$.}

For any $w \in \R^2$, $\ell \in \R$ and $r > 0$, define $V_\ell^{r; w}$ as the rectangle $w + [0, r2^{-\ell}] 
\times [0, 2^{-\ell}]$. We will suppress $\ell$ or $w$ 
from this notation whenever they are 0. We will also omit 
$r$ when it is 1. The longer and shorter dimensions of a rectangle are called its \emph{length} and \emph{width} 
respectively. Two rectangles $R$ and $R'$ are called \emph{copies} of each other if $R$ can be obtained from 
$R'$ via translation and / or rotation by an angle. The rectangles $R$ and $R'$ are called \emph{non-overlapping} 
if their interiors are disjoint. If $R$ and $R'$ have the same dimensions then we say that they are \emph{adjacent} if 
they share one of their shorter boundary segments. A \emph{smooth path} is a $C^1$ map $P: [0, 1] \to \R^2$. We will also use $P$ to denote the image set of $P$ which is a subset of $\R^2$. This distinction should be clear from 
the context. For any rectangle $R = [a, b] \times [c, d]$ with sides parallel to the coordinate axes, we define its left, right, upper and lower boundary segments in the obvious way and denote them as $\lb R, \rb R, \ub R$ 
and $\db R$ respectively. Thus $\lb R$ is the path 
described by $(a, c + t(d - c)); t \in [0, 1]$ etc. 

For convenience, we identify (and denote) the points 
in $\R^2$ as complex numbers. {If $w = a + \iota b$ where $a, b \in \R$ and $\iota = \sqrt{-1}$, then 
we write $\mathrm{Re}(w) = a$ and $\mathrm{Im}(w) = b$. To avoid potential confusion, we will use bold fonts for real numbers when they are considered as complex 
numbers, i.e., as points on the real line. We denote the absolute value (i.e., the $\ell_2$-norm) of a complex 
number $w$ by $|w|$. It will also be used for the absolute value of a real number $r$, 
and when the argument of $|\cdot|$ is a finite set $A$ it should be interpreted as the cardinality of $A$.}

{We denote the set of all 
nonnegative real numbers by $\R^+$. For a subset $E$ of $\R^n$, we denote by $\mathcal B(E)$ the class of all Borel subsets of $E$.}

We will also continue to use the $O, \Omega, \Theta$ notations introduced at the beginning of this section.

\section{{White noise decomposition of the circle average process}}\label{sec-preliminary}
\label{subsec:white_noise}
In this section, we will introduce an approximation of the circle average process via the white noise decomposition 
of the Gaussian free field. A white noise $W$ distributed on $\R^2 \times \R^+$ refers to a centered Gaussian process $\{(W, f): f \in L^2(\R^2 \times \R^+)\}$ whose covariance kernel is given by 
$$\E (W, f) (W, g) = \int_{\R^2 \times \R^+}fgdzds\,.$$ 
An alternative notation for $(W, f)$ is $\int_{\R^2 \times 
\R^+}f W(dz, ds)$, which we use in this paper. For any $B \in \B(R^2)$ and $I \in \B(R^+)$, we let $\int_{B \times I}f W(dz, ds)$ denote the variable $\int_{\R^2 \times \R^+}f_{B \times I} W(dz, ds)$ where $f_{B \times 
I}$ is the restriction of $f$ to $B \times I$. Now define a Gaussian process $\{\hat h^{\mathcal U}_\delta(v) : v \in \mathcal U_\delta \}$ as follows (the factor $\sqrt{\pi}$ below corresponds to $\pi$ in \eqref{eq:GFF_cov}):
\begin{equation}
\label{eq:WND_decomposition*}
\hat h^{\mathcal U}_\delta(v) = \int_{\mathcal U \times (0, \infty)}\big(\int_{\partial B_\delta(v)} \sqrt{\pi}p_{\mathcal U}(s/2; w, z)\mu_\delta^v(dw)\big)W(dz, ds)\,.
\end{equation}
{\begin{lemma}\label{lem-identity-in-law}
The process $\{\hat h^{\mathcal U}_\delta(v):  v\in \mathcal U_\delta\}$ is identically distributed as $\{h_\delta^{\mathcal U}(v): v\in \mathcal U_\delta\}$ for all 
$\delta \in (0, r_\mathcal U)$ where $r_\mathcal U = $ $\sup\{r > 0: \mathcal U_r \neq \emptyset\}$. 
\end{lemma}
\begin{proof}
We only need to show that the covariance kernels are 
the same for these two processes. For $v, v'\in \mathcal U_\delta$, the covariance between $\hat h^{\mathcal U}_\delta(v)$ and $\hat h^{\mathcal U}_\delta(v')$ can be easily computed by Fubini's theorem and the Markov property of Brownian motion as shown below:
\begin{align*}
&\cov\big(\hat h^{\mathcal U}_\delta(v), \hat h^{\mathcal U}_\delta(v')\big) = \pi \int_{\partial B_\delta(v) \times \partial B_\delta(v')}\big(\int_{\R^2 \times \R^+}p_{\mathcal U}(s/2; w, z)p_{\mathcal U}(s/2; z, w')dzds\big)\mu_\delta^v(dw)\mu_\delta^{v'}(dw')\\
=&\pi \int_{\partial B_\delta(v) \times \partial B_\delta(v')}\big(\int_{\R^+}p_{\mathcal U}(s; w, w')ds\big)\mu_\delta^v(dw)\mu_\delta^{v'}(dw') =\pi \int_{\partial B_\delta(v) \times \partial B_\delta(v')}G_{\mathcal U}(w, w')\mu_\delta^v(dw)\mu_\delta^{v'}(dw')\,,
\end{align*}
where in the last step we used \eqref{eq:Green_fxn}. Comparing with \eqref{eq:GFF_cov}, we 
can now conclude that the two Gaussian processes are identically distributed. 
\end{proof}

We are interested in the white noise decomposition $\hat h_\delta^{\mathcal U}$ for the reason  that (as shown in Proposition~\ref{prop:coupling}) it is well-approximated  by a new family of Gaussian processes $\eta \coloneqq \{\eta_\delta^{\delta'}(v): v \in \R^2, 0 < \delta < \delta' \leq 1\}$ defined as:
\begin{equation}
\label{eq:field_definition}
\eta_\delta^{\delta'}(v) = \sqrt{\pi}\int_{\R^2 \times [\delta^2, {\delta'}^2]}p(s/2; v, w)W(dw, ds)  \mbox{ where } p(s/2; v, w) =  \frac{\e^{-\frac{|v - w|^2}{s}}}{\pi s}
\end{equation}
(namely, $p(t; v, w)$ is the transition probability density 
function of a standard two-dimensional Brownian motion). We can immediately observe the following properties of $\eta$ from the representation \eqref{eq:field_definition}:
\begin{enumerate}[(a)]
\item \emph{Invariance with respect to {isometries} of the plane.} Law of $\eta$ remains the same under any {isometry} (i.e. translation, rotation, reflection etc.) of $\R^2$.
\item \emph{Scaling property.} The processes $\{\eta_{\delta^*\delta}^{\delta^*\delta'}(\delta^* v): v \in \R^2\}$ and $\{\eta_\delta^{\delta'}(v): v \in \R^2\}$ are identically distributed for all $0 < \delta < \delta' \leq 1$ and $\delta^* \in (0, 1]$.
\item \emph{Independent increment.} The processes $\{\eta_{\delta_1}^{\delta_2}(v): v \in \R^2\}$ and $\{\eta_{\delta_3}^{\delta_4}(v): v \in \R^2\}$ are independent for all $0 < \delta_1 < \delta_2 < \delta_3 < \delta_4\leq 1$.
\end{enumerate}
We will suppress the superscript $\delta'$ in $\eta_\delta^{\delta'}$ whenever $\delta' = 1$. Notice that
\begin{equation}
\label{eq:variance}
\var (\eta_\delta(v)) = \pi \int_{[\delta^2, 1]}p(s; v, v)ds = \log \delta^{-1}\,.
\end{equation}
{The following important estimate is a simple consequence of \eqref{eq:field_definition}.}
\begin{lemma}
\label{lem:smoothness}
For all $v, w \in \R^2$, we have $\var (\eta_\delta(v) - \eta_\delta(w)) \leq  \frac{|v - w|^2}{\delta^2}$. 
\end{lemma}
\begin{proof}\belowdisplayskip=-12pt
This follows from \eqref{eq:field_definition} by straightforward computations:
\begin{align*}
\var (\eta_\delta(v) - \eta_\delta(w)) =& 2 \big(\var (\eta_\delta(v)) - \cov(\eta_\delta(v), \eta_\delta(w)) \big) = \int_{[\delta^2, 1]}\frac{1 - \e^{-\frac{|v - w|^2}{2s}}}{s}ds \\
\leq& \int_{s \in [\delta^2, 1]}\frac{|v - w|^2}{2s^2}ds \leq  \frac{|v - w|^2}{\delta^2}\,.
\end{align*}\qedhere
\end{proof}
As $\eta_\delta$ is a Gaussian process, {Lemma~\ref{lem:smoothness}} implies by Kolmogorov-Centsov theorem that there is a version of $\eta_\delta$ with 
continuous sample paths. Thus we can work with a continuous version of $\eta_\delta$ for any given $\delta$ and hence for any fixed, finite collection of $\delta$'s 
that we consider at any given instance.

{We next show that the process $h_\delta^{\mathcal U}$ is well-approximated by the process $\eta_\delta$.
\begin{proposition}
	\label{prop:coupling}
Let $V \subseteq \mathcal U_\ep$ for some $\ep > 0$. Then there exists $C_{\mathcal U, \ep} > 0$, depending only on $(\mathcal U, \ep)$, such that 
$$\P\big(\max_{v \in V}(\hat h^{\mathcal U}_\delta(v) - \eta_\delta(v)) > C_{\mathcal U, \ep}\sqrt{\log \delta^{-1}} + x\big) = \e^{-\Omega_{\mathcal U, \epsilon}(x^2)}\,$$ 
for all $\delta \in (0, \ep/4)$ and $x \geq 0$.
\end{proposition}}
The reader may skip the rest of the section for now, and come back to the  proof of Proposition~\ref{prop:coupling} after reading Sections~\ref{sec-construction} and \ref{sec-analysis}.
In order to prove Proposition~\ref{prop:coupling} we need the following tail bound for the maximum of Gaussian processes of a certain type. 
\begin{lemma}
\label{lem:gaussian_tail}
Let $\mathfrak R$ be a finite collection of subsets of $\R^2$ such that for every $R \in \mathfrak R$, there is a continuous, centered Gaussian process $\{Y_R(v): v \in R\}$ satisfying the following conditions:
\begin{enumerate}[(I)]
\item There exists $C > 0$ (same for all $R \in \mathfrak R$) such that for all $v, w \in R$, 
$$\var(Y_R(v) - Y_R(w)) \leq C \frac{|v - w|}{\diam(R)}$$
where $\diam(R)$ is the diameter of $R$.
\item There exists $C' > 0$ (same for all $R \in \mathfrak R$) such that $\max_{v  \in R}\var(Y_R(v)) \leq C'$.
\end{enumerate}
Then there exists $C'' > 0$, depending only on $C$ and $C'$, such that
$$\P(\max_{R \in \mathfrak R}\max_{v \in R} Y_R(v) > \sqrt{2 C'\log |\mathfrak R|} + C'' + x) \leq \e^{-x^2/2C'} \mbox{ for all }x \geq 0\,.$$
\end{lemma}
\begin{proof}
Applying Dudley's entropy bound on the expected supremum of a Gaussian process (see, e.g., \cite[Theorem~4.1]{A90}), we get from Condition~$(I)$ that
\begin{equation}
\label{eq:gaussian_tail1}
\E \max_{v \in R}Y_R(v) \leq C'''\,
\end{equation}
for any $R \in 
\mathfrak R$ and some $C''' > 0$ depending on $C$. Also Condition~$(II)$ and the Gaussian concentration inequality (see e.g., \cite[Equation (7.4), Theorem~7.1]{L01}) give us
\begin{equation}
\label{eq:gaussian_tail2}
\P(\max_{v \in R}Y_R(v) - \E \max_{v \in R}Y_R(v) > x) \leq \e^{-x^2 / 2C'}\,
\end{equation}	
for all $x \geq 0$. Recalling the identity
\begin{equation*}
\label{eq:gaussian_tail3}
\E \max(X, 0) = \int_{[0, \infty)}\P(X > x)dx\,,
\end{equation*}
we then get from \eqref{eq:gaussian_tail1} and \eqref{eq:gaussian_tail2} {and a union bound at the second step} that (we assume $|\mathfrak R| > 1$)
\begin{eqnarray}
\label{eq:gaussian_tail4}
\E \max_{R \in \mathfrak R}(\max_{v \in R}Y_R(v)) &\leq& \max_{R \in \mathfrak R}\E \max_{v \in R}Y_R(v) + \E \max_{R \in \mathfrak R}\big(\max_{v \in R}Y_R(v) - \E \max_{v \in R}Y_R(v)\big)\nonumber \\
&\leq& C''' + \sqrt{2C'\log |\mathfrak R|} + \int_{(\sqrt{2C'\log |\mathfrak R|}, \infty)}|\mathfrak R|\e^{-x^2/2C'}dx \nonumber \\
&\leq& C''' + \sqrt{2C'\log |\mathfrak R|} + \int_{[0, \infty)}\e^{-y^2/2C'}dy \nonumber \\
&\leq&  C''' +\sqrt{2C'\log |\mathfrak R|}  + O(\sqrt{C'})\,.
\end{eqnarray}
Furthermore, we have from Condition~$(II)$ and the Gaussian concentration inequality,
$$\P(\max_{R \in \mathfrak R}\max_{v \in R}Y_R(v) - \E \max_{R \in \mathfrak R}\max_{v \in R}Y_R(v) > x) \leq \e^{-x^2 / 2C'} \mbox{ for all } x \geq 0\,.$$
This together with \eqref{eq:gaussian_tail4} yields the desired tail bound.
\end{proof}

\begin{proof}[Proof of Proposition~\ref{prop:coupling}] For convenience let us use $\Delta_\delta(v)$ to denote $\hat h^{\mathcal U}_\delta(v) - \eta_\delta(v)$ and suppose that the following is true for all $v, w \in \mathcal U_{\ep}$ such that $|v - w| \leq \delta$:
\begin{align}
\var(\Delta_\delta(v) - \Delta_\delta(w)) &= O\Big(\frac{|v - w|}{\delta}\Big)\,,\label{eq:var_diff}\\
	 \max_{v \in \mathcal U_\ep}\var(\Delta_\delta(v)) &= O_{\mathcal U, \ep}(1)\,.	\label{eq:coupling1}
	\end{align}
Now subdivide $V$ into rectangles of diameter at most $\delta$ in some arbitrary but fixed fashion such that the resulting collection of rectangles $\mathfrak R$ satisfies 
$|\mathfrak R| \leq 8\delta^{-2}$. It is then clear that the conditions of Lemma~\ref{lem:gaussian_tail} are satisfied for this collection with $Y_R \coloneqq 
\{\Delta_\delta(v) : v \in R\}$, $C = O(1)$ and $C' = 
O_{\mathcal U, \ep}(1)$. Therefore it remains to verify \eqref{eq:var_diff} and \eqref{eq:coupling1}.
	
We first prove \eqref{eq:coupling1}. We remark that despite the fact that the proof is somewhat long, the analysis is all standard. For convenience we introduce some notations. We denote $p(s; v, w) - p_{\mathcal U}(s; v, w)$ by $\overline p_{\mathcal U}(s; 
v, w)$. Also for any function $f$ defined on $\R^+ \times \mathcal U_\epsilon \times \mathcal U_\epsilon$ and $\delta \leq \epsilon$, we let $f^\delta(\cdot; v, \cdot)=\int_{\partial B_\delta(v)}f(\cdot; v', \cdot)\mu_\delta^v(dv')$ be the circle average of $f$. Now notice, {from \eqref{eq:WND_decomposition*} and \eqref{eq:field_definition}}, that we can decompose $\Delta_\delta(v)$ into four components as follows:
\begin{equation*}
\Delta_\delta(v) = \sqrt{\pi} (G_{v; 1} + G_{v; 2} + G_{v; 3} + G_{v; 4})\,,
\end{equation*}
where 
\begin{align*}
& G_{v; 1} = \int_{\mathcal U \times [1, \infty)}p_{\mathcal U}^\delta(s/2; v, w)W(dw, ds)\,, G_{v; 2} = \int_{\mathcal U \times (0, \delta^2]}p_{\mathcal U}^\delta(s/2; v, w)W(dw, ds)\,,\\
& G_{v; 3} = -\int_{\R^2 \times [\delta^2, 1]}\overline p_{\mathcal U}^\delta(s/2; v, w)W(dw, ds) \mbox{ and } G_{v; 4} = \int_{\R^2 \times [\delta^2, 1]}\big(p^\delta(s/2; v, w) - p(s/2; v, w)\big)W(dw, ds)\,.
\end{align*}
We will show that 
the variance of each component is $O_{\mathcal U, \epsilon}(1)$. Let us begin with $\var (G_{v; 1})$. Observe that
\begin{eqnarray}
\label{eq:varG_1_1}
\var(G_{v; 1}) &=& \int_{[1, \infty)}\int_{\partial B_\delta(v) \times \partial B_\delta(v)}p_{\mathcal U}(s; v', v'')\mu_\delta^v(dv')\mu_\delta^v(dv'')ds \nonumber \\
&=& \int_{[1, \infty)}\int_{\partial B_\delta(v) \times \partial B_\delta(v)}p(s; v', v'')P^{\mathcal U}(s; v', v'')\mu_\delta^v(dv')\mu_\delta^v(dv'')ds\,,
\end{eqnarray}
where $P^{\mathcal U}(s; v', v'')$ is the probability that a (two-dimensional) Brownian bridge with starting point $v'$, ending point $v''$ and duration $s$ stays within 
$\mathcal U$. This Brownian bridge is identically distributed as the process $\{W_t - \tfrac{t}{s}(W_s - v'' + v') + v'\}_{t \in [0, s]}$ where $\{W_t\}_{t \in [0, s]} \coloneqq \{(W_{1, t}, W_{2, t})\}_{t \in [0, s]}$ is a standard 
two-dimensional Brownian motion starting at the origin. A simple upper bound on $P^{\mathcal U}(s; v', v'')$ can be given as
\begin{align*}
P^{\mathcal U}(s; v', v'') \leq \P(W_{s/2} - \tfrac{1}{2}(W_s - v'' + v') + v' \in \mathcal U)\leq \P\big(|W_{s/2} - \tfrac{1}{2}W_s| - \tfrac{1}{2}|v' - v''|\leq \mathrm{diam}(\mathcal U) \big)\,.
\end{align*}
Since $(W_{i,s/2} - \tfrac{1}{2}W_{i,s})$'s are independent Gaussian random variables with mean zero and variance $s/4$, $|W_{s/2} - \tfrac{1}{2}W_s|^2$ is distributed as an 
exponential variable with mean $s/2$. Hence the probability on the right hand side in the last display is bounded by
$O(1)(1 - \e^{-O \big(\frac{(\mathrm{diam}(\mathcal U) + |v' - v''|)^2}{s}\big)})$.
Plugging this into \eqref{eq:varG_1_1} we get (using the fact that $1 - \e^{-x} \leq x$ for all $x\geq 0$)
\begin{eqnarray*}
\var(G_{v; 1}) &=& O(1)\int_{[1, \infty)}s^{-1}(1 - \e^{-O \big(\frac{(\mathrm{diam}(\mathcal U) + |v' - v''|)^2}{s}\big)})ds \nonumber \\
&=& O(1)\int_{[1, \infty)}s^{-2}(\mathrm{diam}(\mathcal U)+ |v' - v''|)^2ds = O(\mathrm{diam}(\mathcal U)^2)\,.
\end{eqnarray*}
Next, we bound $\var (G_{v; 2})$. We see that
\begin{eqnarray*}
\var(G_{v; 2}) &=& \int_{(0, \delta^2]}\int_{\partial B_\delta(v) \times \partial B_\delta(v)}p_{\mathcal U}(s; v', v'')\mu_\delta^v(dv')\mu_\delta^v(dv'')ds \nonumber \\
&\leq& \int_{(0, \delta^2]}\int_{\partial B_\delta(v) \times \partial B_\delta(v)}p(s; v', v'')\mu_\delta^v(dv')\mu_\delta^v(dv'')ds \nonumber \\
&=& O(1)\int_{(0, \delta^2]}s^{-1}\int_{[0, 2\pi]}\e^{-\frac{\delta^2(1 - \cos \theta)}{s}}d\theta ds = O(1)\int_{(0, \delta^2]} \frac{\sqrt{s}}{\delta}s^{-1}ds = O(1)\,.
\end{eqnarray*}
Here we used the fact that $\int_{[0, 2\pi]}\e^{-\frac{\delta^2(1 - \cos \theta)}{s}}d\theta =2 \int_{[0, \pi]}\e^{-\frac{\delta^2 \cdot 2\sin^2(\theta/2)}{s}}d\theta = O(\sqrt{s}/\delta)$, since the main contribution to the integration comes from $\theta = O(\sqrt{s}/\delta)$. 
For $\var(G_{v; 3})$ we start with the following upper bound (see the expression of $G_{v, 3}$):
\begin{eqnarray}
\label{eq:varG_3_1}
\var(G_{v; 3}) &=& \int_{[\delta^2, 1]}\int_{\partial B_\delta(v) \times \partial B_\delta(v)} (p(s; v', v'') + p_{\mathcal U}(s; v', v''))\mu_\delta^v(dv')\mu_\delta^v(dv'')ds \nonumber\\
&& - 2\int_{[\delta^2, 1]}\int_{\partial B_\delta(v) \times \partial B_\delta(v)}\int_{\R^2}p_{\mathcal U}(s/2; v', z) p(s/2; v'', z)dz\mu_\delta^v(dv')\mu_\delta^v(dv'')ds \nonumber \\
&\stackrel{p_{\mathcal U} \leq p}{\leq}& \int_{[\delta^2, 1]}\int_{\partial B_\delta(v) \times \partial B_\delta(v)} (p(s; v', v'') + p_{\mathcal U}(s; v', v''))\mu_\delta^v(dv')\mu_\delta^v(dv'')ds \nonumber\\
&& - 2\int_{[\delta^2, 1]}\int_{\partial B_\delta(v) \times \partial B_\delta(v)}\int_{\R^2}p_{\mathcal U}(s/2; v', z) p_{\mathcal U}(s/2; v'', z)dz\mu_\delta^v(dv')\mu_\delta^v(dv'')ds \nonumber\\
&=&\int_{[\delta^2, 1]}\int_{\partial B_\delta(v) \times \partial B_\delta(v)} (p(s; v', v'') - p_{\mathcal U}(s; v', v''))\mu_\delta^v(dv')\mu_\delta^v(dv'')ds\nonumber \\
&=& \int_{[\delta^2, 1]}\int_{\partial B_\delta(v) \times \partial B_\delta(v)}p(s; v', v'') P_*^{\mathcal U}(s; v', v'')\mu_\delta^v(dv')\mu_\delta^v(dv'')ds\,,
\end{eqnarray}
where $P_*^{\mathcal U}(s; v', v'')$ is the probability that a Brownian bridge  with starting point $v'$, ending 
point $v''$ and  duration $s$  hits $\partial \mathcal U$. 
Like $P^{\mathcal U}(s; v', v'')$, we can bound $P_*^{\mathcal U}(s; v', v'')$ in terms of $\{W_t\}_{t \in [0, s]}$ as follows:
\begin{eqnarray*}
P_*^{\mathcal U}(s; v', v'') &\leq& \P\big(\max_{t \in [0, s]} |W_t - \tfrac{t}{s}(W_s - v''+ v')| \geq d_{\ell_2}(v', \partial \mathcal U)\big)\\ 
&\leq& \P\big(M_{1,s} - m_{1,s} + M_{2,s} - m_{2,s} \geq (d_{\ell_2}(v', \partial \mathcal U) - |v' - v''|)_+ \big)
\end{eqnarray*}
where $x_+ = \max(x, 0)$ for any real number $x$ and $M_{i,s} = \max_{t \in [0, s]}W_{i,t}$ and $m_{i,s} = \min_{t 
\in [0, s]}W_{i,t}$ for $i \in [2]$. Since $M_{i,s}^2$ and $m_{i,s}^2$ are distributed as $W_{i,s}^2$, they are Gamma random variables with shape parameter $1/2$ and mean $s$. 
Hence the probability above is bounded by
\begin{equation*}
P_*^{\mathcal U}(s; v', v'') = O(1)\e^{-\Omega\big(\frac{[(d_{\ell_2}(v', \partial \mathcal U) - |v' - v''|)_+]^2}{s}\big)}\,.
\end{equation*}
Combined with \eqref{eq:varG_3_1}, it yields that (recalling $d_{\ell_2}(v, \partial \mathcal U) \geq \epsilon$ and $2\delta < \ep/2$)
\begin{eqnarray*}
\var(G_{v; 3}) \leq O(1) \int_{[\delta^2, 1]} s^{-1}\max_{v', v''\in \partial B_\delta(v)}\e^{-\Omega\big(\frac{[(d_{\ell_2}(v', \partial \mathcal U) - |v' - v''|)_+]^2}{s}\big)}ds = O_\epsilon(1)\,.
\end{eqnarray*}
We are only left with $\var(G_{v; 4})$ now. Notice that
\begin{eqnarray*}
\var(G_{v; 4}) &=& \int_{[\delta^2, 1]}\int_{\partial B_\delta(v) \times \partial B_\delta(v)}p(s; v', v'')\mu_\delta^v(dv')\mu_\delta^v(dv'')ds - 2\int_{[\delta^2, 1]}\int_{\partial B_\delta(v)}p(s; v', v)\mu_\delta^v(dv')ds \nonumber \\
&& + \int_{[\delta^2, 1]}p(s; v, v)ds\nonumber \\
&=& I_1 -2I_2 + I_3\,.
\end{eqnarray*}
Since $p(s; v', v'') \leq (2\pi)^{-1} s^{-1}$ (see \eqref{eq:field_definition}), it follows that $I_1$ and $I_3$ are bounded above by 
$(2\pi)^{-1}\int _{[\delta^2, 1]}s^{-1} ds$. In addition,
\begin{eqnarray*}
I_2 = (2\pi)^{-1}\int _{[\delta^2, 1]}\e^{-\frac{\delta^2}{2s}}s^{-1} ds \geq (2\pi)^{-1}\int _{[\delta^2, 1]}(1 - \delta^2(2s)^{-1})s^{-1} ds\,.
\end{eqnarray*}
Putting all these estimates together we get $\var(G_{v; 
4}) = O(1)$. Thus $\var(\Delta_\delta(v)) = O_{\mathcal U, \epsilon}(1)$ for all $v \in {\mathcal U}_\epsilon$. 

To verify \eqref{eq:var_diff}, first observe that
$$\Delta_\delta(v) - \Delta_\delta(w) = (\hat h^{\mathcal U}_\delta(v) - \hat h^{\mathcal U}_\delta(w)) - (\eta_\delta(v) - \eta_\delta(w))\,.$$
Therefore it suffices to prove similar bounds for each of $\var(\hat h^{\mathcal U}_\delta(v) - \hat h^{\mathcal U}_\delta(w))$ and 
$\var(\eta_\delta(v) - \eta_\delta(w))$. The latter follows from Lemma~\ref{lem:smoothness}. The bound on  $\var(\hat h^{\mathcal U}_\delta(v) - \hat h^{\mathcal U}_\delta(w))$ was derived (in a more general set-up) in the proof of \cite[Proposition~2.1]{HMPeres2010}.\qedhere
\end{proof}

\section{Inductive construction of light paths} \label{sec-construction}
{In this section we will describe our construction of a light path between the shorter boundaries of $V^\Gamma$ when the underlying process is $\eta_{2^{-n}}$ (see \eqref{eq:field_definition} for its definition) --- the construction proceeds via an induction on $n\geq 1$ which is regarded as the scale. Section~\ref{subsec_strategy_broad} contains an overview, and a detailed description is provided in Section~\ref{subsec:strat2} except the optimization strategy in the inductive construction which we discuss in Section~\ref{subsec:construct_strat2}.

\subsection{An overview of the inductive construction}
\label{subsec_strategy_broad}
The main idea is inspired by the construction in 
Section~\ref{sec-history}. As already mentioned in Section~\ref{sec-history}, in place of $L_{i, j}$ (as well as $L_{j, \mathrm{slope}}^i$) we will use a thin rectangular strip $R_{i, j}$ (respectively $S_{i, j}$) of width $\Gamma^{-2}$ with similar length and 
orientation. By Property (c) of $\eta$, we can write $\eta_{2^{-n}}$ as the sum of three independent processes, namely, $\eta_{2^{-n}}^{\Gamma^{-2}} + \eta_{\Gamma^{-2}}^{0.5} + \eta_{0.5}$ for $n\geq 2m_\Gamma$ (recall $2^{m_\Gamma} = \Gamma$). Our construction will be measurable with respect to the processes $\eta_{2^{-n}}^{\Gamma^{-2}}$ and $\eta_{0.5}$. As we will see later in Lemma~\ref{lem:convex_variance}, $\eta_{0.5}$ has properties similar to those of $\zeta$ in 
Section~\ref{sec-history}. Hence we can use the strategy from Section~\ref{sec-history} to construct a tubular pathway across $V^\Gamma$ (see 
Figure~\ref{fig:StrategyII}) based on the value of 
$\eta_{0.5}$ (roughly) along $\ub R_{i, j}$'s (see Figure~\ref{fig:StrategyII_subdivision}). This pathway will essentially be a sequence of $R_{i, j}$'s and $S_{i, j}$'s with the consecutive strips linked at their 
ends. Clearly if there is a path through each such strip connecting its shorter boundaries and the paths through consecutive strips are connected at their ends, their union would  contain a path between the 
shorter boundaries of $V^\Gamma$ (see Figure~\ref{fig:StrategyII}). The construction of such paths requires two important steps, constructing efficient crossings through even smaller rectangles and joining these crossings into a connected path. We elaborate the constructions in these two steps in more detail below.

\noindent {\bf Step 1.} We subdivide each $R_{i, j}$ and $S_{i, j}$ into non-overlapping copies of the rectangle $[0, \Gamma^{-1}] \times [0, \Gamma^{-2}]$, i.e.,\ $V_{2m_\Gamma}^\Gamma$ (see Figures~\ref{fig:StrategyII_subdivision} and 
\ref{fig:sloped_gadget}) which we refer to as blocks. Thus each block has the same 
aspect ratio (i.e., the ratio of longer to shorter dimension) as that of $V^\Gamma$. We want to construct light paths through blocks when the underlying process is $\eta_{2^{-n}}^{\Gamma^{-2}}$ (recall $\eta_{2^{-n}} = \eta_{2^{-n}}^{\Gamma^{-2}} 
+ \eta_{\Gamma^{-2}}^{0.5} + \eta_{0.5}$). Let $R$ be a block and $f$ be an isometry of $\R^2$ such that 
$f(V^\Gamma_{2m_\Gamma}) = R$. The scaling property and the isometric invariance of $\eta$ imply that $\{\eta_{2^{-n}}^{\Gamma^{-2}}\big (f(\Gamma^{-2}v)\big) : v \in V^\Gamma\}$ and $\{\eta_{\Gamma^22^{-n}}(v) : v \in 
V^\Gamma\}$ are identically distributed. So if we already have an \emph{algorithm} to construct a light path connecting the shorter boundaries of $V^\Gamma$ when the underlying process is $\eta_{\Gamma^22^{-n}}$ (this corresponds to a previous scale $n - 2m_\Gamma$ as $2^{m_\Gamma}= \Gamma$), we can use the same algorithm to construct a light path $P_{f, n}$ (say) when the underlying process is $\{\eta_{2^{-n}}^{\Gamma^{-2}}\big 
(f(\Gamma^{-2}v)\big) : v \in V^\Gamma\}$. This will give us a light path $f(\Gamma^{-2}P_{f, n})$ through $R$ for 
the process $\eta_{2^{-n}}^{\Gamma^{-2}}$. 

\noindent {\bf Step 2.} We need to address the problem that the paths through two adjacent blocks in a strip or in two consecutive strips do not necessarily intersect at their ends. To this end, we construct very short paths around the junction of each pair of adjacent blocks (see 
Figure~\ref{fig:tying}) using
similar constructions at smaller scales. After all these constructions we will get a path that connects the shorter boundaries of $V^\Gamma$.}

{As hinted in the sketch above, we will eventually construct a \emph{random} collection of light paths whose union will contain the desired path. It will be convenient to formalize the notion of such collections as well as their ``weights'' (defined below)  for later analysis. To this end let $\mathcal P$ be a finite, deterministic and non-empty collection of smooth paths in $\R^2$. A \emph{(random) polypath} $\xi$ \emph{from} $\mathcal P$ is a collection of $\{0, 1\}$-valued random variables $\{J(\xi, P)\}_{P \in \mathcal P}$ such that $\cup_{P \in \mathcal P: J(\xi, P) = 1} P$ is a connected subset 
of $\R^2$. Thus one can view $\xi$ as a random sub-collection of $\mathcal P$ forming a connected set. 
We treat any smooth path $P$ as a polypath from 
$\mathcal P = \{P\}$ with $J(P, P) = 1$. We will often omit the reference to $\mathcal P$ when it is clear from 
the context and simply say that $\xi$ is a polypath.} A polypath $\xi$ is said to \emph{connect} two polypaths $\xi'$ and $\xi''$ if $\xi$ intersects $\xi'$ and $\xi''$ considered as subsets of $\R^2$. More generally we say that the polypaths $\xi_1, \xi_2, \cdots, \xi_k$ \emph{form} 
a polypath if their union is a connected subset of $\R^2$. A \emph{crossing} for ({or through}) a rectangle $R$ is any polypath $\xi$ that stays entirely within $R$ and connects 
its two shorter boundaries --- we remark that this is a slight modification of the notion of crossing used in Section~\ref{sec-history} and we stick to the new notion for the rest of the paper. {If $X = \{X(v): v \in \mathcal U\}$ is a continuous process and $\xi$ is a polypath from $\mathcal P$, then we define its \emph{weight computed with respect to $X$} (or simply the \emph{weight} when $X$ is clear from the context) 
as the quantity $\sum_{P \in \mathcal P}J(\xi, P)\int_P \e^{\gamma X(z)}|dz|$. We will denote this quantity by $\int_\xi \e^{\gamma X(z)}|dz|$.}

\subsection{A description of the inductive construction except the optimization strategy}
\label{subsec:strat2}
The goal in this subsection is to describe in details the geometric framework of our inductive construction for light crossings. 
{Fixing $\beta = \min_{k \in \N}\{2^k: 2^k \geq 
\gamma^{-2/3}\}$ (recall the choice of $\beta$ in Section~\ref{sec-history}), subdivide the strips $V_{2m_\Gamma}^{\Gamma^3; 0.75\iota}$ (recall that $\iota = \sqrt{-1}$) and $V_{2m_\Gamma}^{\Gamma^3; 0.25\iota}$ into non-overlapping translates of 
$V_{2m_\Gamma}^{\Gamma^2\beta}$. Since $V_{2m_\Gamma}^{\Gamma^2\beta}$ has length $2^{-2m_\Gamma}\Gamma^2\beta = \beta$, there are $\Gamma / \beta$ 
many translates of $V_{2m_\Gamma}^{\Gamma^2\beta}$ in each of the two strips.} Let us denote them as $R_{1, 1}, R_{1, 2}, \cdots, R_{1, \Gamma / \beta}$ and $R_{2, 1}, R_{2, 2}, \cdots, R_{2, \Gamma / \beta}$ 
respectively from left to right. Similarly one can subdivide each $R_{i, j}$ into non-overlapping translates of $V_{2m_\Gamma}^\Gamma$ which we call as its 
\emph{blocks}. See Figure~\ref{fig:StrategyII_subdivision} below for an illustration of this set-up. 
\begin{figure}[!htb]
\centering
\begin{tikzpicture}[semithick, scale = 2]
\draw (-3.8, -0.7) rectangle (3.8, 0.7);
\draw [fill = olive!20] (-3.8, 0.3) rectangle (-2.3, 0.36);
\node [scale = 0.8] at (-3.05, 0.5) {$R_{1, 1}$};

\draw [fill = olive!20] (-3.8, -0.36) rectangle (-2.3, -0.3);
\node [scale = 0.8] at (-3.05, -0.16) {$R_{2, 1}$};

\draw [fill = purple!20] (-2.3, 0.3) rectangle (-0.8, 0.36);
\node [scale = 0.8] at (-1.55, 0.5) {$R_{1, 2}$};

\draw [fill = purple!20] (-2.3, -0.36) rectangle (-0.8, -0.3);
\node [scale = 0.8] at (-1.55, -0.16) {$R_{2, 2}$};

\draw [fill = cyan!20] (0.8, 0.3) rectangle (2.3, 0.36);
\node [scale = 0.8] at (1.55, 0.5) {$R_{1, \Gamma/\beta - 1}$};
\draw [fill = cyan!20] (0.8, -0.36) rectangle (2.3, -0.3);
\node [scale = 0.8] at (1.55, -0.16) {$R_{2, \Gamma/\beta - 1}$};
\draw [fill = orange!10] (2.3, 0.3) rectangle (3.8, 0.36);
\node [scale = 0.8] at (3.05, 0.5) {$R_{1, \Gamma/\beta}$};
\draw [fill = orange!10] (2.3, -0.36) rectangle (3.8, -0.3);
\node [scale = 0.8] at (3.05, -0.16) {$R_{2, \Gamma/\beta}$};

\foreach \x in {-0.7, -0.25, ..., 0.7}
{\fill (\x, 0.33) circle [radius = 0.015];
 \fill (\x, -0.33) circle [radius = 0.015];
}

\foreach \x in {-3.6125, -3.425, ..., -2.675, -2.4875}
{\draw (\x, 0.3) -- (\x, 0.36);
 \draw (\x, -0.36) -- (\x, -0.3);
 
 \draw (-\x, 0.3) -- (-\x, 0.36);
 \draw (-\x, -0.36) -- (-\x, -0.3);
 
 \draw (\x + 1.5, 0.3) -- (\x + 1.5, 0.36);
 \draw (\x + 1.5, -0.36) -- (\x + 1.5, -0.3);
 
 \draw (-\x - 1.5, 0.3) -- (-\x - 1.5, 0.36);
 \draw (-\x - 1.5, -0.36) -- (-\x - 1.5, -0.3);
}



\draw [->] (0, 0.6) -- (0, 0.95);
\node [scale = 1, above] at (0, 0.95) {$V^\Gamma$};

\end{tikzpicture}
\caption{{\bf The rectangles $R_{i, j}$'s and its 
blocks.} In this (hypothetical) example each $R_{i, j}$ consists of 8 blocks.}
\label{fig:StrategyII_subdivision}
\end{figure}
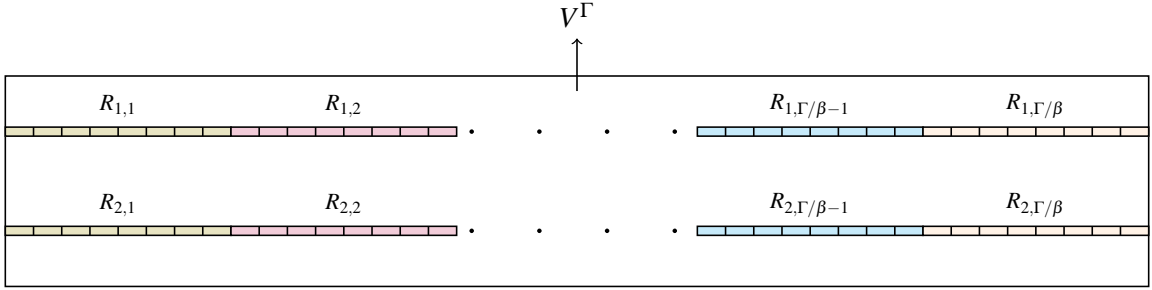

Next we describe our (crucial) sloped rectangle. Let $R_{i, j, \lt}$ and $R_{i, j, \rt}$ respectively 
denote the leftmost and rightmost blocks of $R_{i, j}$. 
\begin{lemma}\label{lem-S-1-j}
There exists a rectangle $S_{1, j}$ such that:
\begin{enumerate}[(I)]
\item $S_{1, j}$ is a copy of $V_{2m_\Gamma}^{l_\gamma\Gamma}$ for some integer $l_\gamma$.
\item The length of $S_{1, j}$, i.e., $l_\gamma \Gamma 2^{-2m_\Gamma} = l_\gamma \Gamma^{-1}$ is at most $|c_{R_{1,j, \lt}} - c_{R_{2, j, \rt}}| + 2\Gamma^{-1}$ where $c_R$ denotes the center of a rectangle $R$.
\item Each of the longer boundary segments of $S_{1, j}$ intersects each of the longer boundary segments of $R_{1, j, \lt}$ and $R_{2, j, \rt}$ (see 
Figure~\ref{fig:sloped_gadget}).
\end{enumerate}
\end{lemma}
\begin{proof}
\begin{figure}[!htb]
\centering
\begin{tikzpicture}[semithick, scale = 2.4]

\draw (-3, 0.5) rectangle (3, 0.6);
\draw (2.5, 0.6) -- (2.5, 0.8);
\draw [->] (2.5, 0.8) -- (2.7, 0.8);
\node [scale = 1, right] at (2.7, 0.8) {$R_{1, j}$};

\draw [fill = blue!20] (-3, 0.5) rectangle (-1.6, 0.6);
\draw [red, style = {decorate, decoration = {snake, amplitude = 0.4}}] (-3, 0.55) -- (-1.6, 0.55);
\draw [->] (-1.7, 0.6) -- (-1.7, 0.8);
\node [scale = 1, above] at (-1.6, 0.8) {$R_{1, j, \lt}$};

\draw (-3, -0.5) rectangle (3, -0.6);
\draw (-2.5, -0.5) -- (-2.5, -0.3);
\draw [->] (-2.5, -0.3) -- (-2.7, -0.3);
\node [scale = 1, left] at (-2.7, -0.3) {$R_{2, j}$};

\draw [fill = blue!20] (3, -0.6) rectangle (1.6, -0.5);
\draw [red, style = {decorate, decoration = {snake, amplitude = 0.4}}] (3, -0.55) -- (1.6, -0.55);
\draw [->] (2.9, -0.5) -- (2.9, -0.3);
\node [scale = 1, above] at (2.9, -0.3) {$R_{2, j, \rt}$};

\draw (-3, 0.7) -- (3, -0.73) -- (3.04, -0.64) -- (-2.96, 0.79) -- cycle; 
\draw [red, style = {decorate, decoration = {snake, amplitude = 0.4}}] (-2.98, 0.745) -- (3.02, -0.685);
\draw [->] (-0.27, 0.05) -- (-0.8, 0.05);
\node [scale = 1, left] at (-0.8, 0.05) {$S_{1, j}$};

\foreach \x in {-3, -1.8, ..., 3}{
\def \y {0.7 - \x * 1.43/6 - 1.43/2};
\draw (\x, \y) -- (\x + 0.04, \y + 0.09);
}
\end{tikzpicture}
\caption{{\bf The rectangles $R_{1, j, \lt}$, $R_{2, j, \rt}$ and $S_{1, j}$.} Each of the five rectangles comprising $S_{1, j}$ is a copy of 
$V_{2m_\Gamma}^\Gamma$. Hence $l_\gamma = 5$ in this example. The red lines inside each rectangle indicate the corresponding crossings.}
\label{fig:sloped_gadget}
\end{figure}
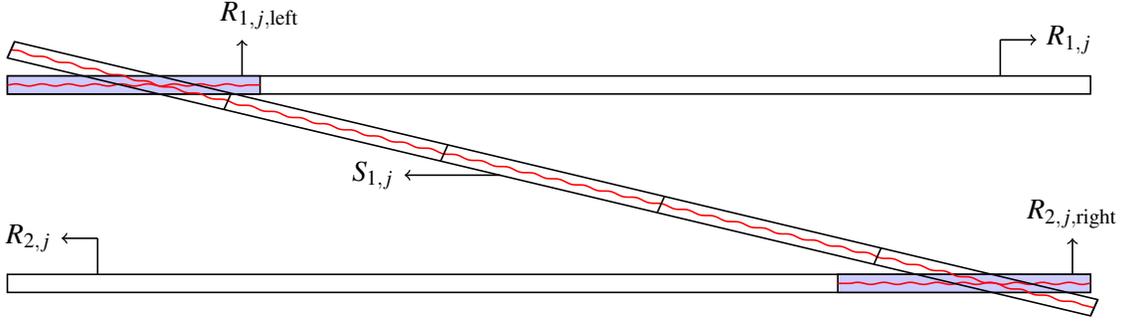

{We construct $S_{1, j}$ as follows (see Figure~\ref{fig:construct_S}). Extend the segment joining $c_{R_{1, j, \lt}}$ and $c_{R_{2, j, \rt}}$ to a segment with endpoints $c_{j, \up}$ and $c_{j, \down}$ where $\mathrm{Im}(c_{j, \up}) = 0.75 + 2\Gamma^{-2}$ and $\mathrm{Im}(c_{j, \down}) = 0.25 - \Gamma^{-2}$ (recall that the height of $\ub R_{1, j, \lt}$ is $0.75 + \Gamma^{-2}$ whereas the height of $\db 
R_{2, j, \rt}$ is 0.25). Let $c_{j, \up}'$ be the point $c_{j, \down} + t_{1, j}(c_{j, \up} - c_{j, \down})$ where $t_{1, j} = \min\{t \geq 1: t|c_{j, \up} 
- c_{j, \down}|\Gamma \in \N\}$. Now define $S_{1, j}$ as the unique copy of $V_{2m_\Gamma}^{|c_{j, \up}' - c_{j, \down}|\Gamma^2}$ for which $c_{j, \up}'$ and $c_{j, \down}$ are the midpoints of its shorter boundary 
segments. Since $l_\gamma = |c_{j, \up}' - c_{j, \down}|\Gamma$ is an integer by the definition of $c_{j, \up}'$, Condition (I) is automatically satisfied.}
\begin{figure}[!htb]
\centering
\begin{tikzpicture}[semithick, scale = 2.4]

\draw (-3, 0.5) rectangle (3, 0.6);

\draw [fill = blue!20] (-3, 0.5) rectangle (-1.6, 0.6);

\draw (-3, -0.5) rectangle (3, -0.6);

\draw [fill = blue!20] (3, -0.6) rectangle (1.6, -0.5);

\draw [dashed] (-3, 0.7) -- (3, -0.73) -- (3.04, -0.64) -- (-2.96, 0.79) -- cycle; 

\draw  (-2.98, 0.745) -- (3.02, -0.685);

\fill (-2.16, 0.55) circle [radius = 0.03];
\draw [->] (-2.16, 0.52) -- (-2.16, 0.3);
\node [scale = 1, below] at (-2.16, 0.3) {$c_{R_{1, j, \lt}}$};

\fill (-2.98, 0.745) circle [radius = 0.03];
\draw [->] (-2.98, 0.745) -- (-2.98, 0.965);
\node [scale = 1, above] at (-2.98, 0.965) {$c_{j, \up}'$};

\fill (-2.84, 0.711) circle [radius = 0.03];
\draw (-2.84, 0.711) -- (-2.84, 0.931);
\draw [->] (-2.84, 0.931) -- (-2.7, 0.931);
\node [scale = 1, right] at (-2.7, 0.931) {$c_{j, \up}$};

\fill (-2.372, 0.6) circle [radius = 0.03];
\draw (-2.372, 0.6) -- (-2.372, 0.82);
\draw [->] (-2.372, 0.82) -- (-2.232, 0.82);
\node [scale = 1, right] at (-2.232, 0.82) {$c_{j, \up}''$};

\fill (2.47, -0.56) circle [radius = 0.03];
\draw [->] (2.47, -0.56) -- (2.47, -0.34);
\node [scale = 1, above] at (2.47, -0.34) {$c_{R_{2, j, \rt}}$};

\fill (3.02, -0.685) circle [radius = 0.03];
\draw [->] (3.02, -0.685) -- (3.02, -0.905);
\node [scale = 1, below] at (3.02, -0.905) {$c_{j, \down}$};


\end{tikzpicture}
\caption{{\bf Construction of $S_{1, j}$.} The dashed lines indicate the boundary of $S_{1, j}$. The solid, slanted line is $L_{1, j}$}
\label{fig:construct_S}
\end{figure}
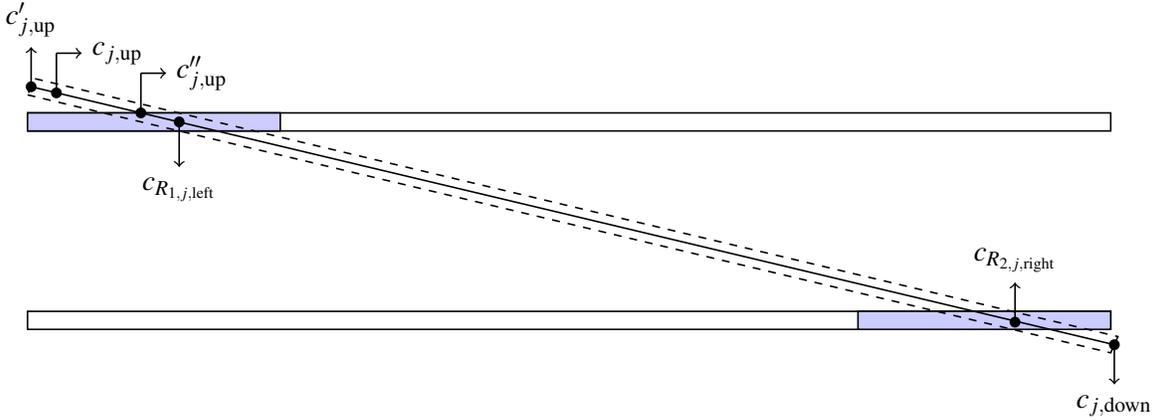

{To show that $S_{1, j}$ satisfies (II) and (III), we need the ``slope'' of the segment $L_{1, 
j}$ joining $c_{j,\up}'$ and $c_{j, \down}$. Let $\theta_{1, j}$ denote the acute angle between $L_{1, j}$ 
and the horizontal direction. We have
$$\tan \theta_{1, j} = \frac{\mathrm{Im}(c_{R_{1, j, \lt}} - c_{R_{2, j, \rt}})}{\mathrm{Re}(c_{R_{2, j, \rt}} - c_{R_{1, j, \lt}})} = \frac{(0.75 + 0.5\Gamma^{-2}) - (0.25 + 0.5\Gamma^{-2})}{\beta - 2\Gamma^{-2}} \geq 0.25 
\gamma^{2/3}\,$$
(since $\beta \leq 2\gamma^{-2/3}$). Hence
$$\mathrm{Re}(c_{R_{1, j, \lt}} - c_{j, \up}) = \frac{\mathrm{Im}(c_{j, \up} - c_{R_{1, j, \lt}})}{\tan \theta_{1, j}} \leq \frac{1.5\Gamma^{-2}}{0.25\gamma^{2/3}} = 6\gamma^{-2/3}\Gamma^{-2}\,.$$
Since $\gamma$ is small and $\Gamma \geq \gamma^{-2}$, it follows that $\mathrm{Re}(c_{R_{1, j, \lt}} - c_{j, \up}) 
\leq 0.2\Gamma^{-1}$. Similarly we get $\mathrm{Re}(c_{j, 
\down} - c_{R_{2, j, \rt}}) \leq 0.2\Gamma^{-1}$. Now triangle inequality gives us,
$$|c_{R_{1, j, \lt}} - c_{j, \up}| + |c_{j, 
\down} - c_{R_{2, j, \rt}}| \leq 2(1.5\Gamma^{-2} + 0.2\Gamma^{-1}) \leq 0.5\Gamma^{-1}\,,$$
where the last step follows from the smallness of $\gamma$. Also $|c'_{j, \up} - c_{j, \up}| \leq 
\Gamma^{-1}$ by the definition of $c_{j, \up}'$. Putting everything together we get
$$\mbox{Length of }L_{1, j} = \mbox{Length of }S_{1, j} = |c'_{j, \up} - c_{j, \down}| \leq |c_{R_{1,j, \lt}} - c_{R_{2, j, \rt}}| + 1.5\Gamma^{-1}\,,$$
which proves (II). Next let us look at the point $c_{j, \up}''$ (see Figure~\ref{fig:construct_S}) which is the point of intersection between $L_{1, j}$ and the line 
$\mathrm{Im}(w) = 0.75 + \Gamma^{-2}$. Notice that $$\mathrm{Re}(c_{R_{1, j, \lt}} - c_{j, \up}'') \leq \mathrm{Re}(c_{R_{1, j, \lt}} - c_{j, \up}) \leq 0.2\Gamma^{-1}\,.$$
Since $\Gamma^{-1}$ is the length of the rectangle $R_{1, j, \lt}$, it follows that $c_{1, j, \up}''$ lies on $\ub R_{1, j, \lt}$, i.e., $L_{1, j}$ intersects $\ub R_{1, j, \lt}$. Similar arguments hold for the intersection of $L_{1, j}$ with $\db R_{1, 
j, \lt}$, $\ub R_{2, j, \rt}$ and $\db R_{2, j, \rt}$. In fact we get the stronger statement that $L_{1, j}$ intersects the longer boundary segments of $R_{1, j, \lt}$ and $R_{2, j, \rt}$ at points which are at most $0.2 \Gamma^{-1}$ away from the midpoints of the respective 
segments. So in order to show (III) it suffices to prove that the longer boundary segments of $S_{1, j}$ intersect the lines $\mathrm{Im}(w) \in \{0.75 + \Gamma^{-2}, 0.75, 0.25 + \Gamma^{-2}, 0.25\}$ at points which are no further than $0.2\Gamma^{-1}$ away from the corresponding 
intersection points for $L_{1, j}$. The definition of $c_{j, \up}$ and $c_{j, \down}$ and the fact that the width of $S_{1, j}$ is $\Gamma^{-2}$ guarantee that each longer boundary segment of $S_{1, j}$ intersects the lines $\mathrm{Im}(w) \in \{0.75 + \Gamma^{-2}, 0.75, 0.25 + \Gamma^{-2}, 0.25\}$. The remaining part follows from the simple geometric observation that any horizontal line cuts $S_{1, j}$ in a segment of length at most
$$\frac{\mbox{width of }S_{1, j}}{\sin \theta_{1,j}} \leq \frac{\sqrt{1 + (0.25 \gamma^{2/3})^2}}{0.25\gamma^{2/3}}\Gamma^{-2} \leq 8\gamma^{-2/3}\Gamma^{-2} \leq 0.2 \Gamma^{-1}\,,$$
where again we used the smallness of $\gamma$ in the last step.}
\end{proof}
For $j \in [\Gamma/\beta]$, let $S_{1, j}$ be chosen as in Lemma~\ref{lem-S-1-j}. Similarly, we let  $S_{2, j}$ be a copy of $V_{2m_\Gamma}^{l_\gamma\Gamma}$ that joins $R_{2, j, \lt}$ and $R_{1, j, \rt}$ (so $S_{2, j}$ is obtained by reflecting $S_{1, j}$ through the horizontal line passing through the center of $S_{1, j}$). Like $R_{i, j}$'s (recall $i\in \{1, 2\}$), we subdivide $S_{i, j}$ into $l_\gamma$ non-overlapping copies of $V_{2m_\Gamma}^\Gamma$ (also called as \emph{blocks}). Henceforth we will refer to the collection of blocks of $S_{i, j}$'s and $R_{i, j}$'s as $\block_\gamma$. 

Now we are ready to describe our inductive construction of light crossings. {Suppose that for all integers $\ell \leq n - 1$, we already have an algorithm $\A_\ell$ that constructs a crossing through $V^\Gamma$ and takes only the processes 
$\{\eta_{2^{-k}}\}_{k \in [\ell]}$ as input. As described in 
Section~\ref{subsec_strategy_broad}, the scaling property and the isometric invariance of $\eta$ imply that we can use $\A_{n - 2m_\Gamma}$ to construct a crossing $\cross_{R, n}$ through each block $R \in \block_\gamma$ using only the processes $\{\eta_{2^{-\ell}}^{\Gamma^{-2}}\}_{2m_\Gamma < 
\ell \leq n}$ as input. Notice that this input is empty when $n < 2m_\Gamma$ and indeed we choose $\mathcal A_\ell$ to construct the straight line connecting the midpoints of the shorter boundary segments of $V^\Gamma$ 
for $\ell \leq 0$.} Henceforth whenever we talk about applying $\A_{\ell'}$ to construct a crossing at scale $n$ (where $\ell' \leq n$), we will suppress the statement that the processes used to construct it are $\{\eta_{2^{-\ell}}^{2^{-(n - \ell')}}\}_{n - \ell' < \ell 
\leq n}$. Furthermore for any rectangle $R \subseteq V^\Gamma$ which is a copy of $V_\ell^{\Gamma}$ for some $\ell \in \N$, we will use $\cross_{R, n}$ to denote the crossing through $R$ constructed  using $\A_{n -\ell}$. 

Next, we describe the construction which ties together the crossings through adjacent blocks. This can be achieved by a \emph{tying} technique, which we describe now in a slightly more general 
setting. The reader is referred to Figure~\ref{fig:tying} for an illustration. Let $k \in [n-1]$. Consider two adjacent copies of $V_k^\Gamma$. Without loss of generality (because of the rotational invariance property of $\eta_\delta^{\delta'}$ ), assume that their longer boundary segments are aligned 
with the horizontal axis. Call the left one as $R = I \times J$ and the right one as $R' = I' \times {J}$. {Also assume that $I = [\ell_I, r_I]$.} We want to link the crossings $\cross_{R, n}$ and $\cross_{R', n}$ to 
build a crossing for $R \cup R'$. To this end define three additional rectangles $R_{1, 2; 1} = [r_I - 2^{-k - m_\Gamma}, r_I] \times J$, $R_{1, 2; 2} = [r_I, r_I + 2^{-k - m_\Gamma}] \times J$ and $R_{1, 2; 3} = [r_I - 2^{-k - m_\Gamma}, r_I + 2^{-k - m_\Gamma}] \times [0, 2^{-k - 2m_\Gamma + 1}]$. 
Notice that all these rectangles have aspect ratio 
$\Gamma:1$. Also note that the union of crossings $\cross_{R, n}, \cross_{R', n}, \cross_{R_{1, 2; 1}, n}, \cross_{R_{1, 2; 2}, n}$ and $\cross_{R_{1, 2; 3}, n}$ is a crossing for the rectangle $R \cup R'$. We refer to this as the crossing obtained from \emph{tying} $\cross_{R, n}$ and $\cross_{R', n}$.

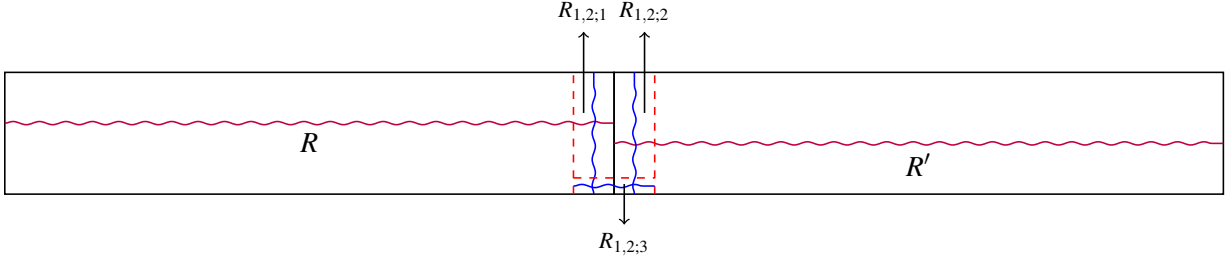
\begin{figure}[!htb]
\centering
\begin{tikzpicture}[semithick, scale = 2.7]
\draw (0, 0) rectangle (3, 0.6);
\draw [purple, style={decorate,decoration={snake,amplitude = 0.6}}] (0, 0.35) -- (3, 0.35);
\node [scale = 1, below] at (1.5, 0.35) {$R$};

\draw (3, 0) rectangle (6, 0.6);
\draw [purple, style={decorate,decoration={snake,amplitude = 0.6}}] (3, 0.25) -- (6, 0.25);
\node [scale = 1, below] at (4.5, 0.25) {$R'$};

\draw [red, dashed] (2.8, 0) -- (2.8, 0.6);
\draw [blue, style={decorate,decoration={snake,amplitude = 0.6}}] (2.9, 0) -- (2.9, 0.6);
\draw [->] (2.85, 0.4) -- (2.85, 0.8);
\node [scale = 0.8, above] at (2.85, 0.8) {$R_{1, 2; 1}$};

\draw [red, dashed] (3.2, 0) -- (3.2, 0.6);
\draw [blue, style={decorate,decoration={snake,amplitude = 0.6}}] (3.1, 0) -- (3.1, 0.6);
\draw [->] (3.15, 0.4) -- (3.15, 0.8);
\node [scale = 0.8, above] at (3.15, 0.8) {$R_{1, 2; 2}$};

\draw [red, dashed] (2.8, 0.08) -- (3.2, 0.08);
\draw [blue, style = {decorate, decoration = {snake, amplitude = 0.6}}] (2.8, 0.04) -- (3.2, 0.04);
\draw [->] (3.05, 0.05) -- (3.05, -0.15);
\node [scale = 0.8, below] at (3.05, -0.15) {$R_{1, 2; 3}$};

\end{tikzpicture}
\caption{{\bf Tying $\cross_{R, n}$ and $\cross_{R', n}$.} The crossings $\cross_{R, n}$ and $\cross_{R', n}$ are indicated by purple lines. The two vertical blue lines indicate the crossings $\cross_{R_{1, 2; 1}, n}$ (left) and $\cross_{R_{1, 2; 2}, n}$ (right). The horizontal blue line indicates the crossing $\cross_{R_{1, 2; 3}, n}$.}
\label{fig:tying}
\end{figure}

Finally, we will describe the construction which allows us to switch between top and bottom strips. We observe that if we tie all the pairs of crossings through adjacent blocks in $S_{i, j}$, we will get a crossing for $S_{i, 
j}$. Crucially, by Property (III) in Lemma~\ref{lem-S-1-j} the crossing for  $S_{i, j}$ then connects $\cross_{R_{i, j, \lt}, n}$ and $\cross_{R_{3-i, j, \rt}, n}$ (note that 
$i \mapsto 3 - i$ switches 1 and 2). Now consider an element $\mathbf i \coloneqq \{i_j\}_{j \in \Gamma / \beta}$ in 
$\{1, 2\}^{\Gamma / \beta}$. 
Given this element, we can build a collection $\mathscr 
C_{\mathbf i, n}$ of crossings as follows. Take any $j \in [\Gamma / \beta]$. If 
$i_j = i_{j + 1}$ (we adopt the convention that $i_{\Gamma/\beta +1} = i_{\Gamma/\beta}$), we include the crossing $\cross_{R, n}$ in $\mathscr C_{\mathbf i, n}$ for all the blocks $R$ of $R_{i_j, j}$. Otherwise we include $\cross_{R, n}$ for all the blocks $R$ of $S_{i_j, j}$ as well as $\cross_{R_{i_j, j, \lt}, n}$ and $\cross_{R_{3 - i_j, j, \rt}, n}$. Unless there is a switch at location~1 (in this case $S_{i, 1}$ can potentially intersect $\R^2 \setminus V^\Gamma$), the crossings in $\mathscr C_{\mathbf i, n}$ together with all the crossings that were used to tie pairs of crossings through adjacent blocks form a crossing for $V^\Gamma$. See 
Figure~\ref{fig:StrategyII} below for an illustration. Thus we are just one step short of defining the algorithm $\mathcal A_n$ namely that of choosing an appropriate $\mathbf i$. We call it the \emph{strategy} and we will choose it to be measurable with respect to $\eta_{0.5}$. 
We will \emph{derive} the strategy based on the this 
restriction in the next section. To simplify notations, we will use $\mathscr C_n$ to denote the collection of crossings corresponding to the strategy 
chosen by $\mathcal A_n$. Also we will refer to the collection of blocks $R$ such that $\cross_{R, n}$ is included in $\mathscr C_n$ from a ``location'' $j \in [\Gamma / \beta]$ as the \emph{bridge} at that location. 
\begin{figure}[!htb]
\centering
\begin{tikzpicture}[semithick, scale = 2]
\draw (-3.8, -0.7) rectangle (3.8, 0.7);
\draw [fill = olive!20] (-3.8, 0.3) rectangle (-2.3, 0.36);

\draw [fill = olive!20] (-3.8, -0.36) rectangle (-2.3, -0.3);

\draw [fill = purple!20] (-2.3, 0.3) rectangle (-0.8, 0.36);

\draw [fill = purple!20] (-2.3, -0.36) rectangle (-0.8, -0.3);

\draw [fill = cyan!20] (0.8, 0.3) rectangle (2.3, 0.36);

\draw [fill = cyan!20] (0.8, -0.36) rectangle (2.3, -0.3);

\draw [fill = orange!10] (2.3, 0.3) rectangle (3.8, 0.36);

\draw [fill = orange!10] (2.3, -0.36) rectangle (3.8, -0.3);


\foreach \x in {-0.7, -0.25, ..., 0.7}
{\fill (\x, 0.33) circle [radius = 0.015];
	\fill (\x, -0.33) circle [radius = 0.015];
}

\foreach \x in {-3.6125, -3.425, ..., -2.675, -2.4875}
{ \draw (\x, 0.3) -- (\x, 0.36);

	\draw (\x, -0.36) -- (\x, -0.3);
	\draw [blue] (\x - 0.03, -0.3) -- (\x - 0.03, -0.36);	
	\draw [blue] (\x + 0.03, -0.3) -- (\x + 0.03, -0.36);
	\draw [blue] (\x - 0.05, -0.34) -- (\x + 0.05, -0.34);

	\draw (-\x, 0.3) -- (-\x, 0.36);
	
	\draw (-\x, -0.36) -- (-\x, -0.3);
	\draw [blue] (-\x - 0.03, -0.3) -- (-\x - 0.03, -0.36);	
	\draw [blue] (-\x + 0.03, -0.3) -- (-\x + 0.03, -0.36);
	\draw [blue] (-\x - 0.05, -0.34) -- (-\x + 0.05, -0.34);

	\draw (\x + 1.5, 0.3) -- (\x + 1.5, 0.36);

	\draw (\x + 1.5, -0.36) -- (\x + 1.5, -0.3);

	\draw (-\x - 1.5, 0.3) -- (-\x - 1.5, 0.36);
	
	\draw (-\x - 1.5, -0.36) -- (-\x - 1.5, -0.3);

}
\draw [blue] (-2.3 + 0.03, -0.3) -- (-2.3 + 0.03, -0.36);
\draw [blue] (-2.3 - 0.03, -0.3) -- (-2.3 - 0.03, -0.36);
\draw [blue] (-2.3 - 0.05, -0.34) -- (-2.3 + 0.05, -0.34);

\draw [blue] (-2.3 + 1.5 - 0.03, 0.3) -- (-2.3 + 1.5 - 0.03, 0.36);
\draw [blue] (-2.3 + 1.5 - 0.05, 0.34) -- (-2.3 + 1.5, 0.34);
\draw [blue] (2.3 - 0.03, -0.3) -- (2.3 - 0.03, -0.36);
\draw [blue] (2.3 + 0.03, -0.3) -- (2.3 + 0.03, -0.36);
\draw [blue] (2.3 + 0.05, -0.34) -- (2.3 - 0.05, -0.34);

\draw [blue] (2.3 - 1.5 + 0.03, 0.3) -- (2.3 - 1.5 + 0.03, 0.36);
\draw [blue] (2.3 - 1.5 + 0.05, 0.34) -- (2.3 - 1.5, 0.34);
\foreach \x in {-3.6125,  -3.2375, ..., -2.3}{
	
	\draw [red, style = {decorate, decoration = {snake, amplitude = 0.4}}] (\x - 0.1875, -0.33) -- (\x, -0.33);
	
		

		
	\draw [red, style = {decorate, decoration = {snake, amplitude = 0.4}}] (-\x + 0.1875, -0.33) -- (-\x, -0.33);
		
		
}
\draw [red, style = {decorate, decoration = {snake, amplitude = 0.4}}] (-3.6125 + 1.5 - 0.1875, -0.33) -- (-3.6125 + 1.5, -0.33);
\draw [red, style = {decorate, decoration = {snake, amplitude = 0.4}}] (3.6125 - 1.5 + 0.1875, -0.33) -- (3.6125 - 1.5, -0.33);
\foreach \x in {-3.425, -3.05, ...,-2.3}{
	
	\draw [red, style = {decorate, decoration = {snake, amplitude = 0.4}}] (\x - 0.1875, -0.32) -- (\x, -0.32);
	
		

		
	\draw [red, style = {decorate, decoration = {snake, amplitude = 0.4}}] (-\x + 0.1875, -0.32) -- (-\x, -0.32);
		
		
}
\draw [red, style = {decorate, decoration = {snake, amplitude = 0.4}}] (-2.3 + 1.5 - 0.1875, 0.32) -- (-2.3 + 1.5, 0.32);
\draw [red, style = {decorate, decoration = {snake, amplitude = 0.4}}] (2.3 - 1.5 + 0.1875, 0.32) -- (2.3 - 1.5, 0.32);
	\draw [thin] (-2.25, -0.37) -- (-0.85, 0.38) -- (-0.81, 0.32) -- (-2.21, -0.43) -- cycle; 

	\draw [thin] (2.25, -0.37) -- (0.85, 0.38) -- (0.81, 0.32) -- (2.21, -0.43) -- cycle; 
	
	\foreach \x in {-2.075, -1.9, -1.725, -1.55, -1.375, -1.2, -1.025} {
		
		\def \y {-0.37 + \x*0.75/1.4 + 2.25*0.75/1.4};
		
		\draw (\x, \y) -- (\x + 0.04, \y - 0.06);
		
		\draw [blue] (\x - 0.025, \y - 0.025*0.75/1.14) -- (\x - 0.025 + 0.04, \y - 0.025*0.75/1.14 - 0.06);	
		\draw [blue] (\x + 0.025, \y + 0.025*0.75/1.14) -- (\x + 0.025 + 0.04, \y + 0.025*0.75/1.14 - 0.06);	
		\draw [blue] (\x - 0.04 + 0.08/3, \y - 0.04*0.75/1.14 - 0.04) -- (\x + 0.04 + 0.08/3, \y + 0.04*0.75/1.14 - 0.04);
		
		
		\draw (-\x, \y) -- (-\x - 0.04, \y - 0.06);
		\draw [blue] (-\x + 0.025, \y - 0.025*0.75/1.14) -- (-\x + 0.025 - 0.04, \y - 0.025*0.75/1.14 - 0.06);	
		\draw [blue] (-\x - 0.025, \y + 0.025*0.75/1.14) -- (-\x - 0.025 - 0.04, \y + 0.025*0.75/1.14 - 0.06);	
		\draw [blue] (-\x + 0.04 - 0.08/3, \y - 0.04*0.75/1.14 - 0.04) -- (-\x - 0.04 - 0.08/3, \y + 0.04*0.75/1.14 - 0.04);
		
		}

\foreach \x in {-2.075, -1.725, ..., -0.85} {

\def \y {-0.37 + \x*0.75/1.4 + 2.25*0.75/1.4};
						
\draw [red, style = {decorate, decoration = {snake, amplitude = 0.4}}] (\x - 0.175 + 0.02, \y - 0.175*0.75/1.4 - 0.03) -- (\x + 0.02, \y - 0.03);

\draw [red, style = {decorate, decoration = {snake, amplitude = 0.4}}] (-\x + 0.175 - 0.02, \y - 0.175*0.75/1.4 - 0.03) -- (-\x - 0.02, \y - 0.03);
			
}
			
\foreach \x in {-1.9, -1.55, ..., -0.85} {

\def \y {-0.37 + \x*0.75/1.4 + 2.25*0.75/1.4};

\draw [red, style = {decorate, decoration = {snake, amplitude = 0.4}}] (\x - 0.175 + 0.04*2/3, \y - 0.175*0.75/1.4 - 0.04) -- (\x + 0.04*2/3, \y - 0.04);

\draw [red, style = {decorate, decoration = {snake, amplitude = 0.4}}] (-\x + 0.175 - 0.04*2/3, \y - 0.175*0.75/1.4 - 0.04) -- (-\x - 0.04*2/3, \y - 0.04);
}			
	
\end{tikzpicture}
\caption{{\bf Construction of $\cross_n$.} In this example $i_1 = i_2 = 2$ and $i_3 = 1$; $i_{\Gamma / \beta - 1} = 1$ and $i_{\Gamma / \beta} = 2$. The red lines indicate the crossings in $\mathscr C_{\mathbf i, n}$ and the blue lines indicate the crossings used to tie them.}
\label{fig:StrategyII}
\end{figure}
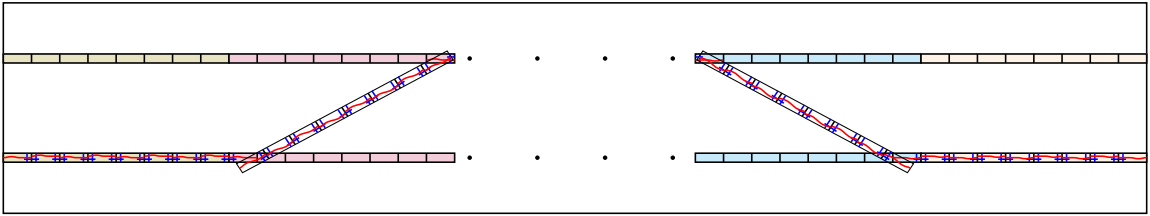
\section{Multi-scale analysis on expected weight of crossings}\label{sec-analysis}
For $n\geq 0$, let $\Lambda_{\gamma, n}$ denote the total weight of $\cross_n \coloneqq \cross_{V^\Gamma, n}$ computed with respect to $\eta_{2^{-n}}$ as the underlying process and let
$\lambda_{\gamma, n}$ denote its expectation. In Section~\ref{subsec:construct_strat2} we will describe our strategy (i.e., the strategy $\mathbf i$ mentioned at the end of Section~\ref{subsec:strat2}) and in Section~\ref{subsec:upper_bound_strat2} we will use it to derive a recurrence relation for $\lambda_{\gamma, n}$'s when $n \in \N$ (it is useful to recall at this point that $\lambda_{\gamma, m}$ should be interpreted as $\Gamma$ whenever $m \leq 0$) and show how this
relation leads to a bound on $\lambda_{\gamma, n}$. Finally in Section~\ref{sec:main_thm}, we use this bound along with Proposition~\ref{prop:coupling} to prove Theorem~\ref{thm:main}.
\subsection{Choosing the optimization strategy inductively}
\label{subsec:construct_strat2}
In this subsection we will analyze the weights for crossings following the geometric framework in Section~\ref{subsec:strat2}, which then naturally leads to our strategy. 
Since $\Lambda_{\gamma, 0}$ is trivial, we consider $n\geq 1$. We decompose $\Lambda_{\gamma, n}$ into two $\Lambda_{\gamma, n, \main}$ and $\Lambda_{\gamma, n, \gadg}$, where $\Lambda_{\gamma, n, \main}$ is the total weight of the crossings in $\mathscr C_n$ and $\Lambda_{\gamma, n, \gadg}$ is the total weight of gadgets that we use to tie pairs of crossings $(\cross_{R, n}, \cross_{R', n})$ in $\mathscr C_n$ for all adjacent $R, R'$ (see Section~\ref{subsec:strat2}). All the weights are computed with respect to the process $\eta_{2^{-n}}$. 
We will see that $\Lambda_{\gamma, n, \main}$ is the major component and will dictate our choice of the strategy.

When dealing with the weights of crossings, we will use $\mathcal P(\xi)$ to denote the finite collection of smooth paths underlying a crossing $\xi$ (see the definition of polypaths in 
Section~\ref{subsec_strategy_broad}). Our construction provides a natural way of defining these collections in an 
inductive manner. Note  that if we have already defined $\mathcal P(\cross_\ell)$, then we can define $\mathcal P(\xi)$ for any crossing $\xi$ that is constructed using $\mathcal A_\ell$ at any scale (see the discussion in the second paragraph in 
Section~\ref{subsec:strat2}). We can do this simply by applying the appropriate isometry to the elements of 
$\mathcal P(\cross_\ell)$. Now recall that when $\ell \leq 0$, $\mathcal A_\ell$ constructs a straight line connecting the midpoints of the shorter boundary 
segments of a rectangle. Thus $\mathcal P(\cross_\ell)$ in this case is simply the singleton set consisting of the 
corresponding path. For $\ell > 0$, suppose that we have already defined $\mathcal P(\cross_{\ell'})$ for all 
$\ell' < \ell$. Then, in light of discussions in Section~\ref{subsec:strat2},  there exists a finite and fixed collection of rectangles such that
\begin{itemize}
\item Through each of these rectangles a crossing can be constructed  using $\mathcal A_{\ell'}$ for some $\ell' < \ell$ and hence $\mathcal P(\xi)$ is defined for each such crossing $\xi$ (we then define $\mathcal P(\cross_\ell)$ as the union of all these $\mathcal P(\xi)$'s).  
\item A random subcollection of the aforementioned crossings (constructed  using $\mathcal A_{\ell'}$ for some $\ell' < \ell$)  forms
$\cross_\ell$.
\end{itemize}
Thus, we  see from our inductive construction that $\mathcal P(\cross_\ell)$ is a \emph{fixed} collection of paths, and each \emph{realization} of $\cross_\ell$ is formed by a sub-collection of paths in $\mathcal P(\cross_\ell)$, as desired. 

\medskip

We now turn to the choice of our strategy, whose derivation is divided into three steps. The strategy will be chosen based on an approximate expression for $\E (\Lambda_{\gamma, n, \main} | \eta_{0.5})$. 
The key is to analyze $\Lambda_{\gamma, n, \main; j}$ which is the combined weight of crossings through all 
the blocks in the bridge at location $j$ for $j\in [\Gamma/\beta]$.  

\smallskip

\noindent{\bf Step 1: Approximate $\Lambda_{\gamma, n, \main; j}$}. We first derive approximations for  $\Lambda_{\gamma, n, \main; j}$ (the bound on the 
approximation error is delayed to Section~\ref{subsec:upper_bound_strat2}, for a smoother flow of the derivation for our optimization strategy). Our approximation relies heavily on the fact that our ``desired'' strategy is determined by 
$\eta_{0.5}$. Recall the convention that $i_{\Gamma/\beta +1} = i_{\Gamma/\beta}$.

\noindent {\bf Case 1: $i_j = i_{j 
+ 1}$.} Since $\mathbf i \coloneqq \{i_j\}_{j \in [\Gamma / \beta]}$ is measurable with respect to $\eta_{0.5}$, we can write
{\begin{align*}
\label{eq:added_later1}
\E(\mathbf 1_{\{i_j = i_{j + 1}\}}\Lambda_{\gamma, n, \main; j}|\eta_{0.5}) = \mathbf 1_{\{i_j = i_{j + 1}\}} \sum_{R \in \block_\gamma, R \subseteq R_{i_j, j}} \sum_{P \in \mathcal P(\cross_{R, n})}\E\big(\int_P J(\cross_{R, n}, P)\e^{\gamma \eta_{2^{-n}}(z)}|dz| \big \vert \eta_{0.5}\big)\,,
\end{align*}
(recall from Section~\ref{subsec_strategy_broad} that $J(\cross_{R, n}, P) = \mathbf 1_{\{P  \in \cross_{R, 
n}\}}$). When $n > 2m_\Gamma$, $\cross_{R, n}$'s are constructed using the processes $\{\eta_{2^{-k}}^{\Gamma^{-2}}\}_{2m_\Gamma < k \leq n}$ which, by the independent increment, are 
independent of $(\eta_{\Gamma^{-2}}^{0.5}, \eta_{0.5})$. Also $\eta_{\Gamma^{-2}}^{0.5}$ and $\eta_{0.5}$ are independent of each other. 
Incorporating these observations into the previous display and using Fubini theorem we get
\begin{equation}
\E(\mathbf 1_{\{i_j = i_{j + 1}\}}\Lambda_{\gamma, n, \main; j}|\eta_{0.5}) = \mathbf 1_{\{i_j = i_{j + 1}\}} \sum_{R \in \block_\gamma, R \subseteq R_{i_j, j}} \sum_{P \in \mathcal P(\cross_{R, n})}\int_P a_{R, n, P}(z) b_n(z)\e^{\gamma\eta_{0.5}(z)}|dz|\,, \label{eq:conditional_expect}
\end{equation}
where $a_{R, n, P}(z) = \E (J(\cross_{R, n}, P)\e^{\gamma \eta_{2^{-n}}^{\Gamma^{-2} \vee 2^{-n}}(z)})$ and $b_n(z) = \E(\e^{\gamma \eta_{\Gamma^{-2} \vee 
2^{-n}}^{0.5}(z)})$. Notice that this expression remains valid even when $1\leq n \leq 2m_\Gamma$ ($J(\cross_{R, n}, P)$'s are deterministic in this case as $\cross_{R, n}$ is the straight line joining the midpoints of the 
shorter boundary segments of $R$). Replacing $b_n(z)$ in the expression with $1$ we get an approximation which we denote as
\begin{equation}
\label{eq-def-approx-1}
\Approx_{n, j, 1} = \mathbf 1_{\{i_j = i_{j + 1}\}} \sum_{R \in \block_\gamma, R \subseteq R_{i_j, j}} \sum_{P \in \mathcal P(\cross_{R, n})}\int_P a_{R, n, P}(z)\e^{\gamma\eta_{0.5}(z)}|dz|\,.
\end{equation}
We further approximate $\e^{\gamma \eta_{0.5}(z)}$ with $1 + \gamma \eta_{0.5}(z)$ and by Fubini we get a new approximation denoted as
\begin{eqnarray}
\label{eq-def-approx-2}
\Approx_{n, j, 2} &=& \mathbf 1_{\{i_j = i_{j + 1}\}} \sum_{R \in \block_\gamma, R \subseteq R_{i_j, j}}\E \int_{\cross_{R, n}}\e^{\gamma \eta_{2^{-n}}^{\Gamma^{-2} \vee 2^{-n}}(z)}|dz| + \mathbf 1_{\{i_j = i_{j + 1}\}}  Z_{\gamma, n, i, j} \,,\end{eqnarray}
where $$ Z_{\gamma, n, i, j} = \sum_{R \in \block_\gamma, R \subseteq R_{i, j}} \sum_{P \in \mathcal P(\cross_{R, n})}\int_P a_{R, n, P}(z)\gamma\eta_{0.5}(z)|dz|\,.$$
The integrals inside the first summation are weights of $\cross_{R, n}$'s computed with respect to 
$\eta_{2^{-n}}^{\Gamma^{-2} \vee 2^{-n}}$. Recall that we use $\mathcal A_{n - 2m_\Gamma}$ to construct these crossings and each block $R$ is a copy of 
$\Gamma^{-2}V^\Gamma$. The scaling property and the isometric invariance of $\eta$ then implies (recall $\lambda_{\gamma, m} = \Gamma$ for $m \leq 0$)
\begin{equation}
\label{eq:cost_compute_strat1_1}
\E \int_{\cross_{R, n}} \e^{\gamma \eta_{2^{-n}}^{\Gamma^{-2} \vee 2^{-n}}(z)}dz = \Gamma^{-2}\lambda_{\gamma, n - 2m_\Gamma}\,
\end{equation}
(compare to the situation when $\gamma = 0$). Plugging this into \eqref{eq-def-approx-2} we get
\begin{eqnarray}\label{eq-approximation-to-cite-1}
\Approx_{n, j, 2} = \mathbf 1_{\{i_j = i_{j + 1}\}}(\beta\Gamma^{-1} \lambda_{\gamma, n - 2m_\Gamma} +  Z_{\gamma, n, i_j, j}),
\end{eqnarray}
where $i_j \in [2]$ and we used the fact that $R_{i_j, j}$ contains $\beta\Gamma$ many blocks.} 
The small magnitude of $\gamma$ is crucial for these approximations. 

\smallskip

\noindent {\bf Case 2:  $i_j \neq i_{j + 1}$.} In this case, there is a switch at the location $j$, deriving 
{similar approximations requires slightly more work. Note that 
\begin{align*}
&\E(\mathbf 1_{\{i_j \neq i_{j + 1}\}}\Lambda_{\gamma, n, \main; j}|\eta_{0.5}) \\
=& \mathbf 1_{\{i_j \neq i_{j + 1}\}} \sum_{R \in \block_\gamma, R \subseteq S_{i_j, j} \cup R_{i_j, j, \mathrm{left}} \cup R_{3 - i_j, j, \mathrm{right}}} \sum_{P \in \mathcal P(\cross_{R, n})}\E\big(\int_P J(\cross_{R, n}, P)\e^{\gamma \eta_{2^{-n}}(z)}|dz| \big \vert \eta_{0.5}\big)\,.
\end{align*}}
As in Case 1 we can approximate the preceding expression by approximations denoted as {
\begin{equation}\label{eq-def-approx-prime}
\begin{split}
\Approx_{n, j, 1}'& =  \mathbf 1_{\{i_j \neq i_{j + 1}\}} \sum_{R \in \block_\gamma, R \subseteq S_{i_j, j} \cup R_{i_j, j, \mathrm{left}} \cup R_{3 - i_j, j, \mathrm{right}}}\sum_{P \in \mathcal P(\cross_{R, n})}\int_P a_{R, n, P}(z)\e^{\gamma\eta_{0.5}(z)}|dz|\,,\\
\Approx_{n, j, 2}' &= \mathbf 1_{\{i_j \neq i_{j + 1}\}}
\big(\big|\{R \in \block_\gamma: R \subseteq S_{i_j, j} \cup R_{i_j, j, \mathrm{left}} \cup R_{3 - i_j, j, \mathrm{right}}\}\big| \cdot \Gamma^{-2}\lambda_{n - 2m_\gamma} +   Z'_{\gamma, n, i_j, j}\big)\,,
\end{split}
\end{equation}
where
$$ Z'_{\gamma, n, i, j} = \sum_{R \in \block_\gamma, R \subseteq S_{i, j} \cup R_{i, j, \mathrm{left}} \cup R_{3 - i, j, \mathrm{right}}} \sum_{P \in \mathcal P(\cross_{R, n})}\int_P a_{R, n, P}(z)\gamma\eta_{0.5}(z)|dz|\,.$$}Recall from Lemma~\ref{lem-S-1-j} that the total number of blocks in the bridge in this 
case is $l_\gamma + 2$. {From Property~(II) in Lemma~\ref{lem-S-1-j} and the fact that $1<\beta <\Gamma$, we get
$$l_\gamma\Gamma^{-1} = \sqrt{O(1) + \beta^2} + 2\Gamma^{-1} = \beta + O(\beta^{-1})\,.$$
Hence the first term in the expression of $\Approx_{n, j, 2}'$ is 
$\mathbf 1_{\{i_j \neq i_{j + 1}\}}\big(\beta\Gamma^{-1}   + O(\beta^{-1}\Gamma^{-1})\big)\lambda_{\gamma, n - 2m_\Gamma}$. 
Defining $\loss_{\gamma, n, i, j}$ ($i \in [2]$) by
\begin{equation}
\label{eq:added_later7}
\loss_{\gamma, n, i, j} =  Z'_{\gamma, n, i, j} -  Z_{\gamma, n, i, j}\,,
\end{equation}
we obtain in this case 
\begin{equation}
\label{eq-approximation-to-cite-2}
\Approx_{n, j, 2}' =  \mathbf 1_{\{i_j \neq i_{j + 1}\}}\big(\beta\Gamma^{-1}   + O(\beta^{-1}\Gamma^{-1})\big)\lambda_{\gamma, n - 2m_\Gamma} + \mathbf 1_{\{i_j \neq i_{j + 1}\}}( Z_{\gamma, n, i_j, j}  + \loss_{\gamma, n, i_j, j})\,.
\end{equation}

\medskip

\noindent {\bf Step 2: Analyze the approximations.} In light of \eqref{eq-approximation-to-cite-1} and \eqref{eq-approximation-to-cite-2}, we next bound $\var(\sum_{j \in [N]} Z_{\gamma, n, i, j})$, $\var ( Z'_{\gamma, n, i, j})$ and $\var(\sum_{j \in [N]}( Z_{\gamma, n, 2, j} -  Z_{\gamma, n, 1, j}))$ when $N \in [\Gamma / 
\beta]$. These estimates should be compared to Properties~(b) and (c) of the process $\zeta$ in Section~\ref{sec-history}. The reader may skip the proof of the next lemma for a first time reading without interrupting the flow. 
\begin{lemma}
	\label{lem:convex_variance}
	For $n\geq 1$, we have $\var \big(\sum_{j \in [N]} Z_{\gamma, n, i, j}\big) = (\Gamma^{-1}\lambda_{\gamma, n - 2m_\Gamma})^2O(N\beta)\gamma^2$ for all $i \in [2]$ and $N \in [\Gamma / \beta]$. Also, we have $\var ( Z'_{\gamma, n, i, j}) = (\Gamma^{-1}\lambda_{\gamma, n - 2m_\Gamma})^2O(\beta)\gamma^2$ for all $i \in [2]$ and $j \in [\Gamma/\beta]$. In addition,
	\begin{equation}\label{eq-convex-variance}
	\var\big(\sum_{j \in [N]} Z_{\gamma, n, 2, j} - \sum_{j \in [N]} Z_{\gamma, n, 1, j}\big) = (\Gamma^{-1}\lambda_{\gamma, n - 2m_\Gamma})^2\Omega(N\beta)\gamma^2\,.
	\end{equation}
\end{lemma}
\begin{proof}
We will only prove \eqref{eq-convex-variance} since the other bounds follow from similar but simpler arguments. For clarity we divide the proof into three steps. First we show that for any block $R \subseteq R_{i, j}$,
\begin{equation}\label{eq:var_difference1}
\var\big( Z_{\gamma, n, R} - \Gamma^{-1}\lambda_{\gamma, n - 2m_\Gamma}\int_{\ub R}\gamma \eta_{0.5}(z)|dz| \big)
\end{equation}
is tiny where
\begin{equation}\label{eq-Z-n-R}
 Z_{\gamma, n, R} = \sum_{P \in \mathcal P(\cross_{R, n})}\int_P a_{R,n, P}(z)\gamma \eta_{0.5}(z)|dz|\,.
\end{equation}
Next we use this to show that 
\begin{equation}\label{eq:var_difference2}
\var \Big(\big(\sum_{j \in [N]} Z_{\gamma, n, 2, j} - \tilde Z_{\gamma, N, 2}\big) - \big(\sum_{j \in [N]} Z_{\gamma, n, 1, j} - \tilde Z_{\gamma, N, 1}\big)\Big)
\end{equation}
is small where 
\begin{equation}
\label{eq:z-tilde}
\tilde Z_{\gamma, N, i} = \sum_{j \in [N]}\sum_{R \in \block_\gamma, R \subseteq R_{i, j}}\Gamma^{-1}\lambda_{\gamma, n - 2m_\Gamma}\int_{\ub R}\gamma \eta_{0.5}(z)|dz| = \Gamma^{-1}\lambda_{\gamma, n - 2m_\Gamma}\int_{\ub V_{2m_\Gamma}^{N\beta\Gamma^2; v_i}}\gamma \eta_{0.5}(z)|dz|\,
\end{equation}
for $i \in [2]$, $v_1 = 0.75\iota$ and $v_2 = 0.25\iota$ (recall $\iota = \sqrt{-1}$). 
Finally, we derive a lower bound on $\var(\tilde Z_{\gamma, N, 2} - \tilde Z_{\gamma, N, 1})$ by an explicit computation which together with the bound on \eqref{eq:var_difference2} completes the proof.

Let us start with the bound on \eqref{eq:var_difference1}. 
To this end let $u \in R$. By Fubini, we can write
\begin{align*}
&\var\big(\int_{\ub R} \eta_{0.5}(z)|dz| - \Gamma^{-1}\eta_{0.5}(u) \big) \\
=& \Gamma^{-2}\int_{[0, 1]^2}\cov\Big(\eta_{0.5}(v_{R} + \bm{\Gamma^{-1}s}) - \eta_{0.5}(u), \eta_{0.5}(v_R + \bm{\Gamma^{-1}t}) - \eta_{0.5}(u)\Big)ds dt\,,
\end{align*}
where $v_{R}$ is the upper-left vertex of the rectangle $R$ (recall that $\ub R$ has length 
$\Gamma^{-1}$). Since the diameter of $R$ is $O(\Gamma^{-1})$, applying the Cauchy-Schwarz inequality and Lemma~\ref{lem:smoothness} to the last expression we get
\begin{equation}
\label{eq:convex_variance1}
\var\big(\int_{\ub R} \eta_{0.5}(z)|dz| - \Gamma^{-1}\eta_{0.5}(u) \big) = O(\Gamma^{-4})\,
\end{equation}
for any $u \in R$. Now for any path $P$ in $\mathcal P(\cross_{R, n})$, let us denote the integral $\int_Pa_{R, 
n,P}(z)|dz|$ as $q_{P, R}$. Write
	\begin{align*}
	&\int_Pa_{R, n,P}(z)\eta_{0.5}(z)|dz| - q_{P, R}\Gamma\int_{\ub R}\eta_{0.5}(z)|dz|\\
	=& \int_Pa_{R, n,P}(z)\eta_{0.5}(z)|dz| - \Gamma\int_Pa_{R, n,P}(z)|dz|\int_{\ub R}\eta_{0.5}(z')|dz'|
\\	=& \Gamma\int_Pa_{R, n,P}(z)\Big(\Gamma^{-1}\eta_{0.5}(z) - \int_{\ub R}\eta_{0.5}(z')|dz'|\Big)|dz|\,.
	\end{align*}
	Using Fubini as before, we get
	\begin{align}
		\label{eq:convex_variance1**}
		&\var \big(\int_Pa_{R,n, P}(z)\eta_{0.5}(z)|dz| - q_{P, R}\Gamma\int_{\ub R}\eta_{0.5}(z)|dz|\big) \nonumber\\
		=& \Gamma^2\int_{[0, 1]^2}a_{R, n, P}(P(s))a_{R, n, P}(P(t))C(s, t)|P'(s)|\cdot|P'(t)|ds dt\,,
		\end{align}
where
$$C(s, t) = \cov\Big(\Gamma^{-1}\eta_{0.5}(P(s)) - \int_{\ub R}\eta_{0.5}(z')|dz'|, \Gamma^{-1}\eta_{0.5}(P(t)) - \int_{\ub R}\eta_{0.5}(z')|dz'|\Big)\,.$$
Since $P(s), P(t) \in R$, we can apply the Cauchy-Schwarz inequality and \eqref{eq:convex_variance1} to obtain
$$C(s, t) \leq \sup_{u \in R}\var\big(\int_{\ub R} \eta_{0.5}(z')|dz'| - \Gamma^{-1}\eta_{0.5}(u) \big) = O(\Gamma^{-4})\,.$$
Plugging this bound into the right hand side of \eqref{eq:convex_variance1**} we get
	\begin{align}
	\label{eq:convex_variance1*}
	&\var \big(\int_Pa_{R,n, P}(z)\eta_{0.5}(z)|dz| - q_{P, R}\Gamma\int_{\ub R}\eta_{0.5}(z)|dz|\big) \nonumber\\
	=& \Gamma^2O(\Gamma^{-4})\big(\int_Pa_{R,n, P}(z)|dz|\big)^2 = O(q_{P,R}^2\Gamma^{-2})\,.
	\end{align}
Now recalling \eqref{eq-Z-n-R} and that 
$$\sum_{P \in \mathcal P(\cross_{R, n})} q_{P, R} = \E \int_{\cross_{R, n}}\e^{\gamma \eta_{2^{-n}}^{\Gamma^{-2} \vee 2^{-n}}(z)}|dz| = \Gamma^{-2}\lambda_{\gamma, n -2m_\Gamma}\,,$$
(see \eqref{eq:cost_compute_strat1_1}) we get from \eqref{eq:convex_variance1*}
\begin{align}
\label{eq:convex_variance2}
&\var\big( Z_{\gamma, n, R} - \Gamma^{-1}\lambda_{\gamma, n - 2m_\Gamma}\int_{\ub R}\gamma \eta_{0.5}(z) \big)\nonumber \\ 
=& \var \Big(\sum_{P \in \mathcal P(\cross_{R, n})}\big(\int_Pa_{R,n, P}(z)\gamma \eta_{0.5}(z)|dz| - q_{P, R}\Gamma\int_{\ub R}\gamma \eta_{0.5}(z)|dz|\big)\Big) \nonumber \\
\leq& \bigg(\sum_{P \in \mathcal P(\cross_{R, n})}\Big(\var \big(\int_Pa_{R,n, P}(z)\gamma \eta_{0.5}(z)|dz| - q_{P, R}\Gamma\int_{\ub R}\gamma \eta_{0.5}(z)|dz|\big)\Big)^{0.5} \bigg)^2 \nonumber \\
\stackrel{\eqref{eq:convex_variance1*}}{=}& O\big((\sum_{P \in \mathcal P(\cross_{R, n})}q_{P, R})^2\big)\Gamma^{-2}\gamma^2  = \big(\Gamma^{-1}\lambda_{\gamma, n - 2m_\Gamma}\big)^2O(\Gamma^{-4})\gamma^2\,.
\end{align}
Next let us bound \eqref{eq:var_difference2}. We have
\begin{align}
\label{eq:convex_variance3}
&\var \Big(\big(\sum_{j \in [N]} Z_{\gamma, n, 2, j} - \tilde Z_{\gamma, N, 2}\big) - \big(\sum_{j \in [N]} Z_{\gamma, n, 1, j} - \tilde Z_{\gamma, N, 1}\big)\Big) \leq 2 \sum_{i \in [2]}\var \big(\sum_{j \in [N]} Z_{\gamma, n, i, j} - \tilde Z_{\gamma, N, i}\big)\nonumber \\
=& 2 \sum_{i \in [2]}\var \Big( \sum_{j \in [N]} \sum_{R \in \block_\gamma, R \subseteq R_{i, j}}\big( Z_{\gamma, n , R} - \Gamma^{-1}\lambda_{\gamma, n - 2m_{\Gamma}}\int_{\ub R}\gamma \eta_{0.5}(z)|dz|\big)\Big)\nonumber \\
=& O(1) \sum_{i \in [2]} \bigg( \sum_{j \in [N]} \sum_{R \in \block_\gamma, R \subseteq R_{i, j}}\Big(\var\big( Z_{\gamma, n , R} - \Gamma^{-1}\lambda_{\gamma, n - 2m_{\Gamma}}\int_{\ub R}\gamma \eta_{0.5}(z)|dz|\big)\Big)^{0.5}\bigg)^2\nonumber \\
\stackrel{\eqref{eq:convex_variance2}}{=} & \sum_{i \in [2]} \bigg( \sum_{j \in [N]} \sum_{R \in \block_\gamma, R \subseteq R_{i, j}}\Gamma^{-1}\lambda_{\gamma, n - 2m_\Gamma}O(\Gamma^{-2})\gamma\bigg)^2 = \big(\Gamma^{-1}\lambda_{\gamma, n - 2m_\Gamma}\big)^2O(\Gamma^{-4})\gamma^2(N\beta\Gamma)^2 \nonumber \\
=& \big(\Gamma^{-1}\lambda_{\gamma, n - 2m_\Gamma}\big)^2 N\beta O(\Gamma^{-1})\gamma^2\,,
\end{align}
where in the last but one step we used that the number of blocks in $R_{i, j}$ is $\beta \Gamma$ and in 
the last step we used $N\beta \leq \Gamma$. It remains to estimate $\var (\tilde Z_{\gamma, N, 2} - 
\tilde Z_{\gamma, N, 1})$. For this purpose we can use the definition of $\eta_{0.5}(v)$ in \eqref{eq:field_definition} and Fubini to obtain:
\begin{eqnarray}
\label{eq:convex_variance4}
\var (\tilde Z_{\gamma, N, 2} - \tilde Z_{\gamma, N, 1}) &=& \var(\tilde Z_{\gamma, N, 1}) + \var(\tilde Z_{\gamma, N, 2}) - 2\cov(\tilde Z_{\gamma, N, 1}, \tilde Z_{\gamma, N, 2})\nonumber \\
&=&\big(\Gamma^{-1}\lambda_{\gamma, n - 2m_\Gamma}\big)^2\gamma^2\int_{[0, N\beta]^2 \times [0.25, 1]}s^{-1}\e^{-\frac{(x - z)^2}{2s}}(1 - \e^{-\frac{|0.75\iota - 0.25\iota|_2^2}{2s}})dx dz ds \nonumber \\
&=& \Omega\big(\Gamma^{-1}\lambda_{\gamma, n - 2m_\Gamma}\big)^2\gamma^2\int_{[0.25, 1]} \int_{[0, N\beta]}\int_{[0, N\beta]}\e^{-\frac{(x - z)^2}{2s}}dx dz ds \nonumber \\
&=&(\Gamma^{-1}\lambda_{\gamma, n - 2m_\Gamma})^2\Omega(N\beta)\gamma^2\,,
\end{eqnarray}
where in the final step we used the fact $\int_{[0, N\beta]}\e^{-\frac{(x - z)^2}{2s}}dx = \Omega(1)$ for all $z\in [0, N\beta]$ and $s\in [0.25,1]$. Since $\var X = \var Y + 2 \cov(X-Y, Y) + \var (X-Y)$ for any random variables $X, Y$ with finite second moment, we have 
$$
\var X \geq \var Y + 2\cov(X-Y, Y) \geq \var Y(1 - 2\sqrt{\var(X-Y)/\var Y}\,)\,,
$$
where the second inequality follows from Cauchy-Schwartz. Hence from \eqref{eq:convex_variance3} and \eqref{eq:convex_variance4} we get
\begin{eqnarray*}
\label{eq:convex_variance6}
\var\big(\sum_{j \in [N]} Z_{\gamma, n, 2, j} - \sum_{j \in [N]} Z_{\gamma, n, 1, j}\big) = (\Gamma^{-1}\lambda_{\gamma, n - 2m_\Gamma})^2\Omega(N\beta)\gamma^2(1 - O(\Gamma^{-1})) = (\Gamma^{-1}\lambda_{\gamma, n - 2m_\Gamma})^2\Omega(N\beta)\gamma^2\,. \qedhere
\end{eqnarray*}
\end{proof}}
{From Lemma~\ref{lem:convex_variance} we get
$$\var(\loss_{\gamma, n, i, j}) = \var 
( Z_{\gamma, n, i, j}' -  Z_{\gamma, n, i, j}) = (\Gamma^{-1}\lambda_{\gamma, n - 2m_\Gamma})^2 O(\beta)\gamma^2\,.$$
In addition, $\loss_{\gamma, n, i, j}$'s are (bounded) linear functionals of the white noise $W$ and thus are centered Gaussian 
variables. Therefore, the preceding bound implies
\begin{equation*}
\label{eq:induct_strat2_3}
\mbox{$\sum_{i \in [2]}$}\E (|\loss_{\gamma, n, i, j}|) = \Gamma^{-1}\lambda_{\gamma, n - 2m_\Gamma} O(\sqrt{\beta})\gamma\,.
\end{equation*}
Incorporating this bound into the expression for $\Approx_{n, j, 2}'$ we get the following upper bound on the expectation of $\APPROX_{n, 2} = \sum_{j \in [\Gamma / \beta]} (\Approx_{n, j, 2} + \Approx_{n, j, 2}')$: 
\begin{equation}
\label{eq:induct_strat2_4}
\E (\APPROX_{n, 2} ) \leq \lambda_{\gamma, n - 2m_\Gamma} + \E \sum_{j \in [\Gamma / \beta]}  Z_{\gamma, n, i_j, j} + C(\beta^{-1} + \gamma\sqrt{\beta})\Gamma^{-1}\lambda_{\gamma, n - 2m_\Gamma}N_\switch\,,
\end{equation}
where $C$ is a positive absolute constant and $N_\switch$ (to be selected shortly) is the number of  $j$'s where $i_j \neq i_{j+1}$ is allowed (i.e., the actually switching locations will be a random subset of these $j$'s). It might be helpful to compare this to \eqref{eq:toy_prob2}.} Since $ Z_{\gamma, n, i, j}$'s are centered, we have
\begin{equation*}
\label{eq:induct_strat2_5}
\E \sum_{j \in [\Gamma / \beta]}  Z_{\gamma, n, i_j, j} = \frac{1}{2}\E \sum_{j \in [\Gamma / \beta]}(-1)^{i_j + 1}( Z_{\gamma, n, 1, j} -  Z_{\gamma, n, 2, j})\,.
\end{equation*}

\noindent {\bf Step 3: choose our strategy.}  We naturally choose the strategy to minimize the following expectation:
\begin{equation}
\label{eq:induct_strat2_6}
E_{\gamma, n} = \E \Big( \frac{1}{2} \sum_{j \in [\Gamma / \beta]}(-1)^{i_j + 1}\Delta  Z_{\gamma, n, j} + C(\beta^{-1} + \gamma\sqrt{\beta})\Gamma^{-1}\lambda_{\gamma, n - 2m_\Gamma}N_\switch \Big)\,,
\end{equation}
where $\Delta  Z_{\gamma, n, j} =  Z_{\gamma, n, 1, j} -  Z_{\gamma, n, 2, j}$'s are centered Gaussian variables. From Lemma~\ref{lem:convex_variance} we can deduce that for any $1 \leq j_1 \leq j_2 \leq \Gamma/\beta$,
\begin{eqnarray*}
\var \big(\sum_{j_1 \leq j \leq j_2}\frac{1}{2}\Delta  Z_{\gamma, n, j}\Big) \geq  c(\Gamma^{-1}\lambda_{\gamma, n - 2m_\Gamma})^2(j_2 - j_1 + 1)\beta\gamma^2
\end{eqnarray*}
{for some absolute constant $c > 0$.} Recalling that $\E |Z| = \sqrt{\tfrac{2}{\pi}}$ for a standard Gaussian $Z$, we then get
\begin{equation}
\label{eq:induct_strat2_8}
\E \bigg |\sum_{j_1 \leq j \leq j_2}\frac{1}{2}\Delta  Z_{\gamma, n, j}\bigg| \geq 2 C (\beta^{-1} + \gamma\sqrt{\beta})\Gamma^{-1}\lambda_{\gamma, n - 2m_\Gamma}\,,
\end{equation}
whenever 
\begin{equation}
\label{eq:added_later9}
j_2 - j_1 + 1 \geq N'_\gamma \mbox{ where we denote } N'_\gamma = \frac{4C^2(\beta^{-1} + \gamma\sqrt{\beta})^2}{\frac{2}{\pi}c\beta\gamma^2}\,.
\end{equation}
Let $N_\gamma = \min_{k \geq 1} \{2^k: 2^k \geq 
N_\gamma'\}$. We are now ready to define our strategy. For $j\in [\Gamma/\beta]$, set
\begin{equation*}
\label{eq:induct_strat2_10}
i_j = \begin{cases}
2 &\mbox{ if } \sum_{(k_j - 1)N_\gamma + 1 \leq j' \leq k_jN_\gamma}\Delta  Z_{\gamma, n, j'} > 0\,, \\
1 &\mbox{ otherwise} \,,
\end{cases}
\end{equation*}
where $k_j \in \N$ is such that $(k_j - 1)N_\gamma + 1 \leq j \leq 
k_jN_\gamma$. {This is a valid strategy since 
$$\beta N_\gamma \leq 2\beta N_\gamma'  = O(\gamma^{-2/3}) < \Gamma\,$$(recall that $\beta = \Theta(\gamma^{-2/3})$ and thus $N_\gamma' = O(1)$). Also notice that $N_{\switch} = \Gamma(N_\gamma\beta)^{-1} - 1$ for this strategy.} It then follows from \eqref{eq:induct_strat2_6} and \eqref{eq:induct_strat2_8}  that
\begin{eqnarray}
\label{eq:induct_strat2_11}
E_{\gamma, n} &=& - \sum_{k \in [\Gamma(N_\gamma\beta)^{-1}]}\E\big\lvert \sum_{(k - 1)N_\gamma + 1 \leq j \leq kN_\gamma}\Delta  Z_{\gamma, n, j}\big \rvert + C (\beta^{-1} + \gamma\sqrt{\beta})\Gamma^{-1}\lambda_{\gamma, n - 2m_\Gamma}N_{\switch}\nonumber \\
&\leq& - C (\beta^{-1} + \gamma\sqrt{\beta})\Gamma^{-1}\lambda_{\gamma, n - 2m_\Gamma}\Gamma(N_\gamma\beta)^{-1} 
= -\Omega(\gamma^{4/3})\lambda_{\gamma, n - 2m_\Gamma}\,.
\end{eqnarray}
Notice that this strategy ensures $i_1 = i_2$, i.e., there is no switch at location~1 (since $N_\gamma \geq 2$) which implies we get a ``legitimate'' crossing (see the discussions at the end of Section~\ref{subsec:strat2}).
\subsection{{An upper bound on  $\lambda_{\gamma, n}$} via recursion}
\label{subsec:upper_bound_strat2}
The goal in this subsection is the following upper bound on $\lambda_{\gamma, n}$ for $n\geq 0$ (recall $\lambda_{\gamma, m} = \Gamma$ for $m \leq 0$).
\begin{lemma}\label{lem-upper-bound-lambda}
For all $n \geq 0$ we have (below $\lfloor n / 2m_\Gamma\rfloor$ is the largest integer $\leq n / 2m_\Gamma$)
\begin{equation}
\label{eq:induct}
\lambda_{\gamma, n} \leq \Gamma (1 - 0.5c\gamma^{4/3})^{\lfloor n / 2m_\Gamma\rfloor}\,.
\end{equation}
\end{lemma}
The key ingredient for the upper bound on $\lambda_{\gamma, n}$ is a recursive relation for $\lambda_{\gamma, n}$, as in the next lemma.
\begin{lemma}\label{lem-lambda-recursion}
The following recursive inequality holds for an absolute constant $c > 0$ and all $n\geq 1$:
\begin{equation}
\label{eq:cost_bound_strat2}
\lambda_{\gamma, n} \leq \lambda_{\gamma, n - 2m_\Gamma}(1 - c\gamma^{4/3}) + O(1)(\Gamma^{-1}\lambda_{\gamma, n - 3m_\Gamma} + \Gamma^{-2}\lambda_{\gamma, n - 4m_\Gamma + 1})\,.
\end{equation}
\end{lemma}
\begin{proof}
 In order to derive the recursion, we first estimate the expected errors that we made at every stage of our approximation in
Section~\ref{subsec:construct_strat2}. {Denote  $\APPROX_{n,1} = \sum_{j \in [\Gamma / \beta]}(\Approx_{n, j, 1} + \Approx_{n, j, 1}')$ (see \eqref{eq-def-approx-1} and \eqref{eq-def-approx-prime} for definitions of $\Approx_{n, j, 1}$ and $\Approx_{n, j, 1}'$). 
From \eqref{eq:variance} and the translation invariance of $\eta$ (see also \eqref{eq:conditional_expect} and \eqref{eq-def-approx-1}) we get for $n\geq 1$
\begin{equation}
\label{eq:cost_compute1}
\E \Lambda_{\gamma, n, \main} = \E \e^{\gamma \eta_{\Gamma^{-2} \vee 2^{-n}}^{0.5}(\bm 0)}\E (\APPROX_{n, 1}) 
 = (1 + O(\gamma^2 \log \gamma^{-1}))\E (\APPROX_{n, 1})\,.
\end{equation}
Next we control the difference between $\APPROX_{n,1}$ and $\APPROX_{n,2} = \sum_{j \in [\Gamma / \beta]}(\Approx_{n, j, 2} 
+ \Approx_{n, j, 2}')$ (see \eqref{eq-def-approx-2} and \eqref{eq-def-approx-prime} for their definitions).} Since $\e^x \geq 1 + x$, it follows that $\APPROX_{n, 1} \geq \APPROX_{n, 2}$. {In addition, applying Fubini to the expressions of $\Approx_{n, j, 1} - \Approx_{n, j, 2}$ and $\Approx_{n, j, 1}' - \Approx_{n, j, 2}'$ and using the translation invariance of $\eta$ we get that
\begin{eqnarray*}
\label{eq:cost_compute2}
\E (\APPROX_{n, 1} - \APPROX_{n, 2}) &\leq& \E(\e^{\gamma \eta_{0.5}(\bm 0)} - 1 - \gamma \eta_{0.5}(\bm 0))\sum_{R \in \block_{\gamma}}\E \int_{\cross_{R, n}} \e^{\gamma \eta_{2^{-n}}^{\Gamma^{-2} \vee 2^{-n}}(z)}|dz| \\
&=& O(\gamma^2)|\block_\gamma|\Gamma^{-2}\lambda_{\gamma, n - 2m_\Gamma}\,.
\end{eqnarray*}
Combined with the fact that
$$|\block_\gamma| = \sum_{i \in [2], j \in [\Gamma/ \beta]}\Big(\frac{\mbox{length of }R_{i, j}}{\Gamma^{-1}} + \frac{\mbox{length of }S_{i, j}}{\Gamma^{-1}}\Big) = O(\Gamma / \beta)O(\beta \Gamma) = O(\Gamma^2),$$
it follows that (recall $\lambda_{\gamma, m} = \Gamma$ for $m \leq 0$)
\begin{equation*}
\label{eq:cost_compute3}
\E (\APPROX_{n, 1} - \APPROX_{n, 2}) = O(\gamma^2)\lambda_{\gamma, n - 2m_\Gamma}\,.
\end{equation*}}Since $\E (\APPROX_{n, 2}) \leq \lambda_{\gamma, n - 2m_\Gamma} + E_{\gamma, n}$ (see \eqref{eq:induct_strat2_4}, \eqref{eq:induct_strat2_6}), we combine the preceding display with \eqref{eq:cost_compute1} and \eqref{eq:induct_strat2_11}  and get that 
\begin{equation}
\label{eq:cost_compute8}
\E \Lambda_{\gamma, n, \main} \leq \lambda_{\gamma, n - 2m_\Gamma}(1 - \Omega(\gamma^{4/3}))\,.
\end{equation}

Next we bound $\E \Lambda_{\gamma, n, \gadg}$. 
Recall from Section~\ref{subsec:strat2} that we use three additional crossings for tying the pair $(\cross_{R, n}, 
\cross_{R', n})$ for any pair of adjacent blocks $(R, R')$. 
Two of these crossings are constructed using $\mathcal A_{n - 3m_\Gamma}$ and the other one is constructed using 
$\mathcal A_{n - 4m_\Gamma + 1}$. Hence by a similar reasoning as used for \eqref{eq:cost_compute_strat1_1} and the independent increment and the translation invariance of $\eta$, we obtain the following upper bound on the expected total weight of these crossings (computed with respect to $\eta_{2^{-n}}$):
\begin{equation}
\label{eq:cost_compute_strat1_3}
2\E \e^{\gamma \eta_{\Gamma^{-3}\vee 2^{-n}}(\bm 0)}\Gamma^{-3}\lambda_{\gamma, n - 3m_\Gamma} + \E \e^{\gamma \eta_{2\Gamma^{-4}\vee 2^{-n}}(\bm 0)}2\Gamma^{-4}\lambda_{\gamma, n - 4m_\Gamma + 1}\,.
\end{equation}
Since there are at most $|\block_\gamma| = O(\Gamma^2)$ many tyings, this along with \eqref{eq:variance} implies  that
\begin{equation}
\label{eq:cost_compute11}
\E \Lambda_{\gamma, n, \gadg} \leq (2 + O(\gamma^2\log \gamma^{-1}))(\Gamma^{-1}\lambda_{\gamma, n - 3m_\Gamma} + \Gamma^{-2}\lambda_{\gamma, n - 4m_\Gamma + 1})\,.
\end{equation}
Since $\lambda_{\gamma, n} = \E \Lambda_{\gamma, n, \main} + \E \Lambda_{\gamma, n, \gadg}$, by \eqref{eq:cost_compute8}, \eqref{eq:cost_compute11} we conclude the proof of the lemma.
\end{proof}

\begin{proof}[Proof of Lemma~\ref{lem-upper-bound-lambda}]
The proof is based on Lemma~\ref{lem-lambda-recursion} and an induction argument.  Note that \eqref{eq:induct} obviously holds for $n = 0$ since $\lambda_{\gamma, n} = \Gamma$ for $n \leq 0$. 
Now fix $n \in \N$ and assume that \eqref{eq:induct} holds for all $m < n$. Combined with Lemma~\ref{lem-lambda-recursion} and the fact $\Gamma \geq \gamma^{-2}$, this implies
\begin{eqnarray*}
\label{eq:upper_bound}
\lambda_{\gamma, n} &\leq& \Gamma (1 - 0.5c\gamma^{4/3})^{\lfloor n / 2m_\Gamma \rfloor - 1}(1 - c\gamma^{4/3}) + O(\gamma^2)\Gamma(1 - 0.5c\gamma^{4/3})^{\lfloor n / 2m_\Gamma - 1.5\rfloor} \\
&=& \Gamma (1 - 0.5c\gamma^{4/3})^{\lfloor n / 2m_\Gamma \rfloor}(1 - 0.5c\gamma^{4/3} + O(\gamma^2)) \leq \Gamma (1 - 0.5c\gamma^{4/3})^{\lfloor n / 2m_\Gamma \rfloor}\,,
\end{eqnarray*}
thus completing the induction step. Therefore, \eqref{eq:induct} holds for all $n\geq 1$.
\end{proof}

\subsection{Proof of Theorem~\ref{thm:main}}
\label{sec:main_thm}
We are now ready to give the proof for Theorem~\ref{thm:main}. By Lemma~\ref{lem-identity-in-law}, we can couple the processes $h_\delta^{\mathcal U}$ and  $\hat h^{\mathcal U}_\delta$ so that they are identical. Let us define, for $v, w \in V$, 
$$D_{\eta, \gamma, \delta}(v, w) = \inf_{P} \int_{P}\e^{\gamma \eta_\delta(z)}|dz|\,$$ 
where $P$ ranges over all piecewise smooth paths in $V$ connecting $v$ and $w$. Also denote by $\Lambda_{h^\mathcal U, \gamma, \delta}^\mathrm{straight}(v, w)$ the weight of the 
straight line joining $v$ and $w$ when the underlying process is $h_\delta^{\mathcal U}$. Then, we have
\begin{equation}
\label{eq:field_change}
D_{\gamma, \delta}^{\mathcal U}(v, w) \leq \e^{\gamma(C_{\mathcal U, \ep} + 1)(\log \delta^{-1})^{2/3}}D_{\eta, \gamma, \delta}(v, w)\mathbf 1_{E_{\mathrm{couple}}} + \Lambda_{h^\mathcal U, \gamma, \delta}^\mathrm{straight}(v, w)\mathbf 1_{E_{\mathrm{couple}}^c}\,,
\end{equation}
where $C_{\mathcal U, \epsilon}$ is the constant from Proposition~\ref{prop:coupling} and $E_{\mathrm{couple}}$ is defined by
$$E_{\mathrm{couple}} = \{\max_{u \in V}(h_\delta^{\mathcal U}(u) - \eta_\delta(u)) \leq (C_{\mathcal U, \ep} + 1)(\log \delta^{-1})^{2/3}\}\,.$$
Thus, we can conclude the proof of Theorem~\ref{thm:main}, provided with Lemmas~\ref{lem-bound-straight-Lambda} and \ref{lem-bound-eta-distance} below.
\begin{lemma}\label{lem-bound-straight-Lambda}
Assume that $\delta<\ep/4$. Then we have $\E \Lambda_{h^\mathcal U, \gamma, \delta}^\mathrm{straight}(v, w)\mathbf 1_{E_{\mathrm{couple}}^c} =  O_{\mathcal U, \epsilon}(1)\e^{-\Omega_{\mathcal U, \epsilon}(\log \delta^{-1})}$.
\end{lemma}
\begin{proof}
To this end, we can use Proposition~\ref{prop:coupling} 
and the Cauchy-Schwarz inequality to obtain
\begin{equation}
\label{eq:display_cs1}
\E \Lambda_{h^\mathcal U, \gamma, \delta}^\mathrm{straight}(v, w)\mathbf 1_{E_{\mathrm{couple}}^c} \leq \P(E_{\mathrm{couple}}^c)\sqrt{\E \big(\Lambda_{h^\mathcal U, \gamma, \delta}^\mathrm{straight}(v, w)\big)^2} \leq \e^{-\Omega_{\mathcal U, \epsilon}(\log \delta^{-4/3})}\sqrt{\E \big(\Lambda_{h^\mathcal U, \gamma, \delta}^\mathrm{straight}(v, w)\big)^2}\,.
\end{equation}
Since $\var(h_\delta^{\mathcal U}(v) - \eta_\delta(v)) = O_{\mathcal U, \epsilon}(1)$ by \eqref{eq:coupling1} and $\var(\eta_\delta(v)) = O(\log \delta^{-1})$ by \eqref{eq:variance}, we get from Fubini
\begin{equation}
\label{eq:display_cs2}
\E \big(\Lambda_{h^\mathcal U, \gamma, \delta}^\mathrm{straight}(v, w)\big)^2 = \int_{[0, 1]^2}\E \e^{\gamma (h_{\delta}^{\mathcal U}(v + |v-w|s) + h_{\delta}^{\mathcal U}(v + |v-w|t))}|v-w|^2dsdt = O_{\mathcal U, \epsilon}(1) \delta^{-O(\gamma^2)}\,.
\end{equation}
Combining \eqref{eq:display_cs1} and \eqref{eq:display_cs2}, we complete the proof of the lemma.
\end{proof}
\begin{lemma}\label{lem-bound-eta-distance}
For all $\delta>0$ we have that $\E D_{\eta, \gamma, \delta}(v, w) = O(\Gamma)\delta^{\Omega(\frac{\gamma^{4/3}}{\log \gamma^{-1}})}$. 
\end{lemma}
\begin{proof}
The proof of the lemma consists of two steps.

\smallskip
\noindent {\bf Step 1: constructing a light path.} Let $n$ be the unique positive integer satisfying $2^{-n - 1} < \delta \leq 2^{-n}$. Choose a rectangle $R$ from either $V_{m_\Gamma}^{\Gamma; (1 - \Gamma^{-1})\iota}$ or $V_{m_\Gamma}^{\Gamma}$ such that $v 
\notin R$ (there is always such a rectangle). For convenience of exposition we assume that $R = V_{m_\Gamma}^{\Gamma; (1 - \Gamma^{-1})\iota}$. Now notice that there exists a sequence of rectangles $R_{1, v}, \ldots R_{K_v, v}$ with sides parallel to the coordinate axes such that:\\
(a) The shorter boundary of $R_{1, v}$ has width  
$2^{-n}$ and has $v$ as one of its endpoints.\\
(b) $R_{K_v, v}$ intersects $\ub V$.\\
(c) The aspect ratio of each $R_{i, v}$ is $\Gamma:1$ for all $i\in [K_v - 1]$.\\
(d) $R_{i, v} \subseteq V$ for all $i \leq K_v - 2$. \\
(e) $R_{i, v} \subseteq R_{i + 1, v}$ for all $i \leq K_v - 2$. Furthermore one of the shorter boundary segments of $R_{i + 1, v}$ is same as one of the longer boundary segments of $R_{i, v}$ for all such $i$.\\
(f) $R_{K_v - 1, v} \cap V$ (also $R_{K_v, v} \cap V$) is a non-degenerate rectangle whose one boundary segment is same as one of the shorter boundary segments of $R_{K_v - 1, v}$ (respectively $R_{K_v, v}$).\\
(g) $R_{K_v - 1, v} \cap V \subseteq R_{K_v, v} \cap V$ and one of the shorter boundary segments of $R_{K_v, v}$ is contained in one of the longer boundary segments of $R_{K_v - 1, v}$.
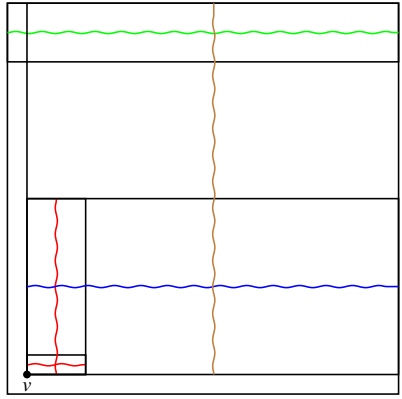
\begin{figure}[!htb]
\centering
\begin{tikzpicture}[semithick, scale = 2.6]
\draw (-1, -1) rectangle (1, 1);

\fill (-0.9, -0.9) circle [radius = 0.02];
\node [scale = 0.7, below] at (-0.9, -0.9) {$v$};


\draw (-1, 0.7) rectangle (1, 1);

\draw [green, style = {decorate, decoration = {snake, amplitude = 0.4}}] (-1, 0.85) -- (1, 0.85);

\draw (-0.9, -0.9) rectangle (-0.6, -0.8);
\draw [red, style = {decorate, decoration = {snake, amplitude = 0.4}}] (-0.9, -0.85) -- (-0.6, -0.85);
\draw (-0.9, -0.9) rectangle (-0.6, 0);
\draw [red, style = {decorate, decoration = {snake, amplitude = 0.4}}] (-0.75, -0.9) -- (-0.75, 0);
\draw (-0.9, -0.9) rectangle (1, 0);
\draw [blue, style = {decorate, decoration = {snake, amplitude = 0.4}}] (-0.9, -0.45) -- (1, -0.45);
\draw (-0.9, 0) -- (-0.9, 1);
\draw [brown, style = {decorate, decoration = {snake, amplitude = 0.4}}] (0.055, -0.9) -- (0.055, 1);
\end{tikzpicture}
\caption{{\bf The rectangles $R_{1, v}, \ldots, R_{K_v, v}$.} In this case $K_v = 4$. The portions of $R_{3, u}$ and $R_{4, v}$ lying outside $V$ have been omitted. The red curved lines indicate $P_{1, v}$ and $P_{2, v}$. The blue and brown curved lines respectively indicate the portions of $P_{3, v}$ and $P_{4, v}$ that lie within $V$. The green curved line indicates $P_\up$.}
\label{fig:pt2bdry}
\end{figure}

Let $P_v$ be the shorter boundary segment of $R_{1, v}$ 
containing $v$. By Properties~(a), (b), (d), (e), (f) and (g), we see that given any choice of a crossing $P_{i, v}$ through $R_{i, v}$ and a crossing $P_\up$ through $V_{m_\Gamma}^{\Gamma; (1 - \Gamma^{-1})\iota}$, the union of $P_v, P_{1, v}, \ldots, P_{K_v, v}$ contains a path between $v$ and $\ub V$ which is contained in $V$ and also intersects $P_\up$ (see Figure~\ref{fig:pt2bdry}). 
Similarly if $w \notin [0, \Gamma^{-1}] \times [0, 1]$ (otherwise we can work with $[1 - \Gamma^{-1}, 1] \times [0, 1]$), we can construct a sequence of rectangles $R_{1, w}, \ldots, R_{K_w, w}$ such that given any choice of a crossing $P_{i, w}$ through $R_{i, w}$ and a crossing $P_\lt$ through $[0, \Gamma^{-1}] \times [0, 1]$, the union of $P_w, P_{1, w}, \ldots, P_{K_w, w}$ contains a path between $w$ and $\lb V$ which is contained in $V$ 
and intersects $P_\lt$. Since $P_\up$ and $P_\lt$ always intersect, it follows that the union of $P_v, P_{1, v}, \ldots, P_{K_v, v}, P_\up, P_\lt, P_w, P_{1, w}, \ldots, 
P_{K_w, w}$ contains a path between $v$ and $w$. Therefore it remains to choose these crossings so that their combined weight is small, as in the next step.

\smallskip

\noindent {\bf Step 2: bounding the weight}. Note that the rectangles $[0, \Gamma^{-1}] \times [0, 1]$ and $V_{m_\Gamma}^{\Gamma; (1 - \Gamma^{-1})\iota}$ have aspect ratio $\Gamma:1$ and width $2^{-m_\Gamma}$. So we can use $\mathcal A_{n - m_\Gamma}$ to construct $P_\lt$ and $P_\up$. Recall from Section~\ref{subsec:strat2} that we only use the processes $\{\eta_{2^{-\ell'}}^{2^{- \ell }}\}_{ \ell  < \ell' \leq n}$ to construct a crossing at scale 
$n$ using $\mathcal A_{n - \ell}$. Thus by Lemma~\ref{lem-upper-bound-lambda}, the independent increment and the scaling property of $\eta$ and \eqref{eq:variance}, we can bound the expected weights of both $P_\up$ and $P_\lt$  from above as follows (recall that we used similar arguments for deriving \eqref{eq:cost_compute_strat1_3}):
\begin{equation}
\label{eq:display_main2**}
 O(\Gamma)2^{-(n - m_\Gamma)\Omega(\frac{\gamma^{4/3}}{\log \gamma^{-1}})}  2^{-m_\Gamma}  2^{O(\gamma^2)m_\Gamma} = O(\Gamma)\delta^{\Omega(\frac{\gamma^{4/3}}{\log \gamma^{-1}})}2^{-\Omega(m_\Gamma)}\,.
\end{equation}
 Now fix $u \in 
\{v, w\}$. Properties~(c) and (e) imply that for all $i \in [K_u-1]$, $R_{i, u}$ is a copy of $V_{\ell_i}^\Gamma$ where $\ell_i \coloneqq n - m_\Gamma(i-1) \in [n]$. Hence we can choose $P_{i, u}$ to be 
the crossing constructed by $\mathcal A_{n - \ell_i}$. Similarly by Lemma~\ref{lem-upper-bound-lambda}, we can bound the expected weight of each $P_{i, u}$ (with respect to $\eta_\delta$) by
\begin{equation}
\label{eq:display_main2*}
\Gamma 2^{-(n - \ell_{i})\Omega(\frac{\gamma^{4/3}}{\log \gamma^{-1}})}2^{-\ell_{i}} 2^{O(\gamma^2) \ell_{i}} =O(\Gamma)\delta^{\Omega(\frac{\gamma^{4/3}}{\log \gamma^{-1}})}2^{-\Omega(\ell_i)}= O(\Gamma)\delta^{\Omega(\frac{\gamma^{4/3}}{\log \gamma^{-1}})}2^{-\Omega(n - m_\Gamma(i-1))}\,.
\end{equation}
 The width of $R_{K_u, u}$, however, may not be an integer power of 2. So suppose, without loss of generality, that $R_{K_u, u} = V_{\ell_u}^{\Gamma; u'}$ 
for some $\ell_u \in [0, n]$ and $u' \in \R^2$. From the construction that we used to derive \eqref{eq:induct} it is clear that it can be adapted to construct a crossing through $V_{\lceil \ell_{u} \rceil}^{2\Gamma; u'}$ that uses only the processes $\{\eta_{2^{-\ell'}}^{2^{- \lceil \ell_{u} \rceil }}\}_{ \lceil \ell_{u} \rceil  < \ell' \leq n}$ as input and has expected weight (computed with respect to $\eta_{2^{-n}}^{2^{-\lceil \ell_{u} \rceil}}$) bounded by
\begin{equation*}
\label{eq:induct2}
2\Gamma (1 - \Omega(\gamma^{4/3}))^{\lfloor (n - \lceil \ell_{u} \rceil) / 2m_\Gamma\rfloor}\,
\end{equation*}
where $\lceil \ell \rceil$ is the smallest integer $\geq \ell$. This crossing contains a crossing through 
$R_{K_{u}, u}$ which we choose as $P_{K_{u}, u}$. Hence like in \eqref{eq:display_main2*}, its expected weight (with respect to $\eta_\delta$) is bounded by
\begin{equation}
\label{eq:display_main2***}
2\Gamma 2^{-(n - \lceil \ell_{u}\rceil)\Omega(\frac{\gamma^{4/3}}{\log \gamma^{-1}})}2^{-\lceil\ell_{u}\rceil} 2^{O(\gamma^2) \lceil \ell_{u} \rceil} = O(\Gamma)\delta^{\Omega(\frac{\gamma^{4/3}}{\log \gamma^{-1}})}2^{-\Omega(\lceil \ell_u\rceil)}\,.
\end{equation}

Summing all the weights from displays \eqref{eq:display_main2**}, \eqref{eq:display_main2*} and \eqref{eq:display_main2***} over $i\in [K_u]$ and $u\in \{v, w\}$ we conclude the proof of the lemma.
\end{proof}

\section{Proof of Theorem~\ref{thm:LQG}}
\label{sec:LQG}
In order to prove Theorem~\ref{thm:LQG}, we need to find a (short) sequence of $(M_\gamma^{\mathcal U}, \delta)$-balls (see the discussion following the statement of Theorem~\ref{thm:main}) whose union contains a path between two endpoints. 
To this end, a natural attempt is to construct a sequence of $(M_\gamma^{\mathcal U}, \delta)$-balls that covers the geodesic of LFPP with respect to some appropriately chosen $\delta'>\delta$. For every $\delta'$-sized piece in the geodesic, we will use a simple and (very) suboptimal strategy to cover it with  $(M_\gamma^{\mathcal U}, 
\delta)$-balls. While implementing this proof idea, one technical complication is that  we wish to construct $(M_\gamma^{\mathcal U}, \delta)$-balls to cover a piece of the geodesic as if we were covering a fixed line segment (as illustrated in Proposition~\ref{prop:second_moment}). To this end, it requires us to ``separate the randomness'' such that the construction of the $(M_\gamma^{\mathcal U}, \delta)$-balls at a location on the geodesic is conditionally independent of the geodesic given the LFPP weight at that 
location. To address this, we introduce a lattice approximation of LFPP, reveal its geodesic, and then construct the $(M_\gamma^{\mathcal U}, \delta)$-balls along a path that is near this geodesic (see Figure~\ref{fig:LQG_covering} for an illustration) --- then the separation of randomness follows from Markov field property for the circle average process. 
  An ancillary benefit of introducing the lattice version of LFPP is that we will reference to it in Section~\ref{sec-discrete}. 
  
We now describe our proof. For that we need some notations which will appear repeatedly throughout this section. Let us denote by $\D$ the open unit ball (in $\R^2$) centered at the 
origin. For any closed ball $B$, we let $B^*$ denote the concentric \textbf{closed} ball with doubled radius, and $B^{**}$ denote the concentric \textbf{open} ball with radius four times the radius of $B$ (we choose $B^{**}$ to be an open ball since harmonicity mentioned below fails on the boundary of $B^{**}$). 

Now consider a (continuum) GFF $h^{\mathfrak D}$ on a bounded domain $\mathfrak D$ with Dirichlet boundary 
condition (in later applications, we may choose $\mathfrak D$ to be $\mathcal U$ or $\D$). {If $B$ is a closed ball such that $\overline {B^{**}}$, i.e., the closure of $B^{**}$, is contained in $\mathfrak D$}, then by the \emph{Markov field property} (see \cite[Section~2.6]{S07} or \cite[Theorem~1.17 ]{berestycki16}) of GFF we can write
$h^{\mathfrak D} = h^{\mathfrak D, B^{**}} + \varphi^{\mathfrak D, B^{**}}\,,$
where 
\begin{itemize}
\item $h^{\mathfrak D, B^{**}}$ is a GFF on $B^{**}$ with Dirichlet boundary condition ($=$ 0 outside $B^{**}$).

\item  $\varphi^{\mathfrak D, B^{**}}$ is harmonic on $B^{**}$.

\item $h^{\mathfrak D, B^{**}}, \varphi^{\mathfrak D, B^{**}}$ are independent. 
\end{itemize}
This decomposition has a useful consequence for us as follows. Since $\varphi^{\mathfrak D, B^{**}}$ is harmonic on $B^{**}$, we get
\begin{equation}
\label{eq:decomposition}
h^{\mathfrak D}_\delta(v) = h^{\mathfrak D, B^{**}}_\delta(v) + \varphi^{\mathfrak D, B^{**}}(v)\,
\end{equation}
for all $v \in B^*$ and $\delta \in (0, 
r]$ where $r$ is the radius of $B$. The process $\{h_\delta^{\mathfrak D, B^{**}}(v): v \in B^*, 0 < \delta \leq r\}$ is independent with $\{\varphi^{\mathfrak D, B^{**}}(v): v \in B^*\}$ and also with $\{h_{\delta'}^{\mathfrak D}(w): w \in \mathfrak D \setminus B^{**}, \delta' < d_{\ell_2}(w, 
B^{**})\}$. The following lemma shows that the process $\varphi^{\mathfrak D, B^{**}}$ is smooth on $B^*$.
\begin{lemma}
\label{lem:smoothness2}
Let $B$ be a closed ball centered at $c_B$ with radius  $r$  such that $\overline {B^{**}}\subseteq \mathfrak D$. 
Then we have for all $v, w \in B^*$
$$\var(\varphi^{\mathfrak D, B^{**}}(v) - \varphi^{\mathfrak D, B^{**}}(w)) = O(|v - w|/r) \mbox{ and }\sup_{v \in B^*}\var(h_r^{\mathfrak D}(c_B) - \varphi^{\mathfrak D, B^{**}}(v)) = O(1)\,.$$
\end{lemma}
\begin{proof}
Since $h_r^{B^{**}}$ and $\varphi^{\mathfrak D, B^{**}}$ are independent, we get from \eqref{eq:decomposition}
$$\var(\varphi^{\mathfrak D, B^{**}}(v) - \varphi^{\mathfrak D, B^{**}}(w)) \leq \var(h_r^{\mathfrak D}(v) - h_r^{\mathfrak D}(w))\,,$$
for all $v, w \in B^*$. But we know (see the proof of \cite[Proposition~2.1]{HMPeres2010})
\begin{equation}
\label{eq:smoothness}
\var(h_r^{\mathfrak D}(v) - h_r^{\mathfrak D}(w)) = O(|v - w|/r)\,,
\end{equation}
which gives the required bound on $\var(\varphi^{\mathfrak D, B^{**}}(v) - \varphi^{\mathfrak D, B^{**}}(w))$. For the second bound, notice that
$$\var(h_r^{\mathfrak D}(c_B) - \varphi^{\mathfrak D, B^{**}}(c_B)) = \var(h^{\mathfrak D, B^{**}}_r(c_B))\,.$$
Thus it suffices to prove $\var(h^{\mathfrak D, B^{**}}_r(c_B)) = O(1)$ in view of the bound on $\var(\varphi^{\mathfrak D, B^{**}}(v) - \varphi^{\mathfrak D, B^{**}}(w))$. But $h^{\mathfrak D, B^{**}}_r(c_B)$ is identically distributed as $h^{\D}_{0.25}(\mathbf{0})$ by the scale and translation invariance of GFF and hence $\var(h^{\mathfrak D, B^{**}}_r(c_B))$ is a finite constant (see the discussions in \cite[Section~2.1]{S07} and \cite[Theorem~1.9]{berestycki16}). 
\end{proof}

Now consider a Radon measure $\mu$ on $\mathfrak D$ and 
some $\delta > 0$. We call a closed Euclidean ball $B \subseteq \mathfrak D$ with a rational center as a \emph{$(\mu, \delta)$-ball} if $\mu(B) \leq 
\delta^2$. For any compact $A \subseteq \mathfrak D$, let $N(\mu, \delta, A)$ denote the minimum number of $(\mu, \delta)$-balls required to cover $A$. Our next proposition provides a crude upper bound on the \emph{second moment} of $N(M_\gamma^{\D}, \delta, A)$ (see \eqref{eq-limit-LQG} for the definition of $M_\gamma^\D$) when $A$ is a 
segment inside $\D$. We remark that the KPZ relation proved in \cite{DS11} gives the sharp exponent for a typical $N(M_\gamma^{\D}, \delta, A)$, and our analysis resembles that of \cite{DS11}.
	
\begin{proposition}
\label{prop:second_moment}
Let $L$ denote the straight line segment 
joining $-\bm{0.25}$ and $\bm{0.25}$. 
For any $\delta \in (0, 1)$, we can find a collection of $(M_\gamma^\D, \delta)$-balls $\mathscr S(M_\gamma^{\D}, \delta, L)$ such that
\begin{enumerate}[(a)]
\item Balls in $\mathscr S(M_\gamma^{\D}, \delta, L)$ cover $L$.
\item All the balls in $\mathscr S(M_\gamma^{\D}, \delta, L)$ are contained in $0.25\overline \D$. 
\item For some positive absolute constant $C$, we have
$\E(|\mathscr S(M_\gamma^{\D}, \delta, L)|^2) = O_\gamma(\delta^{-2 - C\gamma})$.
\end{enumerate}
\end{proposition}
\begin{proof}
For each $k \in \N$, let $\ball_k$ denote the collection of all (closed) balls of radius $2^{-k - 1}$ whose centers lie in the set $\{-\tfrac{\bm{1}}{\bm{4}} + \bm{2^{-k-1}} 
+ j\bm{2^{-k}}: j \in [0, 2^{k-1} - 1] \cap \Z\}$. The balls in $\ball_k$ are nested in a natural way. In particular any ball $B$ in $\ball_k$ has a unique \emph{parent} $B^{k'}$ in $\ball_{k'}$ (where $k' \leq k$) such that $B \subseteq B^{k'}$. 
 We include a $(M_\gamma^{\D}, \delta)$-ball $B\in \ball_k$ in $\mathscr S(M_\gamma^{\D}, \delta, L)$ if the $M_\gamma^\D$ volume of the \emph{most recent} parent of 
$B$ is bigger than $\delta^2$. Since the measure $M_\gamma^\D$ is almost surely finite and has no atoms (see \cite{DS11} and \cite[Theorem~2.1]{berestycki16}), it follows that $\mathscr S(M_\gamma^{\D}, \delta, \mathcal 
L)$ satisfies Condition~(a) (and obviously (b)). It also follows from the construction that
\begin{equation}
\label{eq:second_moment1*}
|\mathscr S(M_\gamma^{\D}, \delta, L)| \leq 2\delta^{-1 - C'\gamma} + \sum_{k > (1 + C'\gamma)\log_2 \delta^{-1}} |\ball(M_\gamma^\D, k, \delta)|\,,
\end{equation}
where $\ball(M_\gamma^\D, k, \delta)$ denotes the collection of balls in $\ball_k$ with 
$M_\gamma^\D$ volume $> \delta^2$, and $C' > 1$ is a constant to be specified 
later.  By a simple bound we get that
\begin{eqnarray}
\label{eq:second_moment1**}
\Big(\sum_{k > (1 + C'\gamma)\log_2 \delta^{-1}} |\ball(M_\gamma^\D, k, \delta)|\Big)^2 &\leq& \sum_{k > (1 + C'\gamma)\log_2 \delta^{-1}} \sum_{B \in \ball_k}2\sum_{k' \leq k}\sum_{B' \in \ball_{k'}} \mathbf 1_{\{M_\gamma^\D( B) > \delta^2, M_\gamma^\D( B') > \delta^2\}}\nonumber \\
&\leq& \sum_{k > (1 + C'\gamma)\log_2 \delta^{-1}} \sum_{B \in \ball_k}2\sum_{k' \leq k}\sum_{B' \in \ball_{k'}} \mathbf 1_{\{M_\gamma^\D( B) > \delta^2\}}\nonumber \\
&\leq& \sum_{k > (1 + C'\gamma)\log_2 \delta^{-1}} 2^{k + 1}\sum_{B \in \ball_k}\mathbf 1_{\{M_\gamma^\D( B) > \delta^2\}}\,.
\end{eqnarray}
Next we compute the probability that any given ball $ B \coloneqq c_B + 2^{-k}\overline \D$ in $\ball_k$ has $M_\gamma^\D$ 
volume at least $\delta^2$. To this end, denote the LQG measure on $B^{**}$ obtained from $h^{\D, B^{**}}$ as $M_\gamma^{\D, B^{**}}$ (recall the decomposition in \eqref{eq:decomposition}). More precisely, $M_\gamma^{\D, B^{**}}$ is the almost sure weak limit of the sequence of measures $M_{\gamma, n}^{\D, B^{**}}$ given by
\begin{equation}\label{eq-limit-LQG2}
M_{\gamma, n}^{\D, B^{**}} = \e^{\gamma h_{2^{-n}}^{\D, B^{**}}(z)}2^{- n\gamma^2/2}\sigma(dz)\,
\end{equation}
(see \eqref{eq-limit-LQG}). Since $M_\gamma^\D$ and $M_\gamma^{\D, B^{**}}$ are the weak limits of $M_{\gamma, n}^{\D}$'s and $M_{\gamma, n}^{\D, B^{**}}$'s respectively, we get from the decomposition in \eqref{eq:decomposition}
\begin{equation}
\label{eq:second_moment1}
M_\gamma^\D( B) \leq 4^{-k}2^{-\frac{k\gamma^2}{2}} \times \e^{\gamma h_{2^{-k}}^\D(c_B)}\times \e^{\gamma \max_{v \in  B^*}(\varphi^{\D, B^{**}}(v) - h_{2^{-k}}^\D(c_B))} \times 4^k M_{\gamma}^{\D, B^{**}}( B)\,.
\end{equation}
From the scale and translation invariance property of GFF it follows that $4^{k}M_{\gamma}^{B^{**},\D}(B)$ is 
identically distributed as $M_\gamma^\D(\tfrac{1}{4}\D)$. Also observe that if each of the last three factors in \eqref{eq:second_moment1} is smaller than $4^{k/3}\delta^{2/3} = \e^{\gamma\frac{2}{3\gamma}\log(\delta 2^k)}$, we will have $M_\gamma^{\D}(B) \leq \delta^2$. Thus, we obtain that
\begin{equation}
\label{eq:second_moment2}
\P(M_\gamma^\D(B) > \delta^2) \leq \P\big(h_{2^{-k}}^\D(c_B) \geq \frac{2}{3\gamma}\log(\delta 2^k)\big) + \P\big(\Max_{B, \D; k} \geq \frac{2}{3\gamma}\log(\delta 2^k)\big) + \P(M_\gamma^\D(\tfrac{1}{4} \D) \geq 4^{k/3}\delta^{2/3})\,,
\end{equation}
where $\Max_{B, \D; k} = \max_{v \in B^*}(\varphi^{\D, B^{**}}(v) - h_{2^{-k}}^\D(c_B))$.
Since $\var(h_{2^{-k}}^\D(c_B)) = k\log 2 +  O(1)$ (see, e.g., \eqref{eq:variance} and \eqref{eq:coupling1}) and $\delta^{1 + C'\gamma} > 2^{-k}$, the first term on the right hand side of \eqref{eq:second_moment2} can be bounded as
$$\P\big(h_{2^{-k}}^{\D}(c_B) \geq \frac{2}{3\gamma}\log(\delta 2^k)\big) \leq 
\P\Big(Z \geq \frac{C'k\log2}{3\sqrt{k\log 2 + O(1)}}\Big) \leq \e^{-C'^2\Omega(k)\log2} = 2^{-C'^2\Omega(k)}\,,$$
where $Z$ is a standard Gaussian variable. Thus we can choose $C'$ big enough so that the 
bound above becomes $< 2^{-10k}$. From Lemma~\ref{lem:smoothness2} we know that
$\max_{v \in B^*}\var(\varphi^{\D, B^{**}}(v) - h_{2^{-k}}^\D(c_B)) = O(1)$ and that $\var(\varphi^{\D, B^{**}}(v) - \varphi^{\D, B^{**}}(w)) \leq 
O\big(\tfrac{|v - w|}{2^{-k}}\big)$ for all $v, w \in B^*$. Hence by Lemma~\ref{lem:gaussian_tail} we get that
\begin{eqnarray*}
\P\big(\Max_{B, \D; k} \geq \frac{2}{3\gamma}\log(\delta 2^k)\big) &\leq& \P\big(\Max_{B, \D; k} \geq \frac{2}{3\gamma}\log(2^{-\frac{k}{1 + C'\gamma}} 2^k)\big) = \P\big(\Max_{B, \D; k} \geq \frac{2C'}{3 + 3C'\gamma}\log 2^k\big)\\
&=& \e^{-\Omega(k^2)} = O(2^{-10k})\,.
\end{eqnarray*}
The only remaining term is $\P(M_\gamma^\D(\tfrac{1}{4} \D) \geq 4^{k/3}\delta^{2/3})$. In order to bound this probability we will use the fact that $\E \big(M_\gamma^\D(\tfrac{1}{4} \D)\big)^4 < \infty$ for small $\gamma$ (see \cite{Kahane85} and also \cite[Theorem~2.11 and Theorem~5.5]{RV14}). Hence by Markov's inequality
$$\P(M_\gamma^\D(\tfrac{1}{4} \D) \geq 4^{k/3}\delta^{2/3})= O_\gamma(\delta^{-8/3})2^{-8k/3}\,.$$
Plugging the last three estimates into \eqref{eq:second_moment2} we get
$$\P(M_\gamma^\D(B) > \delta^2) \leq O_\gamma(\delta^{-8/3})2^{-8k/3}\,.$$
Taking expectation on both sides in \eqref{eq:second_moment1**} and using the bound above one gets:
\begin{eqnarray*}
&&\E\Big(\sum_{k > (1 + C'\gamma)\log_2 \delta^{-1}} |\ball(M_\gamma, k, \delta)|\Big)^2 \leq \sum_{k > (1 + C'\gamma)\log_2 \delta^{-1}} 2^{2k + 1}O_\gamma(\delta^{-8/3})2^{-8k/3} \nonumber \\
&=& O_\gamma(\delta^{-8/3})\sum_{k > (1 + C'\gamma)\log_2 \delta^{-1}}2^{-2k/3} = O_\gamma(\delta^{-8/3})\delta^{2/3(1 + \gamma C')} = O_\gamma(\delta^{-2 + 2C'\gamma/3})\,.
\end{eqnarray*}
The lemma follows from this bound and \eqref{eq:second_moment1*} for $C =2C'$.
\end{proof}
The proof of Proposition~\ref{prop:second_moment} can be easily adapted to accommodate the following set-up. 
\begin{proposition}
\label{cor:second_moment}
Let $S \subseteq V$ be a closed square of length $2^{-k}$ whose vertices lie in $2^{-k} \Z^2$. Then for any $\delta \in (0, 2^{-k})$ we have
$$\E N(M_\gamma^{\mathcal U}, \delta, S)^2 = O_{\gamma, \mathcal U, \epsilon}((2^k\delta)^{-4 - O(\gamma)}2^{kO(\gamma)})\,.$$
\end{proposition}
The other ingredient we need for the proof of Theorem~\ref{thm:LQG} is the lattice approximation of LFPP. 
Toward this end fix any integer $k > 4$ and define the set $\mathbb 
L_k$ as $2^{-(k - 3)} \Z^2 \cap V$. We can treat $\mathbb L_k$ as a subgraph of the lattice $2^{-(k - 3)} \Z^2$. The centers of the squares in $\mathbb L_k$ (i.e. the squares of side length $2^{-(k -3)}$ with vertices in $\mathbb L_k$) form another set $\mathbb L_k^\star \subseteq V$ which will be treated as the dual graph of $\mathbb L_k$. 
We can define a LFPP distance $D_{\gamma, k}^\star(\cdot, \cdot)$ on $\mathbb L^\star_k$ as follows:
\begin{equation}\label{eq-LFPP-def2}
D_{\gamma, k}^\star(v, w)=\min_{P}\sum_{u \in \pi}e^{\gamma(1 + C\gamma)h_{2^{-k}}^{\mathcal U}(u)/2} \mbox{ for } v, w\in \mathbb L_k^\star,
\end{equation}
where $P$ ranges over all paths in $\mathbb L_k^\star$ connecting $v$ and $w$, and $C$ is the same absolute constant as in 
Proposition~\ref{prop:second_moment}. Our next lemma is a consequence of Theorem~\ref{thm:main}.
\begin{lemma}
\label{lem:star_lfpp}
There exists an absolute constant $\gamma_0>0$ such that for all $0<\gamma \leq \gamma_0$ and  any integer $k > 4$
$$\max_{v, w\in  \mathbb L_k^\star}\E D_{\gamma, k}^\star(v, w) = O_{\gamma, \mathcal U, \epsilon}(1)2^{k(1 - \Omega(\gamma^{4/3} / \log \gamma^{-1}))}\,.$$
\end{lemma}
\begin{proof}
We assume without loss of generality that $\epsilon \geq 2^{-k}$. For $v, w \in \mathbb L_k^\star$, we see from Theorem~\ref{thm:main} that there exists a (random) simple, piecewise smooth path 
$P_{k, \gamma}(v, w)$ such that
\begin{equation}
\label{eq:star_lfpp1}
\E \big(\int_{P_{k, \gamma}(v, w)} \e^{\gamma(1 + C\gamma) h_{2^{-k}}^{\mathcal U}(z)/2}|dz|\big) = O_{\gamma, \mathcal U, \epsilon}(2^{-k\Omega(\gamma^{4/3} / \log \gamma^{-1})})\,.
\end{equation}
In order to create a lattice path (i.e. in $\mathbb L_k^\star$) between $v$ and $w$ from $P_{k, \gamma}(v, w)$ we follow a simple procedure. Let $p_{k, \gamma; 0} = v$. For $i\geq 1$, we employ the following procedure inductively: start at $p_{k, \gamma, i-1}\in P_{k, \gamma}(v, w)$ and continue along the path $P_{k, \gamma}(v, w)$ until it exits the smallest square $S_i$ satisfying (a) $p_{k, \gamma; i} \in S_i$, (b) $d_{\ell_2}({p_{k, \gamma; i}}, \partial S_i) \geq 2^{-(k - 3)}\cdot 2^{-1}$ and (c) the vertices of $S_i$ are in $\mathbb L_k^\star$. We stop the procedure until 
we reach $w$. At the end of this procedure we will get a sequence of squares $S_1, S_2, \cdots,$ where each $S_i$ has side length at most $2\cdot2^{-(k - 3)}$ and the union of $S_i \cap \mathbb L_k^\star$'s 
contains a lattice path $P_{k, \gamma}^\star(v, w)$ between $v$ and $w$. Now let us recall from \eqref{eq:smoothness} that
\begin{equation*}
\label{eq:continuity}
\var (h_{2^{-k}}^{\mathcal U}(z) - h_{2^{-k}}^{\mathcal U}(z')) = O\Big(\frac{|z - z'|}{2^{-k}}\Big)\,
\end{equation*}
for all $z, z' \in V$. Then by Lemma~\ref{lem:gaussian_tail} there exists a positive absolute constant $C_1$ such that
\begin{equation}
\label{eq:star_lfpp2}
\P\Big(\max_{z, z' \in V, |z - z'| \leq 2^{-(k - 4)}}(h_{2^{-k}}^{\mathcal U}(z) - h_{2^{-k}}^{\mathcal U}(z')) \geq C_1\sqrt{k} + x\Big) = \e^{-\Omega(x^2)}\,
\end{equation}
for all $x \geq 0$ (subdivide $V$ into a collection $\mathfrak R$ of squares of length $2^{-k}$ and on each $R \in \mathfrak R$ consider the process $Y_R \coloneqq \{h_{2^{-k}}^{\mathcal U}(z) - h_{2^{-k}}^{\mathcal U}(c_R): z \in R\}$ where $c_R$ is 
the center of $R$). Now define an event $E_k$ as
$$E_k = \big\{\max_{z, z' \in V, |z - z'| \leq 2^{-(k - 4)}}(h_{2^{-k}}^{\mathcal U}(z) - h_{2^{-k}}^{\mathcal U}(z')) \leq   (C_1 + 1)\sqrt{k}\big\}\,.$$
Since the Euclidean length of $P_{k, \gamma}(v, w)$ inside each $S_i$ is $\Omega(2^{-k})$ (by Condition (b) in our construction) and that $|S_i \cap \mathbb L_k^\star| = O(1)$, we get from \eqref{eq:star_lfpp1} that
\begin{eqnarray*}
\E\Big( \big(\sum_{z \in P_{k, \gamma}^\star(v, w)}\e^{\gamma(1 + C\gamma) h_{2^{-k}}^{\mathcal U}(z)}\big) \mathbf 1_{E_k}\Big) &=& O(2^k) \e^{(C_1 + 1)\sqrt{k}}O_{\gamma, \mathcal U, \epsilon}(2^{-k\Omega(\gamma^{4/3} / \log \gamma^{-1})}) \nonumber \\
&=& O_{\gamma, \mathcal U, \epsilon}(2^{k(1 -\Omega(\gamma^{4/3} / \log \gamma^{-1}))})\,.
\end{eqnarray*}
In addition, from \eqref{eq:star_lfpp2} and the Cauchy-Schwarz inequality (similar to \eqref{eq:display_cs1} and \eqref{eq:display_cs2}) we obtain
\begin{eqnarray*}
\E\Big( \big(\sum_{z \in P_k(v, w)}\e^{\gamma(1 + C\gamma) h_{2^{-k}}^{\mathcal U}(z)/2}\big) \mathbf 1_{E_k^c}\Big) = O_{\mathcal U, \epsilon}(2^k)2^{-k(\Omega_{\mathcal U, \epsilon}(1) - O(\gamma^2))} = O_{\mathcal U, \epsilon}\big(2^{k(1 - \Omega_{\mathcal U, \epsilon}(1))}\big)\,,
\end{eqnarray*}
where $P_k(v, w)$ is a path between $v$ and $w$ in the graph $\mathbb L_k^\star$ which minimizes the graph distance. Choosing $P_{k, \gamma}^\star(v, w)$ and $P_k(v, w)$ on $E_k$ and $E_k^c$ respectively as a lattice path between $v$ and $w$, we get the desired bound on $\E D^\star_{\gamma, k}(v, w)$ from the previous two displays.
\end{proof}
For $v, w\in V$, we call the path minimizing $D_{\gamma, k}^\star(v, w)$ as the \emph{$(\gamma, k)$-geodesic} between $v$ and $w$. Now we pick squares $[v]_k$ and 
$[w]_k$ in $\mathbb L_k$ that contain $v$ and $w$ 
respectively (when there are several choices, we will pick an arbitrary but fixed one). Define $\S^\star(k, v, w)$ as the collection of squares in $\mathbb L_k$ whose centers lie along the $(\gamma, k)$-geodesic between $c([v]_k)$ and $c([w]_k)$ in $\mathbb 
L_k^\star$. Here $c([v]_k)$ and $c([w]_k)$ are the centers 
of the squares $[v]_k$ and $[w]_k$ respectively. Thus $\S^\star(k, v, w)$ is  a sequence of neighboring squares connecting $v$ and $w$ (see Figure~\ref{fig:LQG_covering}). An important observation is the following. 
\begin{observation}
\label{observ:covering}
The Euclidean distance between the boundary of any square in $\S^\star(k, v, w)$ and $\mathbb L_k^\star$ is at least $2^{-(k - 2)}$.
\end{observation}
\begin{proof}[Proof of Theorem~\ref{thm:LQG}]
Gven a $\delta \in (0, 1)$ and $v, w \in V$, we will construct a collection of $(M_\gamma^{\mathcal U}, \delta)$-balls $\S(\delta, v, w)$ such that the union of these balls contains a path between 
$v$ and $w$. Then it suffices to show 
\begin{equation}\label{eq-to-show-LGD}
\E|\S(\delta, v, w)| = O_{\gamma, \mathcal U, \epsilon}(1)\delta^{-1 + \Omega(\tfrac{\gamma^{4/3}}{\log \gamma^{-1}})}
\end{equation}
for proving Theorem~\ref{thm:LQG}. Thus, the proof of Theorem~\ref{thm:LQG} is naturally divided into two steps.

\smallskip

\noindent {\bf Step 1. Constructing $\S(\delta, v, w)$.} Given $S \in \S^\star(k, v, w)$ (recall the definition of $\S^\star(k, v, w)$ from just above Observation~\ref{observ:covering}) with $S \neq [v]_k\mbox{ or }[w]_k$, divide each boundary segment of $S$ into 16 segments (with 
disjoint interiors) of length $2^{-(k + 1)}$. For any such segment $T$, let $B_T$ denote the closed ball of 
radius $2^{-(k + 2)}$ centered at the midpoint of $T$. Thus $T$ is a diameter segment of $B_{T}$. Cover $T$ with the minimum possible number of $(M_{\gamma}^{\mathcal U, B_T^{**}}, \delta\e^{-\gamma h_{2^{-k}}^{\mathcal U}(c(S)) / 2}\e^{-C_2\gamma\sqrt{k \log 2}})$-balls contained in $B_T$ where $M_{\gamma}^{\mathcal U, B_T^{**}}$ is the LQG measure on $B_T^{**}$ constructed from 
$h^{\mathcal U, B_T^{**}}$ (see \eqref{eq-limit-LQG2}), $c(S)$ is the center of $S$ and $C_2$ 
is an absolute constant to be specified in \eqref{eq-def-C-2} later. The reason behind the choice of $\delta\e^{-\gamma h_{2^{-k}}^{\mathcal U}(c(S)) / 2}\e^{-C_2\gamma\sqrt{k \log 2}}$ will become clear once we construct $\S(\delta, v, w)$ (see \eqref{eq-become-clear}). Denote the collection of all such balls from all the segments of $\partial S$ as 
$\S(S, \delta)$. If $S = [v]_k$ or $[w]_k$, we simply cover $S$ with minimum possible number of $(M_\gamma^{\mathcal U}, \delta)$-balls and include them in $\S(S, \delta)$. 
Finally define
$$\S^{\star\star}(k, \delta, v, w) = \bigcup_{S \in \S^{\star}(k, v, w)} \S(S, \delta)\,.$$
It is clear that the union of balls in $\S^{\star\star}(k, \delta, v, w)$ contains a path between $v$ and $w$. Figure~\ref{fig:LQG_covering} gives an illustration of this construction.
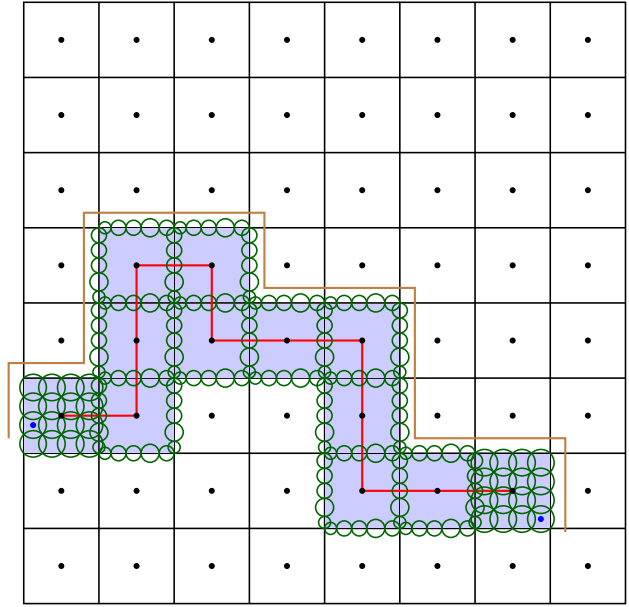
\begin{figure}[!htb]
	\centering
	\begin{tikzpicture}[semithick, scale = 2]
	\draw (0, 0) rectangle (4, 4);
	\foreach \x in {0.5, 1,  ..., 3.5}{
		\draw (\x, 0) -- (\x, 4);
		}
	\foreach \y in {0.5, 1,  ..., 3.5}{
			\draw (0, \y) -- (4, \y);
	}

	\foreach \x/\y in {0.25 / 1.25, 0.75 / 1.25, 0.75/1.75, 0.75 / 2.25, 1.25 / 2.25, 1.25/ 1.75, 1.75/1.75, 2.25/1.75, 2.25/1.25, 2.25/0.75, 2.75/0.75, 3.25/0.75}{
		\fill [blue!20] (\x - 0.245, \y - 0.245) rectangle (\x + 0.245, \y + 0.245);
		}

	\draw [red, thick] (0.25, 1.25) -- (0.75, 1.25) -- (0.75, 2.25) -- (1.25, 2.25) -- (1.25, 1.75) -- (2.25, 1.75) -- (2.25, 0.75) -- (3.25, 0.75);	
	
	\draw[green!40!black] (0.54, 1) circle [radius = 0.04];
	\draw[green!40!black] (0.63, 1) circle [radius = 0.05];
	\draw[green!40!black] (0.73, 1) circle [radius = 0.05];
	\draw[green!40!black] (0.84, 1) circle [radius = 0.06];
	\draw[green!40!black] (0.95, 1) circle [radius = 0.05];
	
	\draw[green!40!black] (0.5, 1.04) circle [radius = 0.04];
	\draw[green!40!black] (0.5, 1.14) circle [radius = 0.06];
	\draw[green!40!black] (0.5, 1.25) circle [radius = 0.05];
	\draw[green!40!black] (0.5, 1.35) circle [radius = 0.05];
	\draw[green!40!black] (0.5, 1.45) circle [radius = 0.05];
	
	\draw[green!40!black] (0.54, 1.5) circle [radius = 0.04];
	\draw[green!40!black] (0.63, 1.5) circle [radius = 0.05];
	\draw[green!40!black] (0.73, 1.5) circle [radius = 0.05];
	\draw[green!40!black] (0.84, 1.5) circle [radius = 0.06];
	\draw[green!40!black] (0.95, 1.5) circle [radius = 0.05];
	
	\draw[green!40!black] (1, 1.04) circle [radius = 0.04];
	\draw[green!40!black] (1, 1.14) circle [radius = 0.06];
	\draw[green!40!black] (1, 1.25) circle [radius = 0.05];
	\draw[green!40!black] (1, 1.35) circle [radius = 0.05];
	\draw[green!40!black] (1, 1.45) circle [radius = 0.05];

	\draw[green!40!black] (0.5, 1.54) circle [radius = 0.04];
	\draw[green!40!black] (0.5, 1.64) circle [radius = 0.06];
	\draw[green!40!black] (0.5, 1.75) circle [radius = 0.05];
	\draw[green!40!black] (0.5, 1.85) circle [radius = 0.05];
	\draw[green!40!black] (0.5, 1.95) circle [radius = 0.05];
	
	\draw[green!40!black] (0.54, 2) circle [radius = 0.04];
	\draw[green!40!black] (0.63, 2) circle [radius = 0.05];
	\draw[green!40!black] (0.73, 2) circle [radius = 0.05];
	\draw[green!40!black] (0.84, 2) circle [radius = 0.06];
	\draw[green!40!black] (0.95, 2) circle [radius = 0.05];
	
	\draw[green!40!black] (1, 1.54) circle [radius = 0.04];
	\draw[green!40!black] (1, 1.64) circle [radius = 0.06];
	\draw[green!40!black] (1, 1.75) circle [radius = 0.05];
	\draw[green!40!black] (1, 1.85) circle [radius = 0.05];
	\draw[green!40!black] (1, 1.95) circle [radius = 0.05];
	
	
	\draw[green!40!black] (0.5, 2.04) circle [radius = 0.04];
	\draw[green!40!black] (0.5, 2.14) circle [radius = 0.06];
	\draw[green!40!black] (0.5, 2.25) circle [radius = 0.05];
	\draw[green!40!black] (0.5, 2.35) circle [radius = 0.05];
	\draw[green!40!black] (0.5, 2.45) circle [radius = 0.05];
	
	\draw[green!40!black] (0.54, 2.5) circle [radius = 0.04];
	\draw[green!40!black] (0.63, 2.5) circle [radius = 0.05];
	\draw[green!40!black] (0.73, 2.5) circle [radius = 0.05];
	\draw[green!40!black] (0.84, 2.5) circle [radius = 0.06];
	\draw[green!40!black] (0.95, 2.5) circle [radius = 0.05];
	
	\draw[green!40!black] (1, 2.04) circle [radius = 0.04];
	\draw[green!40!black] (1, 2.14) circle [radius = 0.06];
	\draw[green!40!black] (1, 2.25) circle [radius = 0.05];
	\draw[green!40!black] (1, 2.35) circle [radius = 0.05];
	\draw[green!40!black] (1, 2.45) circle [radius = 0.05];
	
    \draw[green!40!black] (1.04, 2) circle [radius = 0.04];
    \draw[green!40!black] (1.13, 2) circle [radius = 0.05];
    \draw[green!40!black] (1.23, 2) circle [radius = 0.05];
    \draw[green!40!black] (1.34, 2) circle [radius = 0.06];
    \draw[green!40!black] (1.45, 2) circle [radius = 0.05];

    \draw[green!40!black] (1.04, 2.5) circle [radius = 0.04];
    \draw[green!40!black] (1.13, 2.5) circle [radius = 0.05];
    \draw[green!40!black] (1.23, 2.5) circle [radius = 0.05];
    \draw[green!40!black] (1.34, 2.5) circle [radius = 0.06];
    \draw[green!40!black] (1.45, 2.5) circle [radius = 0.05];
    
    \draw[green!40!black] (1.5, 2.04) circle [radius = 0.04];
    \draw[green!40!black] (1.5, 2.14) circle [radius = 0.06];
    \draw[green!40!black] (1.5, 2.25) circle [radius = 0.05];
    \draw[green!40!black] (1.5, 2.35) circle [radius = 0.05];
    \draw[green!40!black] (1.5, 2.45) circle [radius = 0.05];
\draw[green!40!black] (1.04, 1.5) circle [radius = 0.04];
\draw[green!40!black] (1.13, 1.5) circle [radius = 0.05];
\draw[green!40!black] (1.23, 1.5) circle [radius = 0.05];
\draw[green!40!black] (1.34, 1.5) circle [radius = 0.06];
\draw[green!40!black] (1.45, 1.5) circle [radius = 0.05];
	    
	    \draw[green!40!black] (1.5, 1.54) circle [radius = 0.04];
	    \draw[green!40!black] (1.5, 1.64) circle [radius = 0.06];
	    \draw[green!40!black] (1.5, 1.75) circle [radius = 0.05];
	    \draw[green!40!black] (1.5, 1.85) circle [radius = 0.05];
	    \draw[green!40!black] (1.5, 1.95) circle [radius = 0.05];

\draw[green!40!black] (1.54, 1.5) circle [radius = 0.04];
\draw[green!40!black] (1.63, 1.5) circle [radius = 0.05];
\draw[green!40!black] (1.73, 1.5) circle [radius = 0.05];
\draw[green!40!black] (1.84, 1.5) circle [radius = 0.06];
\draw[green!40!black] (1.95, 1.5) circle [radius = 0.05];

\draw[green!40!black] (1.54, 2) circle [radius = 0.04];
\draw[green!40!black] (1.63, 2) circle [radius = 0.05];
\draw[green!40!black] (1.73, 2) circle [radius = 0.05];
\draw[green!40!black] (1.84, 2) circle [radius = 0.06];
\draw[green!40!black] (1.95, 2) circle [radius = 0.05];

\draw[green!40!black] (2, 1.54) circle [radius = 0.04];
\draw[green!40!black] (2, 1.64) circle [radius = 0.06];
\draw[green!40!black] (2, 1.75) circle [radius = 0.05];
\draw[green!40!black] (2, 1.85) circle [radius = 0.05];
\draw[green!40!black] (2, 1.95) circle [radius = 0.05];

\draw[green!40!black] (2.04, 1.5) circle [radius = 0.04];
\draw[green!40!black] (2.13, 1.5) circle [radius = 0.05];
\draw[green!40!black] (2.23, 1.5) circle [radius = 0.05];
\draw[green!40!black] (2.34, 1.5) circle [radius = 0.06];
\draw[green!40!black] (2.45, 1.5) circle [radius = 0.05];

\draw[green!40!black] (2.04, 2) circle [radius = 0.04];
\draw[green!40!black] (2.13, 2) circle [radius = 0.05];
\draw[green!40!black] (2.23, 2) circle [radius = 0.05];
\draw[green!40!black] (2.34, 2) circle [radius = 0.06];
\draw[green!40!black] (2.45, 2) circle [radius = 0.05];

\draw[green!40!black] (2.5, 1.54) circle [radius = 0.04];
\draw[green!40!black] (2.5, 1.64) circle [radius = 0.06];
\draw[green!40!black] (2.5, 1.75) circle [radius = 0.05];
\draw[green!40!black] (2.5, 1.85) circle [radius = 0.05];
\draw[green!40!black] (2.5, 1.95) circle [radius = 0.05];
	\draw[green!40!black] (2.04, 1) circle [radius = 0.04];
	\draw[green!40!black] (2.13, 1) circle [radius = 0.05];
	\draw[green!40!black] (2.23, 1) circle [radius = 0.05];
	\draw[green!40!black] (2.34, 1) circle [radius = 0.06];
	\draw[green!40!black] (2.45, 1) circle [radius = 0.05];
	
	\draw[green!40!black] (2, 1.04) circle [radius = 0.04];
	\draw[green!40!black] (2, 1.14) circle [radius = 0.06];
	\draw[green!40!black] (2, 1.25) circle [radius = 0.05];
	\draw[green!40!black] (2, 1.35) circle [radius = 0.05];
	\draw[green!40!black] (2, 1.45) circle [radius = 0.05];

	\draw[green!40!black] (2.5, 1.04) circle [radius = 0.04];
	\draw[green!40!black] (2.5, 1.14) circle [radius = 0.06];
	\draw[green!40!black] (2.5, 1.25) circle [radius = 0.05];
	\draw[green!40!black] (2.5, 1.35) circle [radius = 0.05];
	\draw[green!40!black] (2.5, 1.45) circle [radius = 0.05];
\draw[green!40!black] (2.04, 0.5) circle [radius = 0.04];
\draw[green!40!black] (2.13, 0.5) circle [radius = 0.05];
\draw[green!40!black] (2.23, 0.5) circle [radius = 0.05];
\draw[green!40!black] (2.34, 0.5) circle [radius = 0.06];
\draw[green!40!black] (2.45, 0.5) circle [radius = 0.05];

\draw[green!40!black] (2, 0.54) circle [radius = 0.04];
\draw[green!40!black] (2, 0.64) circle [radius = 0.06];
\draw[green!40!black] (2, 0.75) circle [radius = 0.05];
\draw[green!40!black] (2, 0.85) circle [radius = 0.05];
\draw[green!40!black] (2, 0.95) circle [radius = 0.05];

\draw[green!40!black] (2.5, 0.54) circle [radius = 0.04];
\draw[green!40!black] (2.5, 0.64) circle [radius = 0.06];
\draw[green!40!black] (2.5, 0.75) circle [radius = 0.05];
\draw[green!40!black] (2.5, 0.85) circle [radius = 0.05];
\draw[green!40!black] (2.5, 0.95) circle [radius = 0.05];
    \draw[green!40!black] (2.54, 0.5) circle [radius = 0.04];
    \draw[green!40!black] (2.63, 0.5) circle [radius = 0.05];
    \draw[green!40!black] (2.73, 0.5) circle [radius = 0.05];
    \draw[green!40!black] (2.84, 0.5) circle [radius = 0.06];
    \draw[green!40!black] (2.95, 0.5) circle [radius = 0.05];

    \draw[green!40!black] (2.54, 1) circle [radius = 0.04];
    \draw[green!40!black] (2.63, 1) circle [radius = 0.05];
    \draw[green!40!black] (2.73, 1) circle [radius = 0.05];
    \draw[green!40!black] (2.84, 1) circle [radius = 0.06];
    \draw[green!40!black] (2.95, 1) circle [radius = 0.05];
    
    \draw[green!40!black] (3, 0.54) circle [radius = 0.04];
    \draw[green!40!black] (3, 0.64) circle [radius = 0.06];
    \draw[green!40!black] (3, 0.75) circle [radius = 0.05];
    \draw[green!40!black] (3, 0.85) circle [radius = 0.05];
    \draw[green!40!black] (3, 0.95) circle [radius = 0.05];
\foreach \x in {0.0625, 0.1875, ..., 0.4375}{
	
	\foreach \y in {1.0625, 1.1875, ..., 1.4375}{
		
		\draw [green!40!black] (\x, \y) circle [radius = 0.088];
		\draw [green!40!black] (\x + 3, \y - 0.5) circle [radius = 0.088];
		
		}
	
	}

\foreach \x in {0.25, 0.75,  ..., 3.75}{
	\foreach \y in {0.25, 0.75, ..., 3.75}{
		\fill (\x, \y) circle [radius = 0.02];
		
	}

}

\fill [blue] (0.0625, 1.1875) circle [radius = 0.02];

\fill [blue] (3.4375, 0.5625) circle [radius = 0.02];

\draw [brown, thick] (-0.1, 1.0995) -- (-0.1, 1.6) -- (0.4, 1.6) -- (0.4, 2.6) -- (1.6, 2.6) -- (1.6, 2.1) -- (2.6, 2.1) -- (2.6, 1.1) -- (3.6, 1.1) -- (3.6, 0.4755);

	
	\end{tikzpicture}
	\caption{{\bf An instance of $\S^{\star \star}(k, \delta, v, w)$.} Squares in $\S^{\star}(k, v, w)$ are filled with light blue color. The black dotted points lie in $\mathbb L_k^\star$. $v$ (left) and $w$ (right) are indicated as blue dotted points. The red (lattice) path is the LFPP path between $c([v]_k)$ and $c([w]_k)$. The green circles indicate the balls in $\S^{\star \star}(k, \delta, v, w)$. Balls that lie parallel to the brown segments define a chain of ball connecting $v$ and $w$.}
	\label{fig:LQG_covering}
\end{figure}

We will now describe the construction of $\S(\delta, v, 
w)$. In view of the variance bounds from Lemma~\ref{lem:smoothness2} and \eqref{eq:smoothness} and also the bound $\var(h_{\delta}^{\mathcal U}(v)) = O(\log \delta^{-1}) + O_{\mathcal U, \ep}(1)$ obtained from \eqref{eq:variance} and \eqref{eq:coupling1}, we can use Lemma~\ref{lem:gaussian_tail} to deduce the following. 
There exists a positive, absolute constant $C_2$ (which we choose) such that for all $k$ sufficiently large (depending on $(\mathcal U, \epsilon)$) we have
\begin{equation}\label{eq-def-C-2}
\begin{split}
&\P (\min_{u \in V} h_{2^{-k}}^{\mathcal U}(u) < -2C_2 k\log 2) \leq 2^{-3k} \\
\mbox{ and } \qquad &\P(\max_S\max_B\max_{v \in B^*} (\varphi^{\mathcal U, B^{**}}(v) - h_{2^{-k}}^{\mathcal U}(c(S))) > 2C_2 \sqrt{k \log 2}) \leq 2^{-3k}\,,
\end{split}
\end{equation}
where $S$ ranges over all squares in $\mathbb L_k$ and $B$ ranges over all balls of radius 
$2^{-(k + 2)}$ around $S$ that we described in the last 
paragraph. Choose $\delta'$ as the smallest number of the form $2^{-k}$ (where $k \in \N$) such that $\delta' \geq \delta^{1 - 2C_2\gamma}$. Now define
$$E_{\delta'} = \{\min_{u \in V} h_{\delta'}^{\mathcal U}(u) < -2C_2 \log {(1 / \delta')} \} \cup \{\max_S\max_B\max_{v \in B^*} (\varphi^{\mathcal U, B^{**}}(v) - h_{\delta'}^{\mathcal U}(c(S))) > 2C_2 \sqrt{\log (1/ \delta')}\}\,.$$
On the event $E_{\delta'}$, we simply cover the straight line segment $\overline{vw}$ joining $v$ and $w$ with the minimum possible number of $(M_\gamma^{\mathcal U}, 
\delta)$-balls. 
Otherwise (i.e., on $E^c_{\delta'}$) set $\S(\delta, v, w) = \S^{\star \star}(m', \delta, v, w)$ where $m' = \log_2 
(1 / \delta')$. We note that $\S^{\star \star}(m', \delta, v, w)$ is a valid choice for $\S(\delta, v, w)$ on $E^c_{\delta'}$, since from the definition of LQG measure as a weak limit and \eqref{eq-def-C-2}, we have
\begin{equation}\label{eq-become-clear}
M_{\gamma}^{\mathcal U}(A) \leq \e^{\gamma\max_S\max_B\max_{v \in B^*} (\varphi^{{\mathcal U}, B^{**}}(v) - h_{\delta'}^{\mathcal U}(c(S)))}\e^{\gamma h_{\delta'}^{\mathcal U}(c(S))}M_{\gamma}^{\mathcal U, B_T^{**}}(A) \leq \delta^2\,
\end{equation}
whenever $M_{\gamma}^{\mathcal U, B_T^{**}}(A) \leq \delta^2\e^{-\gamma h_{\delta'}^{\mathcal U}(c(S))}\e^{-2C_2\gamma\sqrt{\log(1 / \delta')}}$ 
(which holds by the definitions of $\S^{\star\star}(m', \delta, v, w)$). 

\smallskip

\noindent {\bf Step 2: bounding $\E (|\S(\delta, v, w)|)$.} Denote the $\sigma$-field generated by $\{h_{\delta'}^{\mathcal U}(v): v \in \mathbb 
L_{m'}^\star\}$ as $\mathfrak F_{\delta'}$ and the event $\{\min_{v \in \mathbb L_{m'}^\star} h_{\delta'}^{\mathcal U}(v) \geq -2C_2\log 
(1 / \delta')\}$ as $F_{\delta'}$. We  then have (note that $F_{\delta'}$ is measurable in $\mathfrak F_{\delta'}$ and $F_{\delta'} \cup E_{\delta'}$ contains all possible outcomes)
\begin{align}
&\E(|\S(\delta, v, w)| \big | \mathfrak F_{\delta'} ) \nonumber\\
 \leq& \sum_{S \in \S^\star(m', v, w)} \sum_T \E \big( N^\star(M_{\gamma}^{\mathcal U, B_T^{**}}, \delta\e^{-\gamma h_{\delta'}^{\mathcal U}(c(S)) / 2}\e^{-C_2\gamma\sqrt{m' \log 2}}, T) \big | \mathfrak F_{\delta'} \big)\mathbf 1_{F_{\delta'}}\nonumber\\
 &+ \E \big(N(M_\gamma^{\mathcal U}, \delta, \overline{vw})\mathbf 1_{E_{\delta'}} | \mathfrak F_{\delta'}\big) + \E \big(N(M_\gamma^{\mathcal U}, \delta, [v]_{m'}) | \mathfrak F_{\delta'}\big) + \E \big(N(M_\gamma^{\mathcal U}, \delta, [w]_{m'})|\mathfrak F_{\delta'} \big)\,, \label{eq-covering-decomposition}
\end{align}
where $T$ ranges over all the $16 \times 4$ segments of 
$\partial S$ and $N^\star(M_{\gamma}^{\mathcal U, B_T^{**}}, r, T)$ is the minimum possible number of $(M_{\gamma}^{\mathcal U, B_T^{**}}, r)$-balls contained in $B_T$ that are required to cover 
$T$. By the Markov field property of GFF (see the discussions around \eqref{eq:decomposition}) and 
Observation~\ref{observ:covering} it follows that $M_{\gamma}^{\mathcal U, B_T^{**}}$ is identically distributed as $M_\gamma^{B_T^{**}}$ and is independent with 
$\mathfrak F_{\delta'}$. The latter is identically distributed as $\tfrac{{\delta'}^2}{16}M_\gamma^{\D}$ by 
scale and translation invariance property of GFF. Also on $F_{\delta'}$, 
$$\delta\e^{-\gamma h_{\delta'}^{\mathcal U}(c(S)) / 2}\e^{-C_2\gamma\sqrt{m' \log 2}} < \delta {\delta'}^{-C_2\gamma} \leq \delta'^{(1 - 2C_2\gamma)^{-1} - C_2\gamma} < \delta'\,.$$
We can then apply Proposition~\ref{prop:second_moment} to the first term in the right hand side of \eqref{eq-covering-decomposition} and get
\begin{eqnarray*}
\E|\S(\delta, v, w)| &\leq& O_{\gamma}((\delta /\delta')^{-1 - C\gamma/2})\e^{C_2\gamma\sqrt{m'\log 2}(1 + C\gamma/2)}\E \Big(\sum_{S \in \S^\star(m', v, w)} \e^{\gamma (1 + C\gamma)h_{\delta'}^{\mathcal U}(c(S)) / 2}\Big) \\
&& + \E N(M_\gamma^{\mathcal U}, \delta, \overline{vw})\mathbf 1_{E_{\delta'}} +\mbox{ }\E N(M_\gamma^{\mathcal U}, \delta, [v]_{m'}) + \E N(M_\gamma^{\mathcal U}, \delta, [w]_{m'})\,.
\end{eqnarray*}
By Lemma~\ref{lem:star_lfpp}, The first term on the right hand side is at most
\begin{eqnarray*}
O_\gamma(\delta^{-2C_2\gamma(1 + C\gamma/2)})
O_{\gamma, \mathcal U, \epsilon}(1){\delta'}^{-1 + \Omega(\gamma^{4/3} / \log \gamma^{-1})} &=& O_{\gamma, \mathcal U, \epsilon}(\delta^{-2C_2\gamma(1 + C\gamma/2)})\delta^{2C_2\gamma}\delta^{-1 + \Omega(\gamma^{4/3} / \log \gamma^{-1})}\\
&=& O_{\gamma, \mathcal U, \epsilon}(\delta^{-1 + \Omega(\gamma^{4/3} / \log \gamma^{-1})})\,.
\end{eqnarray*}
The second term is $O(1)$ as a consequence of \eqref{eq-def-C-2}, Proposition~\ref{cor:second_moment} and the Cauchy-Schwarz inequality (similar to \eqref{eq:display_cs1} and \eqref{eq:display_cs2}). The last two terms are $O_{\gamma, \mathcal U, \epsilon}(\delta^{-O(\gamma)})$ by 
Proposition~\ref{cor:second_moment}. Adding up these four terms, we get the required bound on $\E|\S(\delta, v, w)|$. This concludes the proof of Theorem~\ref{thm:LQG}.
\end{proof}

\section{Adapting to discrete Gaussian free field} 
\label{sec-discrete}
In this section, we briefly explain how the proof of Theorem~\ref{thm:main} can be adapted to prove Theorem~\ref{thm-discrete} with minor modifications.
Let $N = 2^n$, $V_N^{\Gamma} \equiv ([0, \Gamma N - 1] \times [0 , N-1])  \cap \Z^2$ and $V_N^{\Gamma, \epsilon} = ([-\lfloor \epsilon \Gamma N \rfloor, \Gamma N + \lfloor \epsilon \Gamma N \rfloor - 1] \times [-\lfloor \epsilon N\rfloor, N + \lfloor \epsilon N\rfloor + 1]) \cap \Z^2$. Consider a discrete Gaussian free field $\{\eta_{N}(v): v \in V_N^{\Gamma, \epsilon}\}$ on $V_N^{\Gamma, \epsilon}$ with Dirichlet boundary condition. By interpolation we can extend $\eta_{N}$ to a continuous process on the rectangle $[-\epsilon \Gamma N, (1 + \epsilon)\Gamma N] \times [-\epsilon N, (1 + \epsilon)N]$. After appropriate scaling we then get a continuous Gaussian process $\tilde \eta_{N}$ on the domain $V^{\Gamma, \epsilon} = (-\epsilon \Gamma, (1 + \epsilon)\Gamma) \times (-\epsilon, 
(1 + \epsilon))$. It is clear that we need to find a suitable decomposition for the covariance kernel of $\eta_{N}$ in order to get a decomposition of $\tilde \eta_{N}$ similar to the white noise 
decomposition of $\eta_\delta$. The covariance between $\eta_{N}(v)$ and $\eta_{N}(w)$ is given by  Green's function (for the simple random walk)
$G_{V_N^{\Gamma, \epsilon}}(v, w)$. There is a simple representation of $G_{V_N^{\Gamma, \epsilon}}(\cdot, \cdot)$ as 
a sum of simple random walk probabilities. However here we represent it in terms of \emph{lazy} simple random walk probabilities for reasons that would become clear 
shortly. To this end we write
\begin{equation*}
\label{eq:green_decompose1}
G_{V_N^{\Gamma, \epsilon}}(v, w) = \frac{1}{2}\sum_{t = 0}^\infty \P^v(S_t = w, \tau_\epsilon > t)\,,
\end{equation*}
where $\{S_t\}_{t \geq 0}$ is a lazy simple random walk on $\Z^2$ i.e. it stays put for each step with probability $\tfrac{1}{2}$ and jumps to each of its four neighbors with probability $\tfrac{1}{8}$, $\P^v$ is the measure corresponding to the random walk starting from $v$ and $\tau_\epsilon$ is the first time the random walk hits $\partial 
V_N^{\Gamma, \epsilon}$. Emulating our approach to the approximation of circle average process with $\eta_\delta$, we replace $\tau_\epsilon$ in the above representation with the order of its expectation i.e. $N^2$ (on $V_N^\Gamma$, of course) and obtain a new kernel:
\begin{equation*}
\label{eq:green_decompose2}
K_N(v, w) = \frac{1}{2}\sum_{t = 0}^{N^2 - 1} \P^v(S_t = w)\,.
\end{equation*}
Notice that, thanks to the laziness of $(S_t)$, each matrix $(\P^v(S_t = w))_{v, w \in 
V_N^{\Gamma, \epsilon}}$ is non-negative definite. The similarity between this expression and the integral representation of $\cov(\eta_\delta(v), \eta_\delta(w))$ prompts the following decomposition of $K_N(\cdot, \cdot)$:
\begin{equation*}
\label{eq:green_decompose3}
K_{N}(v, w) = \sum_{k \in [n]}\frac{1}{2}\sum_{4^{k-1} \leq t < 4^{k}}\P^v(S_t = w) = \sum_{k \in [n]}K_{N, k}(v, w)\,.
\end{equation*}
Hence we can ``approximate'' $\tilde\eta_{N}$ with a sum of independent, translation invariant processes $\Delta \tilde \eta_{N, k}$ on $V^{\Gamma}$ where the covariance kernel of $\Delta \tilde \eta_{N, k}$ is given by 
$K_{N, k}$ (after appropriate scaling of the arguments). Denote $\tilde \eta_{N, k} = \sum_{j \in 
[k]} \Delta \tilde \eta_{N, j}$. It is immediate that the sequence of processes $\tilde \eta_{N, k}$'s are translation invariant and have independent increments. 
Using standard results on discrete planar random walk and local central limit theorem estimates (see, e.g.,  \cite[Chapters~2 and 4]{Lawler10}) one can also prove the following properties:
\begin{enumerate}[(a)]
\item $\var(\Delta \tilde \eta_{N, k}( \tilde v)) = O(1)$ and $\var(\Delta \tilde \eta_{N, k}(\tilde v) - \Delta \tilde \eta_{N, k}(\tilde w)) = 4^{n - k}O(| \tilde v - \tilde w|^2)$ for all $\tilde v, \tilde w \in V^{\Gamma, \epsilon}$ (compare this to Lemma~\ref{lem:smoothness}).
\item For any straight line segment $L$ of length at most $\Gamma2^{k - n}$, $\var(\int_{L} \Delta \tilde \eta_{N, k}(z)|dz|) = 4^{k - 
n}{O}(\|L\| 2^{n-k})$. Here $\|L\|$ is the length of $L$. Furthermore if $v \in \R^2$ is orthogonal to $L$, then 
$$\var\Big(\int_{L} \Delta \tilde \eta_{N, k}(z)|dz| - \int_{L + \nu} \Delta \tilde \eta_{N, k}(z)|dz|\Big) = 4^{k - n}\Theta(\|L\| 2^{n-k}) = 2^{k - n}\Theta(\|L\|)\,,$$
whenever $\|L\| \geq 2^{k - n}$ and $|\nu| = \Theta(1)$. Compare this to a similar estimate derived in the proof of Lemma~\ref{lem:convex_variance} and also to the Properties~(b) and (c) of the process $\zeta$ discussed in Section~\ref{sec-history}.
\end{enumerate}
We can now use strategies similar to those used for 
constructing $\cross_n$. Since the processes $\tilde \eta_{N, k}$'s do not have rotational invariance, we will actually construct crossings in all possible directions at any given scale (through appropriately scaled rectangles) and consider the \emph{maximum 
expected weight} of these crossings. In view of Properties~(a) and (b), we can then obtain a recursion relation like \eqref{eq:cost_bound_strat2} on 
the maximum expected weight without any significant 
change in the analysis. Next we build a (lattice) crossing $P_n^\star$ of $\tfrac{1}{N}V_N^\Gamma$ from the crossing $P_n$ which we constructed for $V^\Gamma$ so that 
$$\E \big(\sum_{\tilde v \in P_n^\star}\e^{\gamma \tilde \eta_{N, n}(\tilde v)}\big) = O_{\gamma, \epsilon}(N^{1 - \Omega(\gamma^{4/3} / \log \gamma^{-1})})\,.$$ 
We can do this by following the procedure detailed in the proof of Lemma~\ref{lem:star_lfpp}. 
Indeed, one can show by straightforward computation that
\begin{equation}
\label{eq:smoothness2}
\var( \tilde \eta_{N, n}(\tilde v) -  \tilde \eta_{N, n}(\tilde w)) = O(|\tilde v - \tilde w|N)\,
\end{equation}
for all $|\tilde v - \tilde w| \leq 1/N$ (compare 
this to \eqref{eq:smoothness}).
\eqref{eq:smoothness2} is enough for the arguments employed in the proof of Lemma~\ref{lem:star_lfpp} to work smoothly. The approximation of $\tilde \eta_{N}$ with $\tilde \eta_{N, n}$ can be tackled in a similar way as we tackled the approximation of $h_\delta$ with $\eta_{2^{-\lfloor \log_2 \delta \rfloor}}$ in 
Section~\ref{sec:main_thm}. Once we have bounds on expected weights of crossings between shorter boundaries of rectangles at all scales, we can use such crossings to build a light path connecting any two given points in $V_N$ (we discussed this idea in 
Section~\ref{sec:main_thm} in greater detail). This 
leads to a proof of Theorem~\ref{thm-discrete}. 

\end{document}